\documentclass[a4paper,10pt,reqno]{amsart}
          \usepackage{amssymb}
	  \usepackage{amsmath}
          \usepackage{amsfonts}
          \usepackage[english]{babel}
          \usepackage[utf8]{inputenc}

\usepackage[margin=2.5cm]{geometry}

 \usepackage[unicode,colorlinks,plainpages=false,hyperindex=true,bookmarksnumbered=true,bookmarksopen=false,pdfpagelabels]{hyperref}
 \hypersetup{urlcolor=cyan,linkcolor=blue,citecolor=red,colorlinks=true} 
 \usepackage{color} 
 \usepackage{graphicx}
\usepackage{tikz} 
\usetikzlibrary{arrows,shapes}

\makeatletter
\renewcommand*{\p@section}{\,}
\renewcommand*{\p@subsection}{\S\,}
\renewcommand*{\p@subsubsection}{\S\,}
\makeatother

\newtheorem{thm}{Theorem}[section]
\newtheorem{cor}[thm]{Corollary}
\newtheorem{lem}[thm]{Lemma}
\newtheorem{prop}[thm]{Proposition}
\newtheorem{exmp}[thm]{Example}
\newtheorem{rem}[thm]{Remark}

\numberwithin{equation}{section}

\newcommand{\N}{\ensuremath{\mathbb{N}}}

\newcommand{\R}{\ensuremath{\mathbb{R}}}
\newcommand{\Z}{\ensuremath{\mathbb{Z}}}
\newcommand{\kk}{\ensuremath{\Bbbk}}

 \newcommand{\Tr}{\operatorname{Tr}}
\newcommand{\tr}{\operatorname{tr}}

\newcommand{\Hom}{\operatorname{Hom}}
\newcommand{\Der}{\operatorname{Der}}
\newcommand{\Mat}{\operatorname{Mat}}

\newcommand{\Gl}{\operatorname{GL}}
\newcommand{\Rep}{\operatorname{Rep}}

\newcommand{\gl}{\ensuremath{\mathfrak{gl}}}
\newcommand{\g}{\ensuremath{\mathfrak{g}}}
\newcommand{\OO}{\ensuremath{\mathcal{O}}}
\newcommand{\X}{\ensuremath{\mathcal{X}}}
\newcommand{\vep}{\ensuremath{\varepsilon}}
\newcommand\dgal[1]{  \left\{\!\!\left\{#1\right\}\!\!\right\} }
\newcommand\dSN[1]{\left\{\!\!\left\{#1\right\}\!\!\right\}_{\operatorname{SN}}}
\newcommand\br[1]{\{ #1 \}} 
\newcommand\brSN[1]{\{ #1 \}_{\operatorname{SN}}}

\begin{document}

\title{Double quasi-Poisson brackets : fusion and new examples}   

\author{Maxime Fairon}
 \address[Maxime Fairon]{School of Mathematics and Statistics, University of Glasgow, University Place, Glasgow G12 8QQ, UK.}
 \email{Maxime.Fairon@glasgow.ac.uk}


\begin{abstract}
We exhibit new examples of double quasi-Poisson brackets, based on some classification results and the method of fusion.  This method was introduced by Van den Bergh for a large class of double quasi-Poisson brackets which are said differential, and our main result is that it can be extended to arbitrary double quasi-Poisson brackets. We also provide an alternative construction for the double quasi-Poisson brackets of Van den Bergh associated to quivers, and of Massuyeau--Turaev associated to the fundamental groups of surfaces. 
\end{abstract}

\maketitle

 \setcounter{tocdepth}{2}

 
\section{Introduction}  \label{intro} 

We fix a finitely generated associative unital algebra $A$ over a field $\kk$ of characteristic $0$, and we write $\otimes=\otimes_\kk$ for brevity. Following Van den Bergh's initial construction  \cite{VdB1}, we define on $A$ a \emph{double bracket}  $\dgal{-,-}:A\times A \to A \otimes A$ as a $\kk$-bilinear map satisfying for any $a,b,c \in A$
\begin{equation} \label{Eq:cycanti}
 \dgal{a,b}=-\dgal{b,a}^\circ \qquad \text{(cyclic antisymmetry)}, 
\end{equation}
where $(-)^\circ$ denotes the permutation of factors in $A \otimes A$, together with 
\begin{equation}\label{Eq:outder}
 \dgal{a,bc}=\dgal{a,b}c+b\dgal{a,c} \qquad \text{(right derivation rule)}.
\end{equation}
Here, the multiplication refers to the outer $A$-bimodule structure on $A\otimes A$, that is $a \,d\, b=(a d') \otimes (d'' b)$ under Sweedler's notation $d=d'\otimes d''\in A \otimes A$, which we use throughout this text. Assuming that \eqref{Eq:cycanti} holds, one can easily check that \eqref{Eq:outder} is equivalent to 
\begin{equation}\label{Eq:inder}
 \dgal{bc,a}=\dgal{b,a}\ast c+b\ast\dgal{c,a} \qquad \text{(left derivation rule)}, 
\end{equation}
where this time $\ast$ denotes the inner $A$-bimodule structure on $A\otimes A$ given by $a \ast (d'\otimes d'')\ast b=( d' b) \otimes (a d'')$. From these derivation rules, it is easily seen that it suffices to define double brackets on generators of $A$. Associated to such a double bracket, we can define an operation $A^{\times 3}\to A^{\otimes 3}$ by setting
\begin{equation}
\label{Eq:TripBr}
 \dgal{a,b,c}=\dgal{a,\dgal{b,c}'}\otimes \dgal{b,c}''+\tau_{(123)}\dgal{b,\dgal{c,a}'}\otimes \dgal{c,a}''
+\tau_{(123)}^2\dgal{c,\dgal{a,b}'}\otimes \dgal{a,b}'' \,.
\end{equation}
(Here, we define $\tau_{(123)}:\,A^{\otimes 3}\to A^{\otimes 3}$  by 
$\tau_{(123)}(a_1\otimes a_2\otimes a_3)=a_{3}\otimes a_1 \otimes a_{2}$.)
This map is an instance of \emph{triple bracket} : a  $\kk$-trilinear map, which is also a derivation in its last argument for the outer bimodule structure of $A^{\otimes 3}$, and which satisfies a generalisation of the cyclic antisymmetry \eqref{Eq:cycanti} : 
\begin{equation} \label{Eq:TriAnti}
\tau_{(123)}\circ \dgal{-,-,-}\circ \tau_{(123)}^{-1}
=\dgal{-,-,-}\,. 
\end{equation}

An important class of double brackets consists of \emph{double Poisson brackets}. They are such that the associated triple brackets $\dgal{-,-,-}$ identically vanish. Using \eqref{Eq:TripBr}, this condition can be seen as a version of Jacobi identity with value in $A^{\otimes 3}$. These structures have also been introduced by Van den Bergh \cite{VdB1}, and have been a recent subject of study, see e.g. \cite{B,IK,ORS,ORS2,PV,P16,S13,VdW}. 

 Another interesting class of double brackets appears when the unit in $A$ admits a decomposition $1=\sum_{s\in I} e_s$ in terms of a finite set of orthogonal idempotents, i.e. $|I|\in \N^\times$ and $e_s e_t = \delta_{st} e_s$.  In that case, we view $A$ as a $B$-algebra for $B=\oplus_{s\in I} \kk e_s$, and we naturally extend the definition of a double bracket to require $B$-bilinearity, i.e. it vanishes when one of the arguments belongs to $B$. Then, we say that the double bracket is \emph{quasi-Poisson}, or that $(A, \dgal{-,-})$ is a \emph{double quasi-Poisson algebra}, if the associated triple bracket \eqref{Eq:TripBr} satisfies the relation 
\begin{equation}
   \begin{aligned} \label{qPabc}
    \dgal{a,b,c}=&\frac14 \sum_{s\in I} \Big(
c e_s a \otimes e_s b \otimes e_s  - c e_s a \otimes e_s \otimes b e_s - c e_s \otimes a e_s b \otimes e_s 
+ c e_s \otimes a e_s \otimes b e_s \\
&\qquad \quad - e_s a \otimes e_s b \otimes e_s c + e_s a \otimes e_s \otimes b e_s c + e_s \otimes a e_s b \otimes e_s c - e_s \otimes a e_s \otimes b e_s c \Big)\,,
  \end{aligned}
\end{equation}
on any $a,b,c\in A$. Condition \eqref{qPabc} is an expanded form of the original definition \cite[\S 5.1]{VdB1}, and only needs to be checked on generators by the properties of a triple bracket. The main interest of this form is that it is easier to handle in order to classify double quasi-Poisson brackets. Indeed, up to now few cases of double quasi-Poisson brackets are known except associated to quivers \cite{VdB1,VdB2} or fundamental groups of surfaces \cite{MT14}. To have more examples, we provide a complete classification on the free algebra over one generator, and continue the investigation for two generators (with some restrictions). 

The reader could then be tempted to say that such examples do not provide particular insights about double quasi-Poisson brackets in general.  However, an important result of Van den Bergh is that we can perform fusion \cite[\S 5.3]{VdB1} : we can identify idempotents in an algebra with a double quasi-Poisson bracket, and the resulting algebra also admits a double quasi-Poisson bracket. For example, if we respectively denote by $e_1,e_2$ the units of $\kk[t], \kk\langle s_1,s_2 \rangle$ viewed as orthogonal idempotents inside $\kk[t]\oplus \kk\langle s_1,s_2 \rangle$, the fusion algebra obtained by the identification of $e_1$ and $e_2$  is nothing else than $\kk\langle t,s_1,s_2 \rangle$. Hence, knowing a double quasi-Poisson bracket before fusion gives another one on the free algebra over three generators. Therefore, our classification allows to get double quasi-Poisson brackets over any free algebra in general, though not all of them. 
Moving to more exotic examples of double quasi-Poisson algebras, there was a major obstruction to use this fusion process up to now, as we needed the double quasi-Poisson bracket to be differential, see \ref{ss:prelem} for the definition. It was expected by Van den Bergh that this assumption could be removed \cite[\S 5.3]{VdB1}, and the main aim of this paper is to prove this result in its most general form. 
\begin{thm}
\emph{(cf. Theorem \ref{ThmFusBr})} Let $(A,\dgal{-,-})$ be a double quasi-Poisson $B$-algebra, with $B=\oplus_{s\in I} \kk e_s$, $|I|\in \N^\times$, where $e_s e_t=\delta_{st} e_s$ for any $s,t \in I$. Then, if we pick $s,t \in I$ distinct, the algebra $A'$ obtained by identifying the idempotents $e_s,e_t \in A$ has a double quasi-Poisson bracket which coincides with the image of $\dgal{-,-}$  on $\oplus_{s',t'\in I'}e_{s'} A'e_{t'}$, where $I'=I\setminus\{1,2\}$. 
\end{thm}
The advantage of our proof of this theorem is to get an explicit form for the double quasi-Poisson bracket in the algebra $A'$ obtained by identification of the idempotents $e_s,e_t \in A$ : it is given in terms of the double bracket on $A$, together with a second double bracket computed in Lemma \ref{dbrFUS} which was uncovered by Van den Bergh \cite[Theorem 5.3.1]{VdB1}.  Therefore, it becomes easy to see when a double quasi-Poisson bracket has been obtained by fusion. 
In particular, we can show using our classification of double quasi-Poisson bracket on the free algebra on two generators (with some mild restrictions) provided in \ref{classifF2} that any such double bracket is isomorphic to one obtained by fusion, see Theorem \ref{ThmF2fus}. This unexpected result  suggests that knowing double quasi-Poisson brackets on $\kk[t]$ and the path algebra of the (double of the) one-arrow quiver $t:1\to 2$ may be enough to obtain most examples of double quasi-Poisson algebra structures on free algebras.
 
A particular subclass of double quasi-Poisson brackets consists in those that admit a distinguished element. To be precise, given a double quasi-Poisson algebra $(A,\dgal{-,-})$ as above with a complete set of orthogonal idempotents $(e_s)_{s\in I}$, a \emph{multiplicative moment map}  is an invertible element $\Phi=\sum_{s\in I}\Phi_s$ with $\Phi_s\in e_sAe_s$ such that we have  for all $a\in A$ and $s\in I$ 
\begin{equation} \label{Phim}
 \dgal{\Phi_s,a}=\frac12 (ae_s\otimes \Phi_s-e_s \otimes \Phi_s a +  a \Phi_s \otimes e_s-\Phi_s \otimes e_s a)\,.
\end{equation}
We say that the triple $(A,\dgal{-,-},\Phi)$ is a \emph{quasi-Hamiltonian algebra}. As a continuation of the previous result, Van den Bergh showed that we can also obtain a moment map after fusion inside a quasi-Hamiltonian algebra when the double bracket is differential \cite[Theorem 5.3.2]{VdB1}. We also show that this result can be extended to the general case, see Theorem \ref{ThmFusMomap}. 
As a by-product of our method to prove that we keep a double quasi-Poisson bracket or multiplicative moment map after fusion,  we can easily recover the double quasi-Poisson brackets of Van den Bergh \cite{VdB1} and Massuyeau-Turaev \cite{MT14}, see Theorems \ref{ThmVdB} and \ref{ThmPI}.

\medskip

To finish this introduction, let us recall that double brackets have been introduced by Van den Bergh as a non-commutative version of an antisymmetric biderivation following the \emph{non-commutative principle} formulated by  Kontsevich and Rosenberg \cite{K,KR}. More precisely, as explained in \ref{sss:RepSpaces}, any double bracket on an algebra $A$ gives rise to an antisymmetric biderivation on the algebra $\kk[\Rep(A,n)]$ for any $n \geq 1$, i.e. on the coordinate ring of the representation space $\Rep(A,n)$ parametrising $n$-dimensional representations of $A$. In the same way, a double (quasi-)Poisson bracket provides a non-commutative notion of a (quasi-)Poisson bracket under this non-commutative principle. Hence, the present study can be understood as giving new examples of quasi-Poisson brackets on representation spaces. 

\medskip

This article proceeds as follows.  
In Section \ref{fusion}, we recall the necessary constructions needed to understand the fusion procedure, and then prove the main result of this paper which is the fusion of quasi-Hamiltonian algebras in the general case. 
In light of those developments, we give in Section \ref{appli} some examples of double quasi-Poisson brackets obtained by fusion. We also give an alternative (though equivalent) construction of Van den Bergh's quasi-Hamiltonian algebras associated to quivers, and those of Massuyeau-Turaev associated to the fundamental group of compact surfaces with boundary. 
In Section \ref{classif}, we get some first classification results for double quasi-Poisson brackets. 
We finish by explaining in Section \ref{ss:qHscheme} the notion of quasi-Poisson algebra, which is the structure carried by the coordinate ring of  representation spaces of  double quasi-Poisson algebras. There are four appendices that contain some computational proofs. 


{\bf Acknowledgement.} The author is grateful to O. Chalykh for introducing him to the theory of double brackets, and for valuable comments on an earlier draft of this work which greatly improved the presentation of the present paper.  
 The author also thanks A. Alekseev for useful discussions, and the referees for their comments. 
Part of this work was supported by a University of Leeds 110 Anniversary Research Scholarship.


\section{Fusion of quasi-Hamiltonian algebras} \label{fusion}

We consider finitely generated algebras $A,B$ over a field $\kk$ of characteristic zero. We assume that $A$ is a $B$-algebra and, without loss of generality, we identify $B$ with its image in $A$.  Our goal is to prove the main theorems of this paper, which are presented in \ref{ss:main}. To state and prove these results, we need some preliminary constructions associated to double brackets, which were already introduced by Van den Bergh in \cite{VdB1} for most of them.  Since these results easily extend to the case of $n$-brackets (see below for the definition, noting that double brackets are $2$-brackets), we begin by introducing the objects that we will use in full generalities.  

\subsection{Preliminary results} \label{ss:prelem}

We equip the algebra $A^{\otimes n}$ with the \emph{outer} $A$-bimodule structure which is given by 
$b(a_1\otimes \ldots \otimes a_n)c=ba_1\otimes \ldots \otimes a_nc$. For any $s\in S_n$, we introduce  the map $\tau_s:\,A^{\otimes n}\to A^{\otimes n}$ defined by 
$\tau_s(a_1\otimes \ldots \otimes a_n)=a_{s^{-1}(1)}\otimes \ldots \otimes a_{s^{-1}(n)}$.
Following Van den Bergh \cite{VdB1}, we say that  a $B$-linear map 
$\dgal{-,\ldots,-} : A^{\times n}\to A^{\otimes n}$ is a \emph{$n$-bracket} if it is a derivation in its last 
argument for the outer bimodule structure on $A^{\otimes n}$, 
and if it is cyclically anti-symmetric : 
\begin{equation*}
\tau_{(1\ldots n)}\circ \dgal{-,\ldots,-}\circ \tau^{-1}_{(1\ldots n)}
=(-1)^{n+1}\dgal{-,\ldots,-}\,. 
\end{equation*}
By $B$-linearity, we mean that the map $\dgal{-,\ldots,-}$ is $\kk$-linear in each argument and it vanishes on any subset $A^{\times i-1}\times B \times A^{n-i}$, $1\leq i \leq n$. 
Double and triple brackets as defined in the introduction can be equivalently obtained from the above formulation, for which they correspond to the cases $n=2$ and $n=3$.

\subsubsection{Poly-vector fields and $n$-brackets}

Examples of $n$-brackets can easily be obtained by choosing $n$ double derivations, which are elements of $\Der_B(A,A\otimes A)$, with $A\otimes A$ equipped with the outer bimodule structure. 
To state the result, we set  $D_{A/B}:=\Der_B(A,A\otimes A)$ and we see $D_{A/B}$ as an $A$-bimodule by using the inner bimodule structure on $A\otimes A$:  
if $\delta\in D_{A/B}$ and $a,b,c\in A$, then $(b\,\delta\, c) (a)=\delta(a)'\,c\otimes b\,\delta(a)''$. 
We then form the tensor algebra $D_BA:=T_AD_{A/B}$ of this bimodule, which is a graded algebra if we put $A$ in degree $0$ and $D_{A/B}$ in degree $1$. Its elements are called \emph{poly-vector fields}.
\begin{prop}   \emph{(\cite[Proposition 4.1.1]{VdB1})}
\label{Prop:BrQ}

\noindent There is a well-defined linear map $\mu:(D_BA)_n\to \{B$-linear $n$-brackets on $A\}$, 
$Q\mapsto \dgal{-,\ldots,-}_Q$, which on $Q=\delta_1 \ldots \delta_n$ is given by 
\begin{subequations}
  \begin{align}
\dgal{-,\ldots,-}_Q&=\sum_{i=0}^{n-1}(-1)^{(n-1)i}\tau^i_{(1\ldots n)}\circ\dgal{-,\ldots,-}_Q^{\widetilde{ }}
\circ\tau^{-i}_{(1\ldots n)}\,\,, \label{Eq:BrQ1} \\
\dgal{a_1,\ldots,a_n}_Q^{\widetilde{ }}&=\delta_n(a_n)'\delta_1(a_1)''\otimes \delta_1(a_1)'\delta_2(a_2)''\otimes \ldots \otimes 
\delta_{n-1}(a_{n-1})'\delta_n(a_n)''\,. \label{Eq:BrQ2} 
 \end{align}
\end{subequations}
The map $\mu$ factors through $D_BA/[D_BA,D_BA]$, where $[-,-]$ is  the graded commutator.  
\end{prop}
We say that a $n$-bracket is \emph{differential} if it is given by $\mu(Q)$ for some $Q \in (D_BA)_n$. For example, given some  $\delta_1\delta_2\in (D_BA)_2$ we have a differential double bracket by setting 
\begin{equation} \label{Eq:BrQdouble}
\dgal{b,c}_{\delta_1\delta_2}=
 \delta_2(c)'\delta_1(b)'' \otimes \delta_1(b)'\delta_2(c)''
-  \delta_1(c)'\delta_2(b)'' \otimes \delta_2(b)'\delta_1(c)''\,,
\end{equation} 
for any $b,c\in A$. Any differential double bracket is a linear sum of such double brackets. 

By \cite{CB99}, we can write $D_{A/B}=\Hom_{A\otimes A^{op}}(\Omega^1_BA,A \otimes A)$, where $\Omega^1_BA$ is the $A$-bimodule of non-commutative $1$-forms relative to $B$ \cite{CQ95}. The bimodule $\Omega^1_BA$ allows us to give conditions for the map $\mu$ to be an isomorphism.
\begin{prop}  \emph{(\cite[Proposition 4.1.2]{VdB1})}
\label{Prop:BrQSmooth}
Assume that $A$ is left and right flat over $B$, and that   $\Omega^1_BA$ is a projective $A$-bimodule. Then the map $\mu$ from Proposition \ref{Prop:BrQ} is an isomorphism. 
\end{prop}

\begin{exmp} \label{Exmpl1}
 Consider $\kk[x]$, with double bracket given by $\dgal{x,x}=\frac12 (x^2\otimes 1 - 1 \otimes x^2)$ (it is quasi-Poisson by Proposition \ref{Pr:Free1}).  This double bracket is differential : for $d_x \in D_{\kk[x]/\kk}$ given by $d_x(x)=1 \otimes 1$, we have that 
$P=\frac12 x^2d_x d_x\in (D_\kk \kk[x])_2$ defines $\dgal{-,-}$ using Proposition \ref{Prop:BrQ}. 

Fix $k \geq 3$. 
It is not hard to see that $\dgal{x,x^k}\in I_k \otimes \kk[x]\,+\, \kk[x]\otimes I_k$ for $I_k$ the ideal generated by $x^k$, so that the double bracket factors as a map $A_k\times A_k \to A_k \otimes A_k$ with $A_k=\kk[x]/I_k$. We claim that the double bracket is no longer differential on $A_k$. Indeed, any element $P\in D_{A_k/\kk}$ is uniquely defined by the image of the generator $x$, so it can be decomposed as 
\begin{equation*}
  P(x)=c_{0,0} 1 \otimes 1 + c_{1,0}x \otimes 1 + c_{1,1} 1 \otimes x + \sum_{a=2}^{2k-1} \sum_{b=0}^a c_{a,b} \,\, x^b \otimes x^{a-b}\,, \qquad c_{a,b}\in \kk\,,
\end{equation*}
and since we need to satisfy $P(x^k)=0$, we obtain that
\begin{equation*}
  P(x)=c  (x \otimes 1 - 1 \otimes x) + \sum_{a=2}^{2k-1} \sum_{b=0}^a c_{a,b} \,\, x^b \otimes x^{a-b}\,, \qquad c,c_{a,b}\in \kk\,,
\end{equation*}
with possible relations between the coefficients $(c_{a,b})$. If we consider arbitrary $P,Q\in D_{A_k/\kk}$ of that form, we see that the double bracket they define by \eqref{Eq:BrQ2} can be written as 
\begin{equation*}
\begin{aligned}
  \dgal{x,x}_{PQ}=Q(x)'P(x)'' \otimes P(x)' Q(x)''- P(x)'Q(x)'' \otimes Q(x)' P(x)''
=\sum_{a\geq 3} \sum_{b=0}^a d_{a,b} \,\,x^b \otimes x^{a-b}\,,
\end{aligned}
\end{equation*}
for some $d_{a,b}\in \kk$. Thus, any differential double bracket $\dgal{-,-}$ on $A_k$ is such that $\dgal{x,x}\in A_k \otimes A_k$ has homogeneous components of degree $\geq 3$, where we define the degree of $x^a\otimes x^b$ as $a+b$. Hence, the double bracket on $A_k$ given by $\dgal{x,x}=\frac12 (x^2\otimes 1 - 1 \otimes x^2)$ can not be differential.  
\end{exmp}

The algebra $D_BA$ is a noncommutative version of the algebra of polyvector fields on a manifold :  $D_BA$ admits a canonical \emph{double Schouten--Nijenhuis bracket}, which makes $D_BA$ into a double Gerstenhaber algebra \cite[\S 2.7,3.2]{VdB1}. We write this (graded) double bracket $(D_BA)^{\times 2}\to (D_BA)^{\otimes 2}$ as $\dSN{-,-}$. 
We denote by $\brSN{-,-}$ the associated bracket  $\brSN{-,-}:=m\circ\dSN{-,-}$, where $m$ is the multiplication on the algebra $D_BA$. 
We note that the following results hold. 
\begin{prop} \emph{(\cite[\S 4.2]{VdB1})}
 \label{TripleDiff} 
Assume that $\dgal{-,-}$ is a double bracket defined by the bivector $P \in (D_B A)_2$. Then the associated triple bracket given by \eqref{Eq:TripBr} is defined by the trivector $\frac12 \brSN{P,P}\in (D_B A)_3$.   
\end{prop}

\begin{prop} \emph{(\cite[\S 3.4]{VdB1})} 
 \label{PropBeB} 
Assume $e\in B$ is an idempotent such that $BeB=B$. Then $e (D_B A)e=D_{eBe}eAe$, and the (graded) double bracket $\dSN{-,-}$ on $D_B A$ restricted to $e (D_B A)e$ coincides with the double Schouten-Nijenhuis bracket on $D_{eBe}eAe$.  
\end{prop}

\subsubsection{Induced brackets and fusion algebras} \label{subFusAlg}

We now state several ways to get new $n$-brackets from old ones. Most of these results are  straightforward extensions of propositions given in \cite[\S 2.5]{VdB1}, which were originally stated in the case $n=2$. 

Given an algebra $A$ over $B$ and a non-empty subset $S\subset A$, we can consider the universal localisation $A_S$ as an algebra over $B$. The morphism $f: A \to A_S$ induces a unique map of double derivations $f_\ast: D_{A/B} \to D_{A_S/B}$ which satisfies $f_\ast(\delta)(s^{-1})=s^{-1}f(\delta(s)')\otimes f(\delta(s)'')s^{-1}$ for any $\delta \in D_BA$ and $s\in S$. This map can be extended to $f_\ast:D_BA \to D_B A_S$. 
\begin{prop}  \label{Pr:Loc}
Consider a non-empty subset $S \subset A$. Then a $B$-linear $n$-bracket $\dgal{-,\ldots,-}$ on $A$ induces a unique $B$-linear $n$-bracket on $A_S$. 
If $\dgal{-,\ldots,-}$ is differential for $Q \in (D_BA)_n$, 
then the induced $B$-linear $n$-bracket is differential for $f_\ast(Q)\in (D_{B} A_S)_n$.
\end{prop}
\begin{proof}
Note that a $n$-bracket on $A_S$ needs to satisfy 
\begin{equation*}
  \dgal{a_1, \ldots,a_{n-1},s^{-1}}=-s^{-1}\dgal{a_1,\ldots,a_{n-1},s}s^{-1}\,,
\end{equation*}
for any $a_1,\ldots,a_{n-1}\in A_S$ and $s\in S$ due to the derivation property.  Using the cyclic antisymmetry and the derivation property, we can then always rewrite  $\dgal{a_1, \ldots,a_{n}}$ with $a_1,\ldots,a_n \in A_S$ in terms of sums and products in $A_S$ containing only the $n$-bracket evaluated on elements of $A$. 
\end{proof}

We use this result without further mention throughout the text. Next, if $e\in B$ is an idempotent, we get a canonical map $\pi^e:A \to eAe$, $a \mapsto eae$, which extends to double derivations as $\pi^e_\ast:D_{A/B} \to D_{eAe/eBe}$, $\delta \mapsto e \delta e$. In the case where $B=BeB$, we get a non-unique decomposition $1=\sum_i p_i e q_i$, and it yields a trace map $\Tr:A \to eAe$ given by $\Tr(a)=\sum_i e q_i a p_i e$. 
It also gives a map $\Tr:D_{A/B}\to D_{eAe/eBe}$ by setting $\Tr(\delta)=\sum_i e q_i \delta p_i e$, 
which can be written as $\Tr(\delta)(eae)=e\delta'(a) p_ie \otimes e q_i \delta''(a)e$ for any $a\in A$. To extend this to polyvector fields, note that $\Tr:D_BA\to eD_BAe: Q \mapsto\sum_i e q_i Q p_i e$ defines a map $D_BA\to D_{eBe}eAe$ by Proposition \ref{PropBeB}.

\begin{prop}  \label{Pr:Induce2}
Assume that $e\in B$ is an idempotent. Then a $B$-linear $n$-bracket $\dgal{-,\ldots,-}$ on $A$ induces a unique $eBe$-linear $n$-bracket on $eAe$. 
If $B=BeB$ and $\dgal{-,\ldots,-}$ is differential for $Q \in (D_BA)_n$, 
then the induced $eBe$-linear $n$-bracket is differential for $\Tr(Q)\in (D_{eBe} eAe)_n$.
\end{prop}
\begin{proof}
  Fix $a_1,\ldots,a_n\in A$. Denoting $\dgal{a_1,\ldots,a_n}$ as $b_1 \otimes \ldots \otimes b_n \in A^{\otimes n}$ (up to linear combinations), we get the unique induced $n$-bracket 
\begin{equation} \label{Eq:idempt}
  \dgal{ea_1e, \ldots, e a_n e}=(e \otimes \ldots \otimes e)\dgal{a_1,\ldots,a_n}(e \otimes \ldots \otimes e) =e b_1 e \otimes \ldots  \otimes e b_n e \in (eAe)^{\otimes n}. 
\end{equation}
If the $n$-bracket is differential for $Q=\delta_1,\ldots, \delta_n \in (D_BA)_n$, we get from \eqref{Eq:idempt} and Proposition \ref{Prop:BrQ} that  
\begin{equation*}
  \dgal{ea_1e, \ldots, e a_n e}=\sum_{i=0}^{n-1}(-1)^{(n-1)i}(e \otimes \ldots \otimes e) \tau^i_{(1\ldots n)}
\dgal{-,\ldots,-}_Q^{\widetilde{ }} \tau^{-i}_{(1\ldots n)}(a_1, \ldots, a_n)
(e \otimes \ldots \otimes e) \,,
\end{equation*}
with $\dgal{-,\ldots,-}_Q^{\widetilde{ }}$ given by \eqref{Eq:BrQ2}. Assuming that $1=\sum_i p_i e q_i$, we can write for $i=0$
\begin{equation*}
\begin{aligned}
&  (e \otimes \ldots \otimes e)\dgal{a_1,\ldots,a_n}_Q^{\widetilde{ }}(e \otimes \ldots \otimes e) \\
=& e\delta_n(a_n)'1\delta_1(a_1)''e\otimes e\delta_1(a_1)'1\delta_2(a_2)''e\otimes \ldots \otimes 
 e\delta_{n-1}(a_{n-1})'1\delta_n(a_n)''e \\
=& \sum_{i_1} \ldots \sum_{i_n}\,\delta_n(ea_ne)' p_{i_1} e q_{i_1}\delta_1(ea_1e)''\otimes 
\ldots \otimes 
 \delta_{n-1}(ea_{n-1}e)' p_{i_n} e q_{i_n}\delta_n(ea_ne)'' \\
=&\sum_{i_1} \ldots \sum_{i_n} \dgal{ea_1e,\ldots,ea_ne}_{e q_{i_1}\delta_1 p_{i_2}eq_{i_2}\delta_2 p_{i_3}e \ldots e q_{i_n} \delta_n  p_{i_1} e}^{\widetilde{ }} \\
=&\sum_{i_1} \dgal{ea_1e,\ldots,ea_ne}_{e q_{i_1}\delta_1 \delta_2  \ldots\delta_n p_{i_1} e}^{\widetilde{ }}
=\dgal{ea_1e,\ldots,ea_ne}_{\Tr(\delta_1 \delta_2  \ldots\delta_n)}^{\widetilde{ }}\,.
\end{aligned}
\end{equation*}
The argument is similar for $i=1,\ldots,n-1$ so that 
\begin{equation*}
  \dgal{ea_1e, \ldots, e a_n e}=\sum_{i=0}^{n-1}(-1)^{(n-1)i} \tau^i_{(1\ldots n)}
\dgal{-,\ldots, -}_{\Tr(\delta_1 \delta_2  \ldots\delta_n)}^{\widetilde{ }} \tau^{-i}_{(1\ldots n)}
(ea_1e, \ldots, ea_ne) \,,
\end{equation*}
which is differential for $\Tr(\delta_1 \delta_2  \ldots\delta_n)$ by definition. 
\end{proof}

Next, consider algebras $A$ and $A'$ respectively over $B$ and $B'$. We get that $A \oplus A'$ is a $(B \oplus B')$-algebra, and we can identify $D_{A\oplus A' / B \oplus B'}$ with $D_{A/B}\oplus D_{A'/B'}$. This extends to the identification of $D_{B \oplus B'}A \oplus A'$ and $D_BA\oplus D_{B'}A'$.  
\begin{prop}  \label{Pr:Oplus}
Assume that $\dgal{-,\ldots,-}$ is a $B$-linear $n$-bracket on $A$, and $\dgal{-,\ldots,-}'$ is a $B'$-linear $n$-bracket on $A'$. Then there exists a unique $(B \oplus B')$-linear $n$-bracket $\dgal{-,\ldots,-}^\oplus$ on  $A \oplus A'$ extending the $n$-brackets $\dgal{-,\ldots,-}$ and $\dgal{-,\ldots,-}'$, while it is such that  $\dgal{c_1,\ldots,c_n}^\oplus=0$ whenever there exists $i\neq j$ with $c_i=(a,0)$, $c_j=(0,b)$. 
Furthermore, if the $n$-brackets on $A$ and $A'$ are differential for $Q \in (D_BA)_n$ and $Q'\in (D_{B'}A')_n$, 
then $\dgal{-,\ldots,-}^\oplus$ is differential for $(Q,Q')\in (D_{B\oplus B'} A \oplus A')_n$. 
\end{prop}
\begin{proof}
  It follows directly by linearity since 
\begin{equation*}
\begin{aligned}
  \dgal{(a_1,b_1), \ldots, (a_n,b_n)}^\oplus=&\dgal{(a_1,0), \ldots, (a_n,0)}^\oplus+\dgal{(0,b_1), \ldots, (0,b_n)}^\oplus \\
=&(\dgal{a_1,\ldots,a_n},0)+(0,\dgal{b_1,\ldots,b_n}')\,,
\end{aligned}
\end{equation*}
for any $a_1,\ldots,a_n \in A$, $b_1,\ldots,b_n \in A'$.
\end{proof}

Given algebras $A,A'$ over $B$ with algebra monomorphisms $j: B \to A$ and $j': B \to A'$, recall that the free algebra
$A \ast_B A'$ is given by $T_\kk (A \oplus A')/J$, where $J$ is the two-sided ideal generated by the relations $a_1\otimes a_2=a_1 a_2$, $a_1'\otimes a_2'=a_1' a_2'$, $j(b)=j'(b)$ for all $a_1,a_2 \in A$, $a_1', a_2'\in A'$ and $b\in B$. Set $\bar{A}=A \ast_B A'$. The canonical maps $i: A \to \bar{A}, i': A' \to \bar{A}$ yield maps of double derivations 
$i_\ast:D_{A/B}\to D_{\bar{A}/A'}$ and $i'_\ast:D_{A'/B}\to D_{\bar{A}/A}$, which can both be seen to take value  in $D_{\bar{A}/B}$. In particular, they extend to polyvector fields. 
\begin{prop} 
Assume that $\dgal{-,\ldots,-}$ and $\dgal{-,\ldots,-}'$ are $B$-linear $n$-brackets on $A$ and $A'$ respectively. Then there exists a unique $n$-bracket $\dgal{-,\ldots,-}^*$ on  $\bar{A}=A \ast_B A'$ extending the $n$-brackets $\dgal{-,\ldots,-}$ and $\dgal{-,\ldots,-}'$, while it is such that  $\dgal{a_1,\ldots,a_n}^\ast=0$ whenever there exists $i\neq j$ with $a_i\in A$, $a_j\in A'$. 
Furthermore, if the $n$-brackets on $A$ and $A'$ are differential for $Q \in (D_BA)_n$ and $Q'\in (D_BA')_n$, 
then $\dgal{-,\ldots,-}^*$ is differential for $i_\ast(Q)+i_\ast'(Q')\in (D_B \bar{A})_n$. 
\end{prop}
Endowing $A'$ with the zero $n$-bracket, we get the next result. 
\begin{cor} \label{Cor:Induce}
  Assume that $\dgal{-,\ldots,-}$ is a $B$-linear $n$-bracket on $A$. Then there is a unique $A'$-linear $n$-bracket on  $\bar{A}=A \ast_B A'$ extending it. If $\dgal{-,\ldots,-}$ is differential for $Q \in (D_BA)_n$, 
then the induced $A'$-linear $n$-bracket is differential for $i_\ast(Q)\in (D_{A'} \bar{A})_n$. 
\end{cor}
\noindent In particular, $n$-brackets are compatible with base changes.

We now use these results, and assume that there exist orthogonal idempotents $e_1,e_2\in B$. The \emph{extension algebra} $\bar{A}$ of $A$ along the pair $(e_1,e_2)$ is given by 
\begin{equation}
  \bar{A}=A \ast_{\kk e_1 \oplus \kk e_2 \oplus \kk \mu} (\Mat_2(\kk)\oplus \kk \mu) = A \ast_B \bar{B}\,,
\end{equation}
 where $\mu=1-e_1-e_2$, and $\Mat_2(\kk)$ is seen as the $\kk$-algebra generated by $e_1=e_{11},e_{12},e_{21},e_2=e_{22}$ with $e_{st}e_{uv}=\delta_{tu}e_{sv}$. 
The \emph{fusion algebra} $A^f$ of $A$ along $(e_1,e_2)$ is the algebra obtained from $\bar{A}$ by discarding elements of $e_2 \bar{A} + \bar{A} e_2$, i.e. 
\begin{equation}
  A^f=\, \epsilon \bar{A} \epsilon\,, \qquad \text{for } \epsilon=1-e_2\,.
\end{equation}
We also say that $A^f$ is the fusion algebra obtained by fusing $e_2$ onto $e_1$. Note that $A^f$ is a $B^f$-algebra for  $B^f=\epsilon \bar{B} \epsilon$. 
The elements of $A^f$ can be characterised in terms of generators as follows. (This choice of generators was considered by Van den Bergh \cite[Proof of Lemma 5.3.3]{VdB1}.)  
\begin{lem} \label{AfGenerat}
  Elements of $A^f$ can be written in terms of generators of the following forms 
\begin{subequations}
  \begin{align}
   (\text{first type})&\qquad\qquad a=t\,, \qquad t \in \epsilon A \epsilon\,, \label{type1}\\
   (\text{second type})&\qquad\qquad a=e_{12}u\,, \qquad u \in e_2 A \epsilon\,,\label{type2} \\
   (\text{third type})&\qquad\qquad a=v e_{21}\,, \qquad v \in \epsilon A e_2\,, \label{type3} \\
   (\text{fourth type})&\qquad\qquad a=e_{12} w e_{21}\,, \qquad w \in e_2 A e_2\,. \label{type4}
  \end{align}
\end{subequations}
\end{lem}
Remark that $\bar{B}$ satisfies $\bar{B}=\bar{B} \epsilon \bar{B}$ since $1=1 \epsilon 1 + e_{21} \epsilon e_{12}$. 
Using the map  $\Tr:D_{\bar{B}}\bar{A}\to D_{B^f}A^f$ given by 
$\Tr(\bar{Q})=\epsilon \bar{Q} \epsilon +  \epsilon e_{12} \bar{Q} e_{21} \epsilon$ together with $i_\ast:D_BA\to D_{\bar{B}}\bar{A}$, we get a map $\Tr \circ i_\ast : D_B A \to D_{B^f}A^f$.
We combine Corollary \ref{Cor:Induce} and Proposition \ref{Pr:Induce2} to get the following generalisation of 
\cite[Corollary 2.5.6]{VdB1}. 
\begin{prop} 
\label{PropIndbr}
If $A$ is a $B$-algebra with $n$-bracket $\dgal{-,\ldots,-}$, it induces $n$-brackets on $\bar{A}$ over $\bar{B}$ and $A^f$ over $B^f$.   If the $n$-bracket on $A$ is differential for $Q \in (D_BA)_n$, then the induced $n$-brackets are differential for $i_\ast (Q) \in (D_{\bar B}\bar A)_n$ and $\Tr\circ i_\ast (Q)\in (D_{B^f} A^f)_n$ respectively. 
\end{prop}
From now on, we denote the compositions $\Tr\circ i$ and  $\Tr\circ i_\ast$ simply as $\Tr$. 

\subsubsection{Double quasi-Poisson brackets from the gauge elements} \label{ss:dqP}

Assume that $B=\kk e_1 \oplus \ldots \oplus \kk e_N$, where the $(e_s)$ form   a complete set of  orthogonal idempotents.   We define for all $s=1,\ldots,N$ a double derivation $E_s\in D_{A/B}$ such that for any $a\in A$, 
$ E_s(a)=ae_s\otimes e_s - e_s\otimes e_s a$. These are called the \emph{gauge elements}. Following \cite[\S 5.1]{VdB1}, we say that a double bracket $\dgal{-,-}$ on $A$ over $B$ is quasi-Poisson if it satisfies 
\begin{equation} \label{qPabcBis}
  \dgal{-,-,-}=\frac{1}{12}\sum_{s=1}^N \dgal{-,-,-}_{E_s^3}\,,
\end{equation}
where on the left-hand side we have the associated triple bracket given by \eqref{Eq:TripBr}, while 
 the triple brackets in the right-hand side are defined from Proposition \ref{Prop:BrQ} with $E_s^3\in (D_BA)_3$. It is then an easy exercise to check that \eqref{qPabcBis} evaluated on $a,b,c \in A$ gives \eqref{qPabc}, so that this definition coincides with the one given in the introduction. 
 Note that under the assumption of Proposition \ref{Prop:BrQSmooth} the double quasi-Poisson bracket $\dgal{-,-}$ is differential for some $Q \in (D_BA)_2$, and we get the equivalent condition that $\brSN{Q,Q}=\frac16 \sum_{s=1}^N E_s^3$ modulo $[D_B A, D_BA]$  by Propositions \ref{Prop:BrQ} and \ref{TripleDiff}. 
 
In a double quasi-Poisson algebra $(A,\dgal{-,-})$, we say that an element $\Phi\in A^\times$ is a moment map if $\Phi_s=e_s \Phi e_s$ satisfies 
$\dgal{\Phi_s,-}=\frac{1}{2}(\Phi_sE_s+E_s\Phi_s)$ for all $s=1,\ldots,N$. It is an easy exercise to check that the $s$-th condition is equivalent to \eqref{Phim}, hence this definition of moment map is equivalent to the one given in the introduction. 

 \begin{rem} \label{RemOplus}
Assume that $B=\kk e_1 \oplus \ldots \oplus \kk e_N$, $B'=\kk e_1' \oplus \ldots \oplus \kk e_M'$, and we have double quasi-Poisson brackets $\dgal{-,-}$ and $\dgal{-,-}'$ over $A$ and $A'$ respectively. Then $\dgal{-,-}^\oplus$ is a $(B\oplus B')$-linear double quasi-Poisson bracket over $A \oplus A'$. This can be obtained by combining Proposition \ref{Pr:Oplus} and the definition of double quasi-Poisson bracket using the gauge elements given above. Moreover, if $\Phi$ and $\Phi'$ are the corresponding moment maps, then $(\Phi,\Phi')$ turns $A\oplus A'$ into a quasi-Hamiltonian algebra. 
\end{rem}

\subsection{Main theorems} \label{ss:main}

Hereafter, we assume that $A$ is a $B$-algebra for $B=\kk e_1 \oplus \ldots \oplus \kk e_N$ a semisimple $\kk$-algebra. Our aim is to prove the following results. 
\begin{thm} \label{ThmFusBr}
  Assume that $(A, \dgal{-,-})$ is a double quasi-Poisson algebra over $B$. Consider the fusion algebra $A^f$ obtained by fusing $e_2$ onto $e_1$.  Then, $A^f$ has a $B^f$-linear double quasi-Poisson bracket given by 
\begin{equation} \label{dgalf}
  \dgal{-,-}^f:= \dgal{-,-} + \dgal{-,-}_{fus}\,, 
\end{equation}
where  the first double bracket on the right-hand side is induced in $A^f$ by the one of $A$ using Proposition \ref{PropIndbr}, and the second double bracket $\dgal{-,-}_{fus}$ is defined by $-\frac12 \Tr(E_1)\Tr(E_2)\in (D_{B^f}A^f)_2$ using Proposition \ref{Prop:BrQ}.  
\end{thm}

\begin{thm} \label{ThmFusMomap}
  Assume that $(A, \dgal{-,-},\Phi)$ is a quasi-Hamiltonian algebra over $B$, where $\Phi=\sum_s \Phi_s \in \oplus_s e_s A e_s$. 
  Consider the fusion algebra $A^f$ obtained by fusing $e_2$ onto $e_1$. 
  Then  $A^f$ is a quasi-Hamiltonian algebra for the double quasi-Poisson bracket $\dgal{-,-}^f$ given in Theorem \ref{ThmFusBr} and for the multiplicative moment map 
\begin{equation}
  \Phi^f=e_1 \Tr(\Phi_1)\Tr(\Phi_2) e_1 + \sum_{s\neq 1,2} e_s \Tr(\Phi_s) e_s\,.
\end{equation}
\end{thm}

\begin{rem}
In the case where the double quasi-Poisson bracket $\dgal{-,-}$ is differential for some $Q\in (D_BA)_2$, we have that the double quasi-Poisson bracket \eqref{dgalf} is differential for $Q^f:=\Tr(Q)-\frac12 \Tr(E_1)\Tr(E_2)$ by Proposition \ref{PropIndbr} and linearity of the map $\mu$ in Proposition \ref{Prop:BrQ}. Therefore, Theorems \ref{ThmFusBr} and \ref{ThmFusMomap} are nothing else than \cite[Theorems 5.3.1,5.3.2]{VdB1} in such a case. However, if the double quasi-Poisson bracket is \emph{not} differential (which can only happen if $A$ does not satisfy the assumptions from Proposition \ref{Prop:BrQSmooth}), these results extend their analogues proved in the differential case, as expected by Van den Bergh \cite[\S 5.3]{VdB1}. 
\end{rem}

\subsection{Preparation for the proofs} \label{ss:prep} 

\subsubsection{Image of the gauge elements} 

We have well-defined double derivations $E_s\in D_{A/B}$, $1 \leq s \leq N$, and we want to know what are their images in the fusion algebra $A^f$, obtained by fusing the idempotent $e_2$ onto $e_1$ as in \ref{subFusAlg}.  To avoid any conflicting notations, write $E_1,E_2,\ldots,E_N$ for the gauge elements over $A$ and their image under $D_{A/B}\to D_{\bar{A}/\bar{B}}$, and let 
$F_1,F_3,\ldots,F_N$ be the gauge elements in $D_{A^f/B^f}$, with $B^f=\kk e_1 \oplus \kk e_3 \oplus \ldots \oplus \kk e_N$. 
We now relate the double derivations $\Tr E_s$ and $F_s$. (These results first appeared in \cite[\S 5.3]{VdB1}, but we give a proof for the sake of clarity.)

\begin{lem} 
For any $s \neq 1,2$, $\Tr(E_s)=F_s$. \label{LemFs}
\end{lem}
\begin{proof}
We only need to prove the equality on generators of $A^f$. 
By Lemma \ref{AfGenerat}, we can write any $a \in A^f$ as $a=e_{+} \alpha e_{-}$, for $a\in A$ and some $e_+\in \{e_{12},\epsilon\}$, $e_-\in \{e_{21},\epsilon\}$. Hence, by definition of gauge element  and the trace map 
\begin{equation*}
\begin{aligned}
  &\Tr(E_s)(a)=\epsilon \ast E_s(a) \ast \epsilon + \epsilon e_{12} \ast E_s(a) \ast e_{21} \epsilon 
= \,\epsilon \ast e_+ E_s(\alpha) e_- \ast \epsilon + \epsilon e_{12} \ast e_+ E_s(\alpha) e_- \ast e_{21} \epsilon \\
=&(e_+ \alpha e_s \epsilon \otimes \epsilon e_s e_- \,-\, e_+ e_s \epsilon  \otimes \epsilon e_s \alpha e_-)
+(e_+ \alpha e_s e_{21} \epsilon \otimes \epsilon e_{12} e_s e_- - e_+ e_s e_{21} \epsilon \otimes \epsilon e_{12} e_s \alpha e_-) \\
=& (e_+ \alpha e_s  \otimes  e_s e_- \,-\, e_+ e_s   \otimes  e_s \alpha e_-)
\end{aligned}
\end{equation*}
since $e_s \epsilon = e_s = \epsilon e_s$ and $e_s e_{21}=0=e_{12} e_s$ as $s\neq 2$. Now, remark that we can write this as 
\begin{equation*}
    \Tr(E_s)(a)= (e_+ \alpha e_- e_s)  \otimes  e_s  \,-\, e_s   \otimes  e_s (e_+ \alpha e_-)\,.
\end{equation*}
Indeed, for the first term, either $e_-=\epsilon$ and $e_s e_-=e_s=e_- e_s$, or $e_-\neq \epsilon$ and $e_s e_-=0=e_- e_s$. The same applies to the second term. 
\end{proof}

\begin{lem} \label{LemF12}
The double derivations $\Tr(E_1),\Tr(E_2)$ take the following forms on generators  :

\noindent if  $a=t$ for $t \in \epsilon A \epsilon$, 
\begin{equation} \label{TrEepsilon}
  \Tr(E_1)(t)=t e_1 \otimes e_1- e_1 \otimes  e_1 t, \quad \Tr(E_2)(t)=0\,,
\end{equation}
if  $a=e_{12}u$ for $u \in e_2 A \epsilon$,
\begin{equation} \label{TrEe12}
  \Tr(E_1)(e_{12}u)=(e_{12}u) e_1 \otimes e_1, \quad \Tr(E_2)(e_{12}u)=- e_1 \otimes (e_{12}u)\,,
\end{equation}
   if $a=v e_{21}$ for $v \in \epsilon A e_2$,
\begin{equation} \label{TrEe21}
  \Tr(E_1)(v e_{21})=- e_1 \otimes e_1(v e_{21}), \quad \Tr(E_2)(v e_{21})=(ve_{21}) \otimes e_1\,,
\end{equation}
if $a=e_{12} w e_{21}$ for $w \in e_2 A e_2$, 
\begin{equation} \label{TrEw}
  \Tr(E_1)(e_{12} w e_{21})=0, \quad \Tr(E_2)(e_{12} w e_{21})=(e_{12} w e_{21})e_1 \otimes e_1 - e_1 \otimes e_1 (e_{12} w e_{21})\,.
\end{equation}
In particular, $\Tr(E_1)+\Tr(E_2)=F_1$. 
\end{lem}
\begin{proof}
  First, remark that $\Tr(E_1)= \epsilon  E_1 \epsilon$ and $\Tr(E_2)= \epsilon e_{12} E_2 e_{21} \epsilon$, by expansion as in Lemma \ref{LemFs} or using that in $D_B A$ we have $E_s \in e_s D_B A e_s$. Therefore, writing a generator $a\in A^f$ as $a=e_+ \alpha e_-$ as in Lemma \ref{LemFs}, 
\begin{equation*}
\begin{aligned}
  \Tr(E_1)(a)=& e_+ \alpha e_1 \otimes e_1 e_- - e_+ e_1 \otimes e_1 \alpha e_-\,, \\
\Tr(E_2)(a)=& e_+ \alpha e_{21}  \otimes  e_{12}  e_-   - e_+  e_{21}  \otimes  e_{12} \alpha e_-\,,
\end{aligned}
\end{equation*}
using the relations between idempotents. In the first case \eqref{type1}, $\alpha=t$, $e_+=e_-=\epsilon$ so that the identities are clear. In the second case \eqref{type2} with $\alpha=u$, $e_+=e_{12}$ and $e_-=\epsilon$ so that 
\begin{equation*}
  \Tr(E_1)(a)= e_{12} u e_1 \otimes e_1 - e_{12} e_1 \otimes e_1 u  \,, \quad 
\Tr(E_2)(a)= e_{12} u e_{21}  \otimes  e_{1}  - e_{1}  \otimes  e_{12} u \,,
\end{equation*}
and we get our claim by remarking that $e_{12}e_1=0$ and $u e_{21}= u \epsilon e_{21}=0$. 
In the third case \eqref{type3} we take  $\alpha=v$, $e_+=\epsilon$ and $e_-=e_{21}$, which yields 
\begin{equation*}
  \Tr(E_1)(a)= v e_1 \otimes e_1 e_{21} -  e_1 \otimes e_1 v e_{21}\,, \quad
\Tr(E_2)(a)=  v e_{21}  \otimes  e_{1}   - \epsilon  e_{21}  \otimes  e_{12} v e_{21}\,.
\end{equation*}
Hence, it suffices to remark that $e_1 e_{21}=0$ and $e_{12} v = e_{12} \epsilon v=0$. 
Finally for \eqref{type4}, we take $\alpha=w$ and $e_+=e_{12}$, $e_-=e_{21}$ to get 
\begin{equation*}
  \Tr(E_1)(a)= e_{12} w e_1 \otimes e_1 e_{21} -  e_{12} e_1 \otimes e_1 w e_{21}\,, \quad
\Tr(E_2)(a)=  e_{12}w e_{21}  \otimes  e_{12}e_{21}    - \epsilon  e_{12}e_{21}  \otimes  e_{12} w e_{21}\,,
\end{equation*}
so that our claim follows since $e_1 e_{21} =0=  e_{12} e_1$. 
\end{proof}

\subsubsection{Properties of the double bracket $\dgal{-,-}_{fus}$}
Recall that the double bracket $\dgal{-,-}_{fus}$ is defined by $-\frac12 \Tr(E_1)\Tr(E_2)\in (D_{B^f}A^f)_2$ using Proposition \ref{Prop:BrQ}. 
\begin{lem} \label{dbrFUS}
  On generators of $A^f$, the double bracket $\dgal{-,-}_{fus}$ can be written as 
\begin{subequations}
  \begin{align}
&\dgal{\epsilon t \epsilon , \epsilon \tilde{t} \epsilon}_{fus}=0\,, \label{tt}\\
&\dgal{\epsilon t \epsilon , e_{12} u \epsilon}_{fus}=\frac12 \left( e_1 \otimes t e_{12}u - e_1 t  \otimes e_{12}u\right)\,, \label{tu}\\
&\dgal{\epsilon t \epsilon ,\epsilon v e_{21} }_{fus}=\frac12 \left( v e_{21} t   \otimes e_1 - v e_{21} \otimes t e_1\right)\,, \label{tv}\\
&\dgal{\epsilon t \epsilon , e_{12} w e_{21}}_{fus}=\frac12 \left( e_{12}we_{21}t \otimes e_1 + e_1 \otimes t e_{12} w e_{21} - e_{12} w e_{21} \otimes t e_1 - e_1 t \otimes e_{12} w e_{21}\right)\,, \label{tw}
  \end{align}
\end{subequations}
when the first component $\epsilon t \epsilon$ is a generator of the first type \eqref{type1}, 
\begin{subequations}
  \begin{align}
&\dgal{e_{12} u  \epsilon , \epsilon t \epsilon}_{fus}=\frac12 (e_{12}u \otimes e_1 t- t e_{12} u  \otimes e_1) \,, \label{ut}\\
&\dgal{e_{12} u \epsilon , e_{12} \tilde{u} \epsilon}_{fus}= \frac12 (e_1 \otimes e_{12}u e_{12}\tilde{u}-e_{12}\tilde{u} e_{12}u \otimes e_1) \,, \label{uu}\\
&\dgal{e_{12} u \epsilon ,\epsilon v e_{21} }_{fus}=\frac12(e_{12}u \otimes e_1 ve_{21}-v e_{21} \otimes e_{12}u e_1)\,, \label{uv}\\
&\dgal{e_{12} u\epsilon , e_{12} w e_{21}}_{fus}=\frac12(e_1 \otimes e_{12}u e_{12}we_{21} - e_{12}w  e_{21} \otimes e_{12}u e_1) \,, \label{uw}
  \end{align}
\end{subequations}
when the first component $e_{12} u  \epsilon$ is a generator of the second type \eqref{type2}, 
\begin{subequations}
  \begin{align}
 &\dgal{\epsilon v e_{21} , \epsilon t \epsilon}_{fus}=\frac12 (t e_1 \otimes v e_{21}-e_1 \otimes v e_{21}t) \,, \label{vt}\\
&\dgal{\epsilon v e_{21}, e_{12} u \epsilon}_{fus}= \frac12 ( e_{12}u e_1 \otimes v e_{21} - e_1 v e_{21}\otimes e_{12}u) \,, \label{vu}\\
&\dgal{\epsilon v e_{21},\epsilon \tilde{v} e_{21} }_{fus}=\frac12( \tilde{v} e_{21}v e_{21}\otimes e_1 - e_1 \otimes v e_{21}\tilde{v} e_{21})\,, \label{vv}\\
&\dgal{\epsilon v e_{21} , e_{12} w e_{21}}_{fus}=\frac12(e_{12}we_{21}ve_{21}\otimes e_1 - e_1 v e_{21}\otimes e_{12}we_{21} ) \,, \label{vw}
  \end{align}
\end{subequations}
when the first component $\epsilon v e_{21}$ is a generator of the third type \eqref{type3}, 
\begin{subequations}
  \begin{align}
 &\dgal{e_{12} w e_{21} , \epsilon t \epsilon}_{fus}=\frac12 (t e_1 \otimes e_{12}we_{21}+e_{12}we_{21}\otimes e_1 t - t e_{12}we_{21}\otimes e_1 - e_1 \otimes e_{12}we_{21}t) \,, \label{wt}\\
&\dgal{e_{12} w e_{21}, e_{12} u \epsilon}_{fus}= \frac12 ( e_{12}u e_1 \otimes e_{12}w  e_{21}  -  e_{12}u e_{12}we_{21} \otimes e_1) \,, \label{wu}\\
&\dgal{e_{12} w e_{21},\epsilon v e_{21} }_{fus}=\frac12( e_{12}we_{21} \otimes e_1 v e_{21} - e_1 \otimes e_{12}we_{21}ve_{21})\,, \label{wv}\\
&\dgal{e_{12} w e_{21} , e_{12} \tilde{w} e_{21}}_{fus}=0 \,, \label{ww}
  \end{align}
\end{subequations}
when the first component $e_{12} w e_{21}$ is a generator of the fourth type \eqref{type4}.
\end{lem}

\begin{proof}
Remark that from the definition of the double bracket $\dgal{-,-}_{fus}$ together with \eqref{Eq:BrQ2} we can write 
\begin{equation}
\begin{aligned}
    \dgal{a,b}_{fus}=&
-\frac12 \Tr(E_2)(b)' \Tr(E_1)(a)'' \otimes \Tr(E_1)(a)' \Tr(E_2)(b)'' \\
&+\frac12 \Tr(E_1)(b)' \Tr(E_2)(a)'' \otimes \Tr(E_2)(a)' \Tr(E_1)(b)''\,. \label{Eqfus}
\end{aligned}
\end{equation}
It remains to use \eqref{TrEepsilon}--\eqref{TrEw} to get the required identities. For example, to get \eqref{tu} we find from \eqref{TrEepsilon} and \eqref{TrEe12} 
\begin{equation}
\begin{aligned}
    \dgal{\epsilon t \epsilon, e_{12}u \epsilon}_{fus}=&
-\frac12 \Tr(E_2)(e_{12}u)' \Tr(E_1)(t)'' \otimes \Tr(E_1)(t)' \Tr(E_2)(e_{12}u)'' \\
=&-\frac12 (-e_1 e_1 \otimes t e_1 e_{12}u+e_1 e_1 t \otimes e_1 e_{12}u ) 
=\frac12 e_1 \otimes t e_{12}u - \frac12 e_1 t  \otimes e_{12}u\,.
\end{aligned}
\end{equation}
The exact same method works in each case. Note that only ten cases need to be computed as 
 other double brackets can be obtained by cyclic antisymmetry : $\dgal{b,a}_{fus}=- \dgal{a,b}''_{fus} \otimes \dgal{a,b}'_{fus}$. 
\end{proof}

These explicit forms of the double bracket $\dgal{-,-}_{fus}$ are central in the proof of the next result, which we postpone to Appendix \ref{Ann:Kappa}. 

\begin{lem} \label{LemE1E2}
  Assume that $\dgal{-,-}$ is an arbitrary $B$-linear double bracket on $A$. Consider the induced $B^f$-linear double bracket $\dgal{-,-}$ on $A^f$, and define the double bracket $\dgal{-,-}_{fus}$ as in Theorem \ref{ThmFusBr}. 
Furthermore, set $\dgal{-,-}^f:= \dgal{-,-} + \dgal{-,-}_{fus}$. 
Then the $B^f$-linear map $\kappa: (A^f)^{\times 3} \to (A^f)^{\otimes 3}$ defined by 
\begin{equation*}
\begin{aligned}
  \kappa(-,-,-)=&\dgal{-,-,-}^f-\dgal{-,-,-}-\dgal{-,-,-}_{fus}\,,
\end{aligned}
\end{equation*}
vanishes. (Here, the induced triple brackets on the right-hand side are given by \eqref{Eq:TripBr} using $\dgal{-,-}^f$, $\dgal{-,-}$ and $\dgal{-,-}_{fus}$ respectively.)
\end{lem}

\subsection{Fusion for the double quasi-Poisson bracket} \label{ss:pf1} 

We prove Theorem \ref{ThmFusBr}. To do so, we need to show that $\dgal{-,-,-}^f=\frac{1}{12}\sum_{s\neq 2}\dgal{-,-,-}_{F_i^3}$, where $\dgal{-,-,-}^f$ is the triple bracket associated to the double bracket defined by \eqref{dgalf}. By Lemma \ref{LemE1E2}, we simply have that  
\begin{equation*}
    \dgal{-,-,-}^f = \dgal{-,-,-} + \dgal{-,-,-}_{fus}\,. 
\end{equation*}
By assumption, $\dgal{-,-}$ is quasi-Poisson in $A$, hence $\dgal{-,-,-}$ coincides with the  differential double bracket defined by $\frac{1}{12}\sum_s E_s^3 \in (D_BA)_3$, see \ref{ss:dqP}. We get from Proposition \ref{PropIndbr} that we can write  
$\dgal{-,-,-}=\frac{1}{12}\sum_{s}\dgal{-,-,-}_{\Tr(E_s^3)}$ in $A^f$. 

We rewrite each $\Tr(E_s^3)$ in terms of the gauge elements $F_s$, $s\neq 2$, of $A^f$. Since $E_s=e_s E_s e_s$, 
\begin{equation*}
  \Tr(E_s^3)=\epsilon E_s^3 \epsilon = (\epsilon E_s \epsilon)^3=F_s^3\,,
\end{equation*}
for any $s\neq 1,2$ by Lemma \ref{LemFs}. Similarly, since $e_2=e_{21}\epsilon e_{12}$, 
\begin{equation*}
  \Tr(E_1^3)+\Tr(E_2^3)=\epsilon E_1^3 \epsilon + \epsilon e_{12} E_2^3 e_{21} \epsilon = (\epsilon E_1 \epsilon)^3+(\epsilon e_{12} E_2 e_{21} \epsilon)^3\,.
\end{equation*}
Modulo graded commutators, we can write 
\begin{equation*}
\Tr(E_1^3)+\Tr(E_2^3)=[\Tr(E_1)+\Tr(E_2)]^3-3\Tr(E_1)\Tr(E_2)^2-3\Tr(E_1)^2 \Tr(E_2)\,,  
\end{equation*}
which is 
$F_1^3-3\Tr(E_1)\Tr(E_2)^2-3\Tr(E_1)^2 \Tr(E_2)$ using Lemma \ref{LemF12}. By Proposition \ref{Prop:BrQ}, the map $\mu$ defines $n$-brackets modulo graded commutators in $D_{B^f}A^f$ so that 
\begin{equation*}
  \dgal{-,-,-}^f=\frac{1}{12}\sum_{s\neq 2}\dgal{-,-,-}_{F_i^3}-\frac14 \dgal{-,-,-}_{\Tr(E_1)\Tr(E_2)^2+\Tr(E_1)^2\Tr(E_2)}+ \dgal{-,-,-}_{fus}\,.
\end{equation*}
Now, by Proposition \ref{TripleDiff}, the bracket $\dgal{-,-,-}_{fus}$ is defined by $\frac18 \brSN{\Tr(E_1)\Tr(E_2),\Tr(E_1)\Tr(E_2)}$. After a short computation (given e.g. in \cite[\S 5.3]{VdB1}), we find that  
\begin{equation}
   \brSN{\Tr(E_1)\Tr(E_2),\Tr(E_1)\Tr(E_2)} = 2 \Tr(E_1)^2\Tr(E_2)+2\Tr(E_1)\Tr(E_2)^2\,,
\end{equation}
modulo graded commutators, which finishes the proof. 

\subsection{Fusion for the  moment map} \label{ss:pf2}
Note that $\Phi^f$ has an inverse 
\begin{equation*}
  (\Phi^f)^{-1}=e_1 \Tr(\Phi_2^{-1}) \Tr(\Phi_1^{-1}) e_1 + \sum_{s\neq 1,2} e_s \Tr(\Phi_s^{-1}) e_s\,,
\end{equation*}
so that Theorem \ref{ThmFusMomap} directly follows from the following lemma. 
\begin{lem} \label{LemMomap}
  Assume that $s\neq 1,2$. Then for any $a \in A^f$
\begin{equation}
  \dgal{\Tr(\Phi_s),a}^f=\frac12 (a e_s \otimes \Tr(\Phi_s) + a \Tr(\Phi_s) \otimes e_s - e_s \otimes \Tr(\Phi_s) a - \Tr(\Phi_s) \otimes e_s a )\,. \label{EqPhis}
\end{equation}
If we set $\Phi^f_1=\Tr(\Phi_1)\Tr(\Phi_2)$, we have for any $a \in A^f$
\begin{equation}
  \dgal{\Phi^f_1,a}^f=\frac12 (a e_1 \otimes \Phi^f_1 + a \Phi^f_1 \otimes e_1 - e_1 \otimes \Phi^f_1 a - \Phi^f_1 \otimes e_1 a )\,. \label{EqPhi12}
\end{equation}
\end{lem}
The proof consists of checking \eqref{EqPhis} and \eqref{EqPhi12} on generators, which is done in Appendix \ref{App:PfMomap}.


\section{Applications} \label{appli}

\subsection{Elementary examples of fusion} Given two double quasi-Poisson algebras $(A,\dgal{-,-})$ and $(A',\dgal{-,-}')$ over $\kk$, we can use Remark \ref{RemOplus} to get a double quasi-Poisson bracket on $A \oplus A'$ which is $B$-linear for $B=\kk e_1\oplus \kk e_2$ with $e_1=(1,0)$ and $e_2=(0,1)$. Using Theorem \ref{ThmFusBr}, we can get a double quasi-Poisson bracket on the fusion algebra  $(A \oplus A')^f$ obtained by fusing $e_2$ onto $e_1$. By iterating this process, we can create new  double quasi-Poisson algebras using the different examples given in Section \ref{classif}. (The same holds for quasi-Hamiltonian algebras if we have moment maps.) 
Nevertheless, as far as we use differential double brackets, one could argue that this could already be done using Van den Bergh's results \cite[Theorems 5.3.1,5.3.2]{VdB1}. Hence, we now give new examples that involve double brackets that are not differential. 
To do so, recall from Example \ref{Exmpl1} that for any $k \geq 3$, $\kk[x]/(x^k)$ has a double bracket given by $\dgal{x,x}=\frac12(x^2 \otimes 1 - 1 \otimes x^2)$ which is not differential. The double bracket is in fact quasi-Poisson, e.g. as a consequence of Proposition \ref{Pr:Free1}. 

\begin{exmp}
  Fix $k \geq 3$ and form $A=\kk[x]/(x^{k})$ which is a double quasi-Poisson algebra. Let $A'$ be an arbitrary double quasi-Poisson $\kk$-algebra. Then we can consider $A \oplus A'$ with idempotents $e_1=(1,0)$, $e_2=(0,1)$. For $B=\kk e_1 \oplus \kk e_2$, $A \oplus A'$ has a $B$-linear double quasi-Poisson bracket by Remark \ref{RemOplus}. We can form the fusion algebra $\bar{A}=(A\oplus A')^f$ obtained by fusing $e_2$ onto $e_1$, which we see as an algebra over $\kk$ by identifying the only remaining non-zero idempotent $e_1$ with $1$. Using Lemma \ref{AfGenerat}, $\bar{A}$ is the algebra generated by $x$ and $e_{12}w e_{21}$ for $w \in A'$. Thus, we can identify $\bar{A}$ with $A \ast_{\kk}A'$, and see the elements of $A$ as generators of type 1 \eqref{type1} after fusion, while the elements of $A'$ are generators of type 4 \eqref{type4}. Therefore, using Theorem \ref{ThmFusBr}, we have a double quasi-Poisson bracket on $\bar{A}$ given by 
\begin{equation*}
  \dgal{x,w}=\frac12(wx \otimes 1 + 1 \otimes xw - w \otimes x - x \otimes w)\,, \quad w \in A'\,,
\end{equation*}
if we use \eqref{tw} in Lemma \ref{dbrFUS}, while the double brackets $\dgal{x,x}$ and $\dgal{w,w'}$ for $w,w'\in A'$ are just the ones in $A$ and $A'$ respectively. 
\end{exmp}

\begin{exmp} \label{Exmp:xN}
  Fix integers $M \geq 1$ and $k_s\geq 3$ for $1\leq s \leq M$. We can form $A_s=\kk[x_s]/(x_s^{k_s})$ and consider $A=\oplus_s A_s$ where we denote each unit by $e_s$ so that $A$ is an algebra over $B=\oplus_s \kk e_s$. Moreover, it has a double quasi-Poisson bracket by Remark \ref{RemOplus}. 
Fusing $e_2$ onto $e_1$, then $e_3$ onto $e_1$ and so on up to $e_M$, we get the fusion algebra 
\begin{equation*}
  A'=\kk \langle x_1,\ldots, x_M\rangle /I\,, \quad 
\text{where }I\text{ is the ideal generated by }x_1^{k_1},\ldots,x_M^{k_M}\,,
\end{equation*}
which is just a $\kk$-algebra. 
By Theorem \ref{ThmFusBr} and Lemma \ref{dbrFUS}, $A'$ has a double quasi-Poisson bracket given by 
\begin{equation*}
  \begin{aligned}
    \dgal{x_s,x_s}=\frac12(x_s^2 \otimes 1 - 1 \otimes x_s^2)\,,& \quad 1\leq s \leq M\,, \\
\dgal{x_r,x_s}=\frac12 (x_s x_r \otimes 1 + 1 \otimes x_r x_s - x_s \otimes x_r - x_r \otimes x_s)\,, &
\quad 1\leq r < s \leq M\,.
  \end{aligned}
\end{equation*}
\end{exmp}

I have been unable to find a quasi-Hamiltonian algebra whose double bracket is \emph{not} differential. It is an interesting question to determine if such an example exists, in order to see whether Theorem \ref{ThmFusMomap}  is strictly stronger than  \cite[Theorem 5.3.2]{VdB1} or not.

\subsection{Revisiting Van den Bergh's double bracket for quivers} \label{ss:quivers}

\subsubsection{Generalities} \label{sss:quivGen}

Let $Q$ be a finite quiver with vertex set denoted $I$. We define the functions $t,h: Q \to I$ that associate to an arrow $a$ either its tail $t(a)\in I$ or its head $h(a)\in I$. We form the double $\bar{Q}$ of the quiver $Q$ with the same vertex set $I$ by adding an opposite arrow $a^\ast : h(a)\to t(a)$ to each $a\in Q$. We naturally extend $h,t$ to $\bar{Q}$, and set $(a^\ast)^\ast=a$ for each $a\in Q$ so that the map $a \mapsto a^\ast$, $a \in \bar{Q}$, defines an involution on $\bar{Q}$.  
We form the path algebra $\kk \bar{Q}$ which is the $\kk$-algebra generated by the arrows $a\in \bar{Q}$ and idempotents $(e_s)_{s\in I}$ labelled by the vertices such that 
\begin{equation*}
  a=e_{t(a)}a e_{h(a)}\,, \quad e_s e_t = \delta_{st}\,e_s\,.
\end{equation*}
This implies that we read paths from left to right. We see $\kk \bar{Q}$ as a $B$-algebra with $B=\oplus_{s\in I}\kk e_s$.

We define $\epsilon:\bar{Q}\to \{\pm1\}$ as the map which takes value $+1$ on arrows originally in $Q$, and $-1$ on the arrows in $\bar{Q}\setminus Q$. For each $a\in Q$, we also choose $\gamma_a \in \kk$ and set $\gamma_{a^\ast}=\gamma_a$. Finally, we associate to $\kk \bar{Q}$ the algebra $A$ obtained by universal localisation from the set $S=\{1+(\gamma_a-1)e_{t(a)}+a a^\ast \,|\, a \in \bar{Q}\}$. This is equivalent to add local inverses $(\gamma_a e_{t(a)}+a a^\ast)^{-1}$ for each $a \in \bar{Q}$ (i.e. they are inverses to $\gamma_a e_{t(a)}+a a^\ast$ in $e_{t(a)}Ae_{t(a)}$). If $\gamma_a=0$, then $a^{-1}:=a^\ast(a a^\ast)^{-1}$ satisfies $a^{-1}=(a^\ast a)^{-1}a^\ast$, so that $a a^{-1}=e_{t(a)}$ and $a^{-1}a=e_{h(a)}$; the same holds for $a^\ast$.

\subsubsection{The quasi-Hamiltonian structure} 

 For each vertex $s\in I$, consider a total ordering $<_s$ on the set $T_s=\{a\in \bar{Q}\mid t(a)=s\}$. Write $o_s(-,-)$ for the ordering function at vertex $s$ : on arrows $a,b$ we have $o_s(a,b)=+1$ if $a<_s b$, $o_s(a,b)=-1$ if $b<_s a$, while it is zero otherwise, i.e. if $a=b\in T_s$, if $a \notin T_s$ or if $b \notin T_s$. 

\begin{thm} \label{ThmVdB}
   The algebra $A$ has a double quasi-Poisson bracket defined by 
  \begin{subequations}
        \begin{align}
 \dgal{a,a}\,=\,&\frac{1}{2}o_{t(a)}(a,a^*)\left( a^2\otimes e_{t(a)}- e_{h(a)}\otimes a^2 \right) \qquad (a\in\bar{Q})\,, \label{loopG}\\
 \dgal{a,a^*}\,=\,&\gamma_a e_{h(a)}\otimes e_{t(a)}
 +\frac{1}{2} a^*a\otimes e_{t(a)} +\frac{1}{2} e_{h(a)}\otimes aa^* \nonumber\\
 & +\frac{1}{2}o_{t(a)}(a,a^*)\, (a^*\otimes a-a\otimes a^*)\qquad (a\in Q)\,, \label{aastG}
        \end{align}
  \end{subequations}
 and for $b,c\in\bar{Q}$ such that $ c\ne b,b^*$ 
 \begin{equation}
  \begin{aligned}
 \dgal{b,c}\,=\,&-\frac{1}{2}o_{t(b)}(b,c)\,(b\otimes c)-\frac{1}{2}o_{h(b)}(b^*,c^*)\,(c\otimes b)
\\ \label{a<bG}
 &+\frac{1}{2}o_{t(b)}(b,c^*)\, cb\otimes e_{t(b)} + \frac{1}{2}o_{h(b)}(b^*,c)\,e_{h(b)}\otimes bc   \,.
  \end{aligned}
 \end{equation}
Furthermore, $A$ is quasi-Hamiltonian for the multiplicative moment map 
\begin{equation} \label{EqPhiVdB}
  \Phi=\sum_s \Phi_s\,, \quad \Phi_s=\prod_{a\in T_s}^{\longrightarrow}(\gamma_a e_s+ a a^\ast)^{\epsilon(a)}\,.
\end{equation}
\end{thm}
In \eqref{EqPhiVdB}, we take the product defining $\Phi_s$ with respect to the ordering on $T_s$. 
If all $\gamma_a=+1$, this result explicitly gives the double bracket defined from a poly-vector field $P \in (D_BA)_2$ in \cite[Theorem 6.7.1]{VdB1}, which was written in the above form for particular choices of ordering in \cite[Proposition 2.6]{CF}. 
In fact, if all $\gamma_a \neq0$, the result is equivalent to the previous case up to rescaling. 
If some $\gamma_a$ are equal to zero, our result also encompasses the generalisation proposed in \cite[Proposition 2.7]{CF}.

\subsubsection{Proof of Theorem \ref{ThmVdB}}  As in the proof of \cite[Theorem 6.7.1]{VdB1}, we begin with the quiver $Q^{sep}$ which has vertex and arrow sets given by 
\begin{equation}
  I^{sep}=\{v_b,\,v_{b^\ast}\mid b \in Q\}\,, \quad 
Q^{sep}=\{b:v_b \to v_{b^\ast} \mid b \in Q \}\,.
\end{equation}
We form the double $\bar{Q}^{sep}$ of $Q^{sep}$, which amounts to add the arrows $\{b^\ast: v_{b^\ast} \to v_b \mid b \in Q \}$. We define on it the involution $\ast$ given by $b \mapsto b^\ast$ and $b^\ast \mapsto b$. 
We add local inverses $(\gamma_b e_{v_b} + bb^\ast)^{-1}$ in $\kk \bar{Q}^{sep}$ for all $b\in \bar{Q}^{sep}$ to get the algebra $A^{sep}$.  By combining Example \ref{ExpVdBQ1} (with $t=b,s=b^\ast$ for each $b\in Q^{sep}$) and Remark \ref{RemOplus}, $A^{sep}$ is quasi-Hamiltonian for the double quasi-Poisson bracket given by 
\begin{equation} \label{Eq:BasicBr}
  \dgal{b,b^\ast}=\gamma_b e_{v_{b^\ast}}\otimes e_{v_b} + \frac12 b^\ast b \otimes e_{v_{b}} + \frac12 e_{v_{b^\ast}} \otimes bb^\ast\,,
\end{equation}
for all $b\in Q^{sep}$ and which is zero on every other pair of generators, while  the multiplicative moment map is  defined as 
\begin{equation}
  \Phi=\sum_{b \in \bar{Q}^{sep}} \Phi_{v_b}\,,
\quad \Phi_{v_b}=(\gamma_b e_{v_b}+bb^\ast)^{\epsilon(b)}\,. 
\end{equation}

To get a quasi-Hamiltonian structure on $A$, it remains to fuse all these disjoint quivers of $\bar{Q}^{sep}$ according to the ordering that we chose at the vertices of $\bar{Q}$. More precisely, label the vertices in the quiver $\bar{Q}$ as $\{1,\ldots, |I|\}$, and label the arrows according to the ordering, that is if the arrow $b$ is the $k$-th element with respect to the total ordering on $T_s$ (going from the minimal to the maximal element in the chain) where $s=t(b)$, we label it $a_{s,k}$. We use the same names for the arrows in $\bar{Q}^{sep}$. To recover $\bar{Q}$, we rename $v_{a_{1,1}}$ as $1$, then fuse $1$ and $v_{a_{1,2}}$ which we still name $1$, then continue with all vertices labelled $v_{a_{1,k}}$ for increasing values of $k$. Next, we do the same for vertices $2,\ldots, |I|$ and recover the quiver $\bar{Q}$. 
In terms of algebras, this means that we consider the fusion algebra obtained after fusing $e_{v_{a_{1,2}}}$ onto $e_1$, then $e_{v_{a_{1,3}}}$ onto $e_1$, and so on. This finally yields the algebra $A$. Therefore, it suffices to use Theorems  \ref{ThmFusBr} and \ref{ThmFusMomap} to get the desired result. We directly find  that $\Phi$ is given by \eqref{EqPhiVdB}, but understanding the double bracket requires some work. 

We first show \eqref{loopG} and \eqref{aastG}, where there is nothing to prove if $a$ is not a loop. So assume that $a$ is a loop, and $a <_{t(a)}a^\ast$. By construction the only new terms arise when we glue $w_1:=v_a$ with $w_2:=v_{a^\ast}$, so to compute these terms we use Theorem  \ref{ThmFusBr} with the vertices $w_1,w_2$ respectively playing the role of $1,2$. 
We have that after fusion $a$ is a generator of third type \eqref{type3}, so that by \eqref{vv} the fusion amounts to add a term $\frac12 a^2 \otimes e_{t(a)}-\frac12 e_{t(a)}\otimes a^2$ in $\dgal{a,a}$. Similarly, $a^{\ast}$ is a generator of second type \eqref{type2} so by \eqref{uu} we get a term $\frac12 e_{t(a)} \otimes (a^\ast)^2 - \frac12 (a^\ast)^2 \otimes e_{t(a)}$ in $\dgal{a^\ast,a^\ast}$. 
Using \eqref{vu}, we get an additional term $\frac12 a^\ast \otimes a - \frac12 a \otimes a^\ast$ in $\dgal{a,a^\ast}$, which gives the correct double bracket by adding \eqref{Eq:BasicBr}.  
In the case  $a^\ast <_{t(a)}a$, take $w_1:=v_{a^\ast}$ with $w_2:=v_{a}$ and the proof is similar, but now $a$ is of second type and $a^\ast$ is of third type.

Before proving \eqref{a<bG}, we need some preparation.   Consider $\alpha,\beta \in \bar{Q}$ and $s\in I$ with $\alpha<_s \beta$, $\alpha \neq \beta,\beta^\ast$. 
With the labelling given above, we have that $\alpha=a_{s,k_0}$, $\beta=a_{s,k_1}$ for some $1\leq k_0<k_1 \leq |T_s|$, and $v_{\alpha}=v_{a_{s,k_0}}$, $v_{\beta}=v_{a_{s,k_1}}$. 
Write $\bar{Q}^{\alpha}$ for the quiver obtained from $\bar{Q}^{sep}$ by fusing all the vertices $v_{a_{s',k}}$ with either $s'<s$, or $s'=s$ with $k<_s k_1$ (i.e. we fuse all vertices up to excluding $v_\beta$); set $t_\alpha$ and $h_\alpha$ for the tail and head maps in $\bar{Q}^\alpha$.  Write $\bar{Q}^{\beta}$ for the quiver obtained from $\bar{Q}^\alpha$ by additionally fusing the vertex $v_{a_{s,k_1}}$ (i.e. we fuse all vertices in $\bar{Q}^{sep}$ up to including $v_\beta$). Set again $t_\beta$ and $h_\beta$ for the associated tail and head maps.  We let $A^{\alpha}$ and $A^\beta$ respectively denote the algebras obtained from $A^{sep}$ by fusion to arrive at the quivers $\bar{Q}^\alpha$ and $\bar{Q}^\beta$.  

\begin{lem}
The step of performing fusion from $A^\alpha$ to $A^\beta$ amounts to add the following terms in the double quasi-Poisson bracket of $A$ between the elements $\alpha,\alpha^\ast$ and $\beta,\beta^\ast$ : 
\begin{subequations}
  \begin{align}
& -\frac12 \alpha \otimes \beta + \frac12 \delta_{t_\beta(\alpha),t_\beta(\alpha^\ast)}\, e_{t_\beta(\alpha)}\otimes \alpha \beta\quad \text{ in }\dgal{\alpha,\beta}\,, \label{TT} \\
&+\frac12 \beta^\ast \alpha \otimes e_{t_\beta(\alpha)} - \frac12 \delta_{t_\beta(\alpha),t_\beta(\alpha^\ast)}\, \beta^\ast\otimes \alpha \quad \text{ in }\dgal{\alpha,\beta^\ast}\,, \label{TH} \\
&+\frac12 e_{t_\beta(\alpha)} \otimes \alpha^\ast\beta - \frac12 \delta_{t_\beta(\alpha),t_\beta(\alpha^\ast)}\, \alpha^\ast\otimes \beta\quad \text{ in }\dgal{\alpha^\ast,\beta}\,, \label{HT} \\
&-\frac12 \beta^\ast \otimes \alpha^\ast + \frac12 \delta_{t_\beta(\alpha),t_\beta(\alpha^\ast)}\,  \beta^\ast \alpha^\ast \otimes e_{t_\beta(\alpha)}\quad \text{ in }\dgal{\alpha,\beta}\,. \label{HH}
  \end{align}
\end{subequations}
\end{lem}
\begin{proof}
We know that $h_\alpha(\beta)\neq t_\alpha(\beta)$ (otherwise it would contradict the order in which we glue vertices), so we have that $\alpha,\alpha^\ast$ are generators of the first type, $\beta$ is a generator of the second type and $\beta^\ast$ is a generator of the third type in the algebra $A^\beta$ obtained after fusing $w_1:=v_\alpha$ and $w_2:=v_\beta$. We have by \eqref{tu} that the following terms appear in the double quasi-Poisson bracket $\dgal{-,-}_\beta$ on $A^\beta$ for $\dgal{\alpha,\beta}_\beta$ : 
$\frac12 (e_{w_1}\otimes \alpha \beta - \alpha \otimes \beta)$. The first term is non-zero only if $h_\beta(\alpha)=t_\beta(\beta)$, or  $t_\beta(\alpha^\ast)=t_\beta(\alpha)$, hence we can multiply it by $\delta_{t_\beta(\alpha),t_\beta(\alpha^\ast)}$. After all fusions are performed, $w_1$ is just $t_\beta(\alpha)$ and we get \eqref{TT}. 

Using again \eqref{tu} then twice \eqref{tv} amounts to add the terms 
\begin{equation*}
  \begin{aligned}
     &  \frac12 (e_{w_1}\otimes \alpha^\ast \beta - e_{w_1}\alpha^\ast \otimes \beta) \,\, \text{ in }\dgal{\alpha^\ast,\beta}_\beta\, \, , \\
     & \frac12 (\beta^\ast \alpha \otimes e_{w_1}-\beta^\ast \otimes \alpha e_{w_1})\,\, \text{ in }\dgal{\alpha,\beta^\ast}_\beta\, \, , \\
     & \frac12 (\beta^\ast \alpha^\ast \otimes e_{w_1}-\beta^\ast \otimes \alpha^\ast)\,\, \text{ in }\dgal{\alpha^\ast,\beta^\ast}_\beta\, \, .  
  \end{aligned}
\end{equation*}
A  discussion as in the first case allows to get \eqref{TH}--\eqref{HH}.   
\end{proof}

To prove \eqref{a<bG}, we have to show that the equality holds for any kind of ordering when the two arrows meet, as it is trivially zero if they do not. We first show what happens if they meet at exactly one vertex. 

If $t(b)=t(c)$, assuming that $b<_{t(b)}c$ we get by \eqref{TT} with $\alpha=b,\beta=c$ a term $-\frac12 b \otimes c$ in $\dgal{b,c}$. If instead $c<_{t(b)}b$, we get by \eqref{TT} with $\alpha=c,\beta=b$ a term $-\frac12 c \otimes b$ in $\dgal{c,b}$, hence a term $+\frac12 b \otimes c$ in $\dgal{b,c}$ by cyclic antisymmetry. This proves \eqref{a<bG} in this case. 

Next, assuming only $t(b)=h(c)$ and $b<_{t(b)}c^\ast$, we get by \eqref{TH} with $\alpha=b,\beta=c^\ast$ a term $+\frac12 cb \otimes e_{t(b)}$ in $\dgal{b,c}$. If $c^\ast <_{t(b)}b$, we use \eqref{HT} with $\alpha=c^\ast,\beta=b$ to get a term  $+\frac12 e_{t(c^\ast)}\otimes cb$ in $\dgal{c,b}$, so this gives $-\frac12 cb \otimes e_{t(b)}$ as expected.  

Then, for $h(b)=t(c)$ with $b^\ast <_{h(b)}c$, we have from \eqref{HT} with $\alpha=b^\ast,\beta=c$ the term $+\frac12 e_{t(b^\ast)}\otimes bc$ in $\dgal{b,c}$. If $c<_{h(b)}b^\ast$ instead, we have from \eqref{TH} with $\alpha=c,\beta=b^\ast$ the term $+\frac12 bc \otimes e_{t(c)}$ in $\dgal{c,b}$, which yields $-\frac12 e_{h(b)} \otimes bc$ in $\dgal{b,c}$ and also finishes this case. 

Finally, we assume $h(b)=h(c)$. If $b^\ast<_{h(b)} c^\ast$, we get by \eqref{HH} with $\alpha=b^\ast,\beta=c^\ast$ the contributing term $-\frac12 c \otimes b$ in $\dgal{b,c}$, while for   $c^\ast<_{h(b)} b^\ast$ we obtain a term  $-\frac12 b \otimes c$ in $\dgal{c,b}$, and thus $+\frac12  c \otimes b$ in $\dgal{b,c}$ as desired. 

If $b,c$ meet at two vertices but none of them is a loop, we can conclude by adding together the two corresponding results just derived. Hence, it remains the tedious computation to check the cases when at least $b$ or $c$ is a loop. We now write two illuminating cases where $h(b)=t(b)=t(c)$, and leave to the reader the task to verify all the remaining cases (noting that we only need to check half these cases because of the cyclic antisymmetry) using \eqref{TT}--\eqref{HH}. 

Assume that  $h(b)=t(b)=t(c)$ and $b<_{t(b)}b^\ast<_{h(b)}c$. When we first glue the vertices $v_{b},v_{b^\ast}$ in $\bar{Q}^{sep}$ corresponding to $t(b),h(b^\ast)$, no term contributes to $\dgal{b,c}$. Hence, we only need to understand what happens when we glue the vertices corresponding to $t(b)=h(b)$ and $t(c)$, and by \eqref{TT} with $\alpha=b,\beta=c$ we get the term $-\frac12 b \otimes c + \frac12 e_{t(b)}\otimes bc$, as expected.  
(Alternatively, we could have used $\eqref{HT}$ with $\alpha=b^\ast,\beta=c$ to get the same answer. It is important to remark that we glue vertices not arrows, so that only one of these two cases has to be considered, not both together.) 

Assume that  $h(b)=t(b)=t(c)$ and $b<_{t(b)}c<_{h(b)}b^\ast$. When gluing the vertices of $\bar{Q}^{sep}$ corresponding to $t(b)$ and $t(c)$, we get by \eqref{TT} with $\alpha=b,\beta=c$ the only term $-\frac12 b \otimes c$ contributing to $\dgal{b,c}$ since $b$ is not (yet) a loop. Next, when we glue $t(c)=t(b)$ and $h(b)$, we get by \eqref{TH} with $\alpha=c,\beta=b^\ast$ a term $+\frac12 bc \otimes e_{t(c)}$ in $\dgal{c,b}$ since $c$ is not a loop, hence the term $-\frac12 e_{t(b)}\otimes bc$ contributes to $\dgal{b,c}$ and we are done.

\begin{figure}[ht]
   \begin{tikzpicture}[scale=2]
     \draw (-3,0) arc [start angle = 180, end angle = 30,  x radius = 14mm, y radius = 8mm]; 
     \draw (-3,0) arc [start angle = 180, end angle = 330,  x radius = 14mm, y radius = 8mm];
     \draw (-2.3,0) arc [start angle = 180, end angle = 0,  x radius = 7mm, y radius = 3mm];
     \draw (-2.45,0.1) arc [start angle = 210, end angle = 330,  x radius = 10mm, y radius = 5mm];
     \draw[red] (-2.6,0.1) arc [start angle = 180, end angle = 0,  x radius = 10mm, y radius = 5mm]; 
     \draw[red,->] (-2.6,0.1) arc [start angle = 180, end angle = 360,  x radius = 10mm, y radius = 5mm];
     \node [red] (a) at (-0.65,0.4) {$\alpha$};
     \draw[red] (-1.6,-0.15) arc [start angle = 100, end angle = 260,  x radius = 1.5mm, y radius = 3.3mm]; 
     \draw[red,->] (-1.6,-0.8) arc [start angle = 260, end angle = 200,  x radius = 1.5mm, y radius = 3.3mm]; 
     \draw[red,dashed] (-1.6,-0.15) arc [start angle = 80, end angle = -80,  x radius = 1.5mm, y radius = 3.3mm];
     \node [red] (b) at (-1.85,-0.6) {$\beta$};
\draw[black] (-0.388,0.4) to[out=-40,in=200] (0.3,0.3);
\draw[black] (-0.388,-0.4) to[out=40,in=-200] (0.3,-0.3);
     \draw (0.3,0.3) arc [start angle = 100, end angle = 260,  x radius = 1.4mm, y radius = 3.05mm]; 
     \draw (0.3,0.3) arc [start angle = 80, end angle = -80,  x radius = 1.4mm, y radius = 3.05mm]; 
     \node (ast) at (0.5,0) {$\ast$};
     \node[circle,fill=black,inner sep=0pt,minimum size=3pt] (circ) at (0.42,0) {};
     \draw[blue] (0,0.245) arc [start angle = 100, end angle = 260,  x radius = 1.2mm, y radius = 2.5mm]; 
     \draw[blue,->] (0,0.245) arc [start angle = 100, end angle = 185,  x radius = 1.2mm, y radius = 2.5mm]; 
     \draw[blue,dashed] (0,0.245) arc [start angle = 80, end angle = -80,  x radius = 1.2mm, y radius = 2.5mm]; 
     \node [blue] (c) at (-0.2,0.15) {$\Phi$};
\draw (2,0.5) -- (4,0.5); 
\draw (2,-0.5) -- (4,-0.5); 
     \draw (2,0.5) arc [start angle = 100, end angle = 260,  x radius = 1.4mm, y radius = 5.08mm]; 
     \draw[dashed] (2,0.5) arc [start angle = 80, end angle = -80,  x radius = 1.4mm, y radius = 5.08mm]; 
     \draw (4,0.5) arc [start angle = 100, end angle = 260,  x radius = 1.4mm, y radius = 5.08mm]; 
     \draw (4,0.5) arc [start angle = 80, end angle = -80,  x radius = 1.4mm, y radius = 5.08mm]; 
     \node (ast2) at (4.2,0) {$\ast$};
     \node[circle,fill=black,inner sep=0pt,minimum size=3pt] (circ) at (4.12,0) {};
     \draw[blue] (3.6,0.5) arc [start angle = 100, end angle = 260,  x radius = 1.4mm, y radius = 5.08mm]; 
     \draw[blue,->] (3.6,0.5) arc [start angle = 100, end angle = 185,  x radius = 1.4mm, y radius = 5.08mm]; 
     \draw[blue,dashed] (3.6,0.5) arc [start angle = 80, end angle = -80,  x radius = 1.4mm, y radius = 5.08mm]; 
     \node [blue] (cdeux) at (3.4,0.15) {$\Phi$};
     \draw[red] (2.5,0.5) arc [start angle = 100, end angle = 260,  x radius = 1.4mm, y radius = 5.08mm]; 
     \draw[red,->] (2.5,0.5) arc [start angle = 100, end angle = 185,  x radius = 1.4mm, y radius = 5.08mm]; 
     \draw[red,dashed] (2.5,0.5) arc [start angle = 80, end angle = -80,  x radius = 1.4mm, y radius = 5.08mm]; 
     \node [red] (g) at (2.3,0.15) {$\gamma$};
  \end{tikzpicture}
  \caption{A system of loops on $\Sigma$ in the cases $(g,r)=(1,0)$ and $(g,r)=(0,1)$. They can be used as  generators for $\pi_1(\Sigma,\ast)$ after being connected to the base point $\ast \in \partial\Sigma$ in a natural way.}
  \label{fig:Pi}
\end{figure}
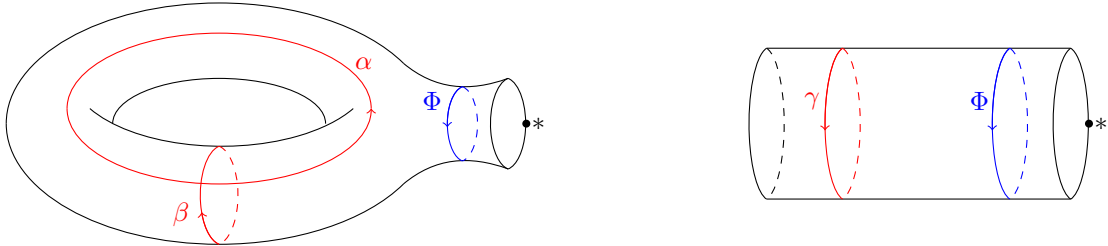

\subsection{Double quasi-Poisson brackets for fundamental groups of surfaces} \label{ss:surfaces} Let $\Sigma$ denote a compact connected surface with fixed orientation, and  such that it has a non-empty boundary $\partial \Sigma$. We denote by $g\geq 0$ its genus, and $r+1\geq 1$ the number of boundary components. Let $\ast \in \partial \Sigma$ be a base point, and denote by $\pi_1(\Sigma,\ast)$ the corresponding fundamental group of $\Sigma$.  The algebra $A=\kk \pi_1(\Sigma,\ast)$ can be presented in terms of generators $\alpha_i^{\pm 1},\beta_i^{\pm1}, \gamma_k^{\pm1},\Phi^{\pm1}$, $1\leq i \leq g$, $1\leq k \leq r$, subject to the relation 
\begin{equation} \label{Eq:Surface}
  \prod^{\rightarrow}_{1\leq i\leq g}[\alpha_i,\beta_i]\prod^{\rightarrow}_{1 \leq k \leq r}\gamma_k=\Phi\,.
\end{equation}
Here, $\Phi$ represents the loop around the boundary component containing $\ast$ (with suitable orientation, see Figure \ref{fig:Pi}), and we used the multiplicative commutator $[\alpha,\beta]=\alpha\beta \alpha^{-1}\beta^{-1}$. Note that in the products we write the factors from the left to the right with increasing indices.

Our aim is to give an alternative proof relying only on fusion of the next result due to Massuyeau and Turaev \cite{MT14}, which endows $A$ with a quasi-Hamiltonian algebra structure. (We rescale their double bracket by a  factor $1/2$.) Hence, this proof is the  non-commutative analogue of the fusion process for representation varieties \cite{QuasiP}.

\begin{thm} \label{ThmPI}
For the presentation considered above, the algebra $A=\kk \pi_1(\Sigma,\ast)$ has a double quasi-Poisson bracket defined for any $1\leq i \leq g$ by 
\begin{equation}
  \begin{aligned}
    \dgal{\alpha_i,\alpha_i}=&\frac12( \alpha_i^2 \otimes 1 - 1 \otimes \alpha_i^2)\,, \quad 
   \dgal{\beta_i,\beta_i}=-\frac12(\beta_i^2 \otimes 1 - 1 \otimes \beta_i^2) \,, \\
     \dgal{\alpha_i,\beta_i}=&\,\,\,\frac12\, \left(\beta_i \alpha_i \otimes 1 + 1 \otimes \alpha_i \beta_i - \alpha_i \otimes \beta_i + \beta_i \otimes \alpha_i \right) \,  \,, 
  \end{aligned} 
\end{equation}
for any $\phi_i \in \{\alpha_i,\beta_i\}$, $1\leq i \leq g$, and $i<j$, it is defined by 
\begin{equation} \label{Eq:phiij}
  \dgal{\phi_i,\phi_j}=\frac12 (\phi_j \phi_i \otimes 1 + 1 \otimes \phi_i \phi_j - \phi_i \otimes \phi_j - \phi_j \otimes \phi_i)\,,
\end{equation}
for any $\phi_i \in \{\alpha_i,\beta_i\}$, $1\leq i \leq g$, and $1 \leq k \leq r$, it is defined by 
\begin{equation}
  \dgal{\phi_i,\gamma_k}
=\frac12 (\gamma_k \phi_i \otimes 1 + 1 \otimes \phi_i \gamma_k - \phi_i \otimes \gamma_k - \gamma_k \otimes \phi_i)\,,
\end{equation}
and for any $1\leq k \leq r$ and $k<l$, it is defined by   
\begin{equation}
\begin{aligned}
  \dgal{\gamma_k,\gamma_k}=&\frac12(\gamma_k^2 \otimes 1 - 1 \otimes \gamma_k^2)\,, \\
\dgal{\gamma_k,\gamma_l}=&\frac12 (\gamma_l \gamma_k \otimes 1 + 1 \otimes \gamma_k \gamma_l - \gamma_k \otimes \gamma_l - \gamma_l \otimes \gamma_k).
\end{aligned}
\end{equation}
Furthermore, for any $a=\alpha_i,\beta_i,\gamma_k$, the double bracket with $\Phi$ is given by 
\begin{equation} 
 \dgal{\Phi,a}=\frac12 (a\otimes \Phi-1 \otimes \Phi a +  a \Phi \otimes 1-\Phi \otimes a)\,.
\end{equation}
In particular, $\Phi$ is a multiplicative moment map, and $A$ is quasi-Hamiltonian. 
\end{thm}

\begin{proof}
We skip the trivial case $g=r=0$ where $A=\kk$. 
  If $g=0,r=1$, we have the generators of the boundary components, call them $\gamma,\Phi$, with $\Phi$ corresponding to the component containing $\ast \in \partial \Sigma$.  
Note that the algebra $\kk [\gamma^{\pm1}]$ has a double quasi-Poisson bracket $\dgal{\gamma,\gamma }=\frac12(\gamma^2 \otimes 1 - 1 \otimes \gamma^2)$ such that $\gamma$ is a moment map as we show in \ref{classifk1}. Since it is isomorphic to  $A_0=\kk \langle \gamma^{\pm1},\Phi^{\pm1}\rangle/(\gamma=\Phi)$, we have a quasi-Hamiltonian algebra structure on $A=A_0$.

If $g=1,r=0$, we have two generating cycles $\alpha,\beta$ and the generator of the boundary component $\Phi$, so that $A$ is just $A_1=\kk\langle \alpha^{\pm1},\beta^{\pm1},\Phi^{\pm1}\rangle /([\alpha,\beta]=\Phi)$. But the algebra $\kk\langle \alpha^{\pm1},\beta^{\pm1}\rangle$ is quasi-Hamiltonian by Example \ref{Fus:MyCas2} (with $t=\alpha$, $s=\beta$, $\delta=1$, $\gamma=0$), with double quasi-Poisson  bracket  
\begin{equation}
  \begin{aligned}
    \dgal{\alpha,\alpha}=&\frac12( \alpha^2 \otimes 1 - 1 \otimes \alpha^2)\,, \quad 
   \dgal{\beta,\beta}=-\frac12(\beta^2 \otimes 1 - 1 \otimes \beta^2) \,, \\
     \dgal{\alpha,\beta}=&\,\,\,\frac12\, \left(\beta \alpha \otimes 1 + 1 \otimes \alpha \beta - \alpha \otimes \beta + \beta \otimes \alpha \right) \,  \,, \label{brPI1}
  \end{aligned} 
\end{equation}
and moment map $\Phi=[\alpha, \beta]$. By identification, we get a quasi-Hamiltonian algebra structure on $A=A_1$. 

We now prove the general case. We consider $g$ copies of the quasi-Hamiltonian algebra $A_1$ and $r$ copies of $A_0$, 
and we form $A_1 \oplus \ldots \oplus A_1 \oplus A_0 \oplus \ldots \oplus A_0$. By Remark \ref{RemOplus}, this is a quasi-Hamiltonian algebra. We denote the element $(0,\ldots,0,1,0,\ldots,0)$ with $1$ in $i$-th position as $e_i$, $1\leq i \leq g+r$.   By fusing $e_2$ onto $e_1$, then $e_3$ onto $e_1$ and so on, we get a quasi-Hamiltonian algebra structure by fusion on 
\begin{equation*}
  \kk \langle \alpha_i^{\pm 1},\beta_i^{\pm1},\Phi_i^{\pm1}, \gamma_k^{\pm1},\bar{\Phi}_k^{\pm1} \mid 1\leq i \leq g,\, 1 \leq k \leq r \rangle / ([\alpha_i,\beta_i]=\Phi_i, \gamma_k=\bar{\Phi}_k)\,,
\end{equation*}
where $\alpha_i,\beta_i,\Phi_i$ are the images of $\alpha,\beta,\Phi$ from the $i$-th copy of $A_1$, $1\leq i \leq g$, while $\gamma_k,\bar{\Phi}_k$ are the images of $\gamma,\Phi$ in the $k$-th copy of $A_0$. 
Rewriting the moment map  in the algebra obtained by fusion in terms of the $\Phi_i,\bar{\Phi}_k$ using Theorem \ref{ThmFusMomap}, then removing these unnecessary elements, we can rewrite the latter algebra as 
\begin{equation} \label{AlgPI}
  \kk \langle \alpha_i^{\pm 1},\beta_i^{\pm1}, \gamma_k^{\pm1},\Phi^{\pm1} \mid 1\leq i \leq g,\, 1 \leq k \leq r \rangle / (\prod^{\rightarrow}_{1\leq i\leq g}[\alpha_i,\beta_i]\prod^{\rightarrow}_{1 \leq k \leq r}\gamma_k=\Phi)\,.
\end{equation}
This is precisely $A$. The double quasi-Poisson bracket is then easily obtained from Theorem \ref{ThmFusBr}, Lemma \ref{dbrFUS}, and the ones on $A_0,A_1$. For example, fix $1<j\leq g$. After the step of fusion of $e_j$ onto $e_1$, any $\phi_i \in \{\alpha_i,\beta_i\}$ with $1\leq i <j$ is a generator of first type \eqref{type1} while $\phi_j \in \{\alpha_j,\beta_j\}$ is a generator of fourth type \eqref{type4}, so that $\dgal{\phi_i,\phi_j}$ gets a contribution given by \eqref{tw}. The fusion of $e_k$ onto $e_1$ with $k\neq j$ does not give any additional term in $\dgal{\phi_i,\phi_j}$, and we obtain \eqref{Eq:phiij}. 
\end{proof}

\begin{rem}
  To see that the double bracket from Theorem \ref{ThmPI} coincides with the one of Massuyeau-Turaev, note that the double brackets that do not involve the moment map are just those given in \cite[\S 8.3]{MT14}, while for the moment map they are given in \cite[\S 9.2]{MT14}. In particular, our construction is such that the moment map is the generator of the loop at the boundary component containing $\ast \in \partial \Sigma$. 

We should also note that our proof applies to the case of a weighted surface discussed in \cite[Section 10]{MT14}, i.e. when we fix $n_k \in \N^\times$, $1\leq k \leq r$, so that the generators $\alpha_i,\beta_i, \gamma_k$ (see \eqref{AlgPI}) satisfy the extra constraints $\gamma_k^{n_k}=1$ for $1\leq k \leq r$. Indeed, we can see that the ideal generated by $\gamma^n-1$ in $A_0$ is stable under the double bracket for any $n \in \N^\times$, so that we can start the proof with the algebras  $\kk[\gamma_k^{\pm1}]/(\gamma_k^{n_k}-1)$ instead of $r$ copies of $A_0$. 

Finally, remark that the way we are gluing components is the algebraic analogue of the \emph{boundary connected sum} discussed in \cite[Appendix B.2]{MT14}.
\end{rem}

\begin{rem}
It is an interesting problem to determine whether we can modify the definition of double quasi-Poisson bracket and keep a non-trivial fusion property as in Theorems \ref{ThmFusBr} and \ref{ThmFusMomap}. As a motivation, note that for $A=\kk \pi_1(\Sigma,\ast)$ the double quasi-Poisson bracket given in Theorem \ref{ThmPI} was introduced by Massuyeau-Turaev \cite{MT14} by (cyclically anti-)symmetrizing an operation $A^{\times 2} \to A^{\otimes 2}$ denoted by $\dgal{-,-}^\eta$. This means that for any $a,b \in A$, 
\begin{equation*}
  \dgal{a,b}=\dgal{a,b}^\eta + \frac12 1\otimes ab +\frac12 ba \otimes 1 -\frac12 a \otimes b -\frac12 b \otimes a\,,
\end{equation*}
see \cite[\S 7.2]{MT14} (recall that we rescale their double bracket by a factor $1/2$). 
 Since the couple $(A,\dgal{-,-})$ can be obtained by fusion, it would be interesting to see if there is an analogue proof for $(A,\dgal{-,-}^\eta)$ (note that $\dgal{-,-}^\eta$ is \emph{not} a double quasi-Poisson bracket as \eqref{Eq:cycanti} does not hold). An explicit form similar to Theorem \ref{ThmPI} of this particular operation can be found in \cite[Proposition 2.14]{AKKN}. For other uses of the operation $\dgal{-,-}^\eta$, see   \cite{AKKN18,AKKN} and references therein. 
 \end{rem}

\subsection{Morphisms of double quasi-Poisson algebras} Fix two double quasi-Poisson algebras $(A,\dgal{-,-})$ and $(A',\dgal{-,-}')$ over a $\kk$-algebra $B$. We say that a map $\psi:A \to A'$ is a \emph{morphism of double quasi-Poisson algebras (over $B$)} if $\psi$ is a morphism of $B$-algebras such that for any $a,b \in A$, 
\begin{equation} \label{Eq:MorQP}
  (\psi \otimes \psi)\,\dgal{a,b}=\dgal{\psi(a),\psi(b)}'\,.
\end{equation}
We say that it is an isomorphism of double quasi-Poisson algebras if $\psi$ is an isomorphism of $B$-algebras, which implies that the inverse $\psi^{-1}: A' \to A$ is also an isomorphism of double quasi-Poisson algebras. It seems natural to seek for isomorphisms between the different double quasi-Poisson algebra structures associated to quivers by Van den Bergh \cite{VdB1}, or the slight generalisation given by Theorem \ref{ThmVdB}. The same problem can be formulated for the double bracket of Massuyeau-Turaev \cite{MT14} given in Theorem \ref{ThmPI} if we change the presentation of the fundamental group by swapping factors\footnote{It was pointed out by an anonymous referee that this can be obtained from \cite{MT14} by invariance of the double bracket of Massuyeau-Turaev under self-homeomorphisms of the surface $\Sigma$ preserving $\ast$.} in \eqref{Eq:Surface}. In fact, these results easily follow from the next proposition, which is a non-commutative version of \cite[Proposition 5.7]{QuasiP}.      
\begin{prop}
  Assume that $(A,\dgal{-,-},\Phi)$ is a quasi-Hamiltonian algebra over $B=\oplus_s \kk e_s$. Consider the algebra $A^f_{1\leftarrow 2}$ obtained by fusing $e_2$ onto $e_1$ and the algebra $A^f_{1 \to 2}$ obtained by fusing $e_1$ onto $e_2$, which are both quasi-Hamiltonian algebras. Then there exists an isomorphism of double quasi-Poisson algebras $A^f_{1 \leftarrow 2}\to A^f_{1 \to 2}$ which preserves moment maps. 
\end{prop}
The proof of this statement is quite tedious, so we skip it and we will provide details in further work. 
 Let us simply mention that the isomorphisms between multiplicative preprojective algebras with different  orderings, which are given in the proof of \cite[Theorem 1.4]{CBShaw}, are precisely induced by this map.


\section{Elementary classification} \label{classif}

All our algebras are over a field $\kk$ of characteristic $0$ for convenience, but the discussion may be adapted to any integral domain (with unit) such that $2$ is invertible.  One could get rid of the latter localisation by rescaling the defining property \eqref{qPabc} as in \cite{MT14}.

\subsection{Polynomial ring in one variable} \label{classifk1}

We begin by classifying all double quasi-Poisson brackets on $A=\kk[t]$ over $B=\kk$. Our argument is similar to the classification of Powell \cite[Proposition A.1]{P16} in the case of a double Poisson bracket, i.e. when the associated triple bracket \eqref{Eq:TripBr} identically vanishes. We define a degree on $A$ by setting $|t|=1$, to get the decomposition $A=\oplus_{k\ge 0} \kk t^k$ in homogeneous components, which can clearly be extended to $A^{\otimes n}$ : an element $a_1 \otimes \ldots \otimes  a_n$ is homogeneous of degree $k$ if each $a_i$ is homogeneous in $A$ and $\sum_i |a_i|=k$. 

\begin{prop} \label{Pr:Free1}
  $A$ has a double bracket which is  quasi-Poisson if and only if it is  of the form 
\begin{equation}
  \dgal{t,t}=\lambda (t \otimes 1 - 1 \otimes t) + \mu (t^2 \otimes 1 - 1 \otimes t^2) + \nu (t^2 \otimes t - t \otimes t^2)\,, \label{qPtt}
\end{equation}
for $\lambda,\mu,\nu \in \kk$ with $4(\mu^2-\lambda \nu)=1$. 
\end{prop}
\begin{proof}
First, we remark that the quasi-Poisson property can be rewritten from \eqref{qPabc} as requiring 
\begin{equation}
  \label{qPttbis}
\dgal{t,t,t}=\frac14 (1+\tau_{(123)}+\tau_{(123)}^2)(1 \otimes t^2 \otimes t - 1 \otimes t \otimes t^2)\,.
\end{equation}

  Next, following \cite[Proposition A.1]{P16}, we split the double bracket as $\dgal{-,-}=\sum_{k=\min}^{\max} \dgal{-,-}^k$, where 
$\dgal{-,-}^k$ is its homogeneous component of degree $k$, i.e. $\dgal{t,t}^k\in \oplus_{l\leq k} \kk t^l \otimes \kk t^{k-l}$.  We then obtain that the decomposition of the triple bracket $\dgal{-,-,-}$ in homogeneous components has in highest degree the triple bracket defined by $\dgal{-,-}^{\max}$ of degree $2\max-1$. Since \eqref{qPttbis} is homogeneous of degree $3$, we need that the triple bracket associated to  $\dgal{-,-}^{\max}$ vanishes if $\max \geq 3$, that is we need $\dgal{-,-}^{\max}$ to be a double Poisson bracket. But \cite[Proposition A.1]{P16} gives that such a homogeneous double Poisson bracket is non-zero only if its degree is at most $3$. Moreover, if $\max=3$, this result also yields that it is a multiple of $\dgal{t,t}^{3}:=t^2 \otimes t - t \otimes t^2$. 

We have thus obtained that $\dgal{t,t}$ must be of the form \eqref{qPtt} for some $\lambda,\mu,\nu \in \kk$. The corresponding triple bracket is easily computed  (see e.g. \cite[Proposition A.1]{P16}) and gives 
\begin{equation}
\dgal{t,t,t}=(\mu^2-\lambda \nu) (1+\tau_{(123)}+\tau_{(123)}^2)(1 \otimes t^2 \otimes t - 1 \otimes t \otimes t^2)\,,
\end{equation}
so we can conclude by comparing this last expression with \eqref{qPttbis}. 
\end{proof}

\begin{lem}
Assume that $A=\kk[t]$ is endowed with a double quasi-Poisson bracket in the form \eqref{qPtt}, and set 
$\bar A=\kk[t]_{(t-\lambda)}$. Then $\bar A$ is a quasi-Hamiltonian algebra if and only if $\nu=0$.  
\end{lem}
\begin{proof}
First, remark that when $\nu=0$, we have by Proposition \ref{Pr:Free1} that $\mu=\frac{\delta}{2}$ for some $\delta\in \{\pm 1\}$, and  $\Phi=(t-\lambda)^{\delta}$ is a moment map. 

For the converse, we see $\bar A$ as the graded algebra $\kk[\bar{t}^{\pm 1}]$, where $\bar t = t-\lambda$ has degree $+1$. We also note that \eqref{qPtt} is equivalent to   
\begin{equation}
  \dgal{\bar{t},\bar{t}}=(\lambda+2\lambda \mu+\lambda^2\nu) (\bar{t} \otimes 1 - 1 \otimes \bar{t}) + (\mu+\lambda \nu) (\bar{t}^2 \otimes 1 - 1 \otimes \bar{t}^2) + \nu (\bar{t}^2 \otimes \bar{t} - \bar{t} \otimes \bar{t}^2)\,. \label{qPttBis}
\end{equation}
Since $\bar{A}$ is quasi-Hamiltonian, there exists  an (invertible) element $\Phi$  that satisfies 
\begin{equation} \label{PhimCt}
 \dgal{\Phi,\bar{t}}=\frac12 (\bar{t}\otimes \Phi-1 \otimes \Phi \bar{t} +  \bar{t} \Phi \otimes 1-\Phi \otimes \bar{t})\,,
\end{equation} 
and which we can decompose  as 
\begin{equation} \label{Phidec}
 \Phi=\sum_{k_0\leq l \leq k_1} c_l \bar t^l\,, \quad c_{k_0},c_{k_1} \in \kk^\times\,.
\end{equation}
Then, we get by looking at \eqref{PhimCt} in highest degree that $c_{k_1} \dgal{\bar{t}^{k_1},\bar{t}}$ is of degree at most $k_1+1$. But using the derivation property \eqref{Eq:inder}, this highest degree is exactly  $D+k_1-1$, where $D$ is the maximal degree of $\dgal{\bar t,\bar t}$ given in \eqref{qPttBis}. This implies that $D\leq 2$, i.e. there is no component of degree $3$ in $\dgal{\bar t,\bar t}$. We get from \eqref{qPttBis} that  $\nu=0$. 
\end{proof}


\subsection{Algebra with two idempotents}   \label{classifQ1}

In the previous case, the algebra $A$ was simply a $\kk$-algebra with no non-trivial (i.e. distinct from $0,1$) idempotent elements. The simplest case where such a decomposition occurs consists in taking the path algebra $\kk Q_1$ of the quiver $Q_1$ with vertices $\{1,2\}$ and unique arrow $t:1 \to 2$. (For conventions on quivers and path algebras, see \ref{sss:quivGen}.) 
We can see $\kk Q_1$ as a $B$-algebra with $B=\kk e_1 \oplus \kk e_2$, and if we assume that we have a $B$-linear double bracket on $\kk Q_1$, the derivation rules yield
\begin{equation*}
  \dgal{t,t}=\dgal{e_1 t e_2, e_1 t e_2}=e_1 \ast e_1\dgal{t,t}e_2 \ast e_2\,.
\end{equation*}
Using Sweedler's notation, this implies that $\dgal{t,t}'$ and $\dgal{t,t}''$ are  of the form $\alpha t$ for some $\alpha \in \kk$. Therefore $\dgal{t,t}=\alpha \, t \otimes t$, and the cyclic antisymmetry implies $\alpha=0$ so that  $\kk Q_1$ can only be endowed with the zero double bracket. At the same time, it is easy to see that $\dgal{t,t,t}$ given by  \eqref{qPabc} vanishes for $\kk Q_1$, so we get the next result. 
\begin{lem} \label{LemQ1}
 The zero double bracket is the unique double quasi-Poisson bracket on  $\kk Q_1$.
\end{lem}
As we have seen in \ref{classifk1}, the zero double bracket is \emph{not} quasi-Poisson on $\kk[t]$, and the fact that it is quasi-Poisson on $\kk Q_1$ is only due to the idempotent decomposition which implies $t^2=0$. In fact, if we consider $\kk[t]$ as the fusion algebra obtained by fusing $e_1$ and $e_2$ in $\kk Q_1$, the zero double quasi-Poisson bracket on $\kk Q_1$ yields after fusion the case $\lambda=\nu=0$ in Proposition \ref{Pr:Free1}.

\medskip

To get non-trivial examples of $B$-linear double brackets, we consider the double quiver $\bar{Q}_1$ obtained by adding to $Q_1$ the arrow $s=t^\ast:2\to 1$. 
If we define a degree on $A$ by setting $|s|=|t|=1$ and extend it to $A \otimes A$, we can characterise the  $B$-linear double quasi-Poisson brackets on $A$ that have degree at most $+4$ on generators. By the latter condition, we mean that 
$\dgal{s,s},\dgal{t,t}$ and $\dgal{t,s}$ (hence $\dgal{s,t}$) are sums of homogeneous terms of degree at most $+4$. 

\begin{prop} \label{Pr:Q1}
  Any  $B$-linear double quasi-Poisson bracket $\dgal{-,-}$ on $A=\kk\bar{Q}_1$ which has degree at most $+4$ on generators must be one of the following : 

\underline{\emph{Case 1:}} $\dgal{s,s}=0$, $\dgal{t,t}=0$ and one of the next two conditions holds 
\begin{subequations}
  \begin{align}
 &\text{\emph{1.a)}}\,\,  \,\, \dgal{t,s}=\frac{\delta}{2}(st \otimes e_1 - e_2 \otimes ts)\,, 
\quad \delta\in \{\pm1\} \,,  \label{Eq:Q1cqs11}\\
&\text{\emph{1.b)}}\,\,  \,\, \dgal{t,s}=
\gamma e_2\otimes e_1+\phi st \otimes ts + \alpha(st \otimes e_1 + e_2 \otimes ts)   \,, \quad
\alpha,\gamma,\phi\in \kk,\,\,\alpha^2=\frac14 + \gamma \phi\,; \label{Eq:Q1cqs12}
  \end{align}
\end{subequations}

\underline{\emph{Case 2:}} $\dgal{s,s}=0$, $\dgal{t,t}=\lambda (tst \otimes t - t \otimes tst)$ for $\lambda \in \kk^\times$ and 
\begin{equation*}
\dgal{t,s}= \frac{\delta}{2}(st \otimes e_1 - e_2 \otimes ts)\,, 
\quad \delta\in \{\pm1\} \,;
\end{equation*}

\underline{\emph{Case 3:}} $\dgal{t,t}=0$, $\dgal{s,s}=\lambda (sts \otimes s - s \otimes sts)$ for $\lambda \in \kk^\times$  and 
\begin{equation*}
\dgal{t,s}=  \frac{\delta}{2}(st \otimes e_1 - e_2 \otimes ts)\,, 
\quad \delta\in \{\pm1\} \,.
\end{equation*}
\end{prop}

The proof is given in Appendix \ref{App:PfQ1}.

\begin{exmp} \label{ExpFusQ1}
The simplest double quasi-Poisson brackets that can be obtained from Case 1 are 
\begin{equation} \label{Eq:FusQ1}
  \dgal{t,t}=0\,, \quad \dgal{s,s}=0\,, \quad 
\dgal{t,s}=\frac{\delta}{2}st \otimes e_1 + \frac{\delta'}{2}e_2 \otimes ts\,, \quad \delta,\delta'\in \{\pm1\}\,.
\end{equation}
These double brackets are all obtained by fusion. Indeed, consider the quiver $Q_1$ with vertices $\{1,2\}$ and unique arrow $t:1\to 2$, and the quiver $Q_1'$ with vertices $\{3,4\}$ and unique arrow $s:4 \to 3$. Their path algebras have a double quasi-Poisson bracket which is the zero one by Lemma \ref{LemQ1}. Thus, the zero double bracket on the path algebra $A$ of the quiver $Q_1 \sqcup Q_1'$ is also quasi-Poisson by Remark \ref{RemOplus}. We can see $A$ as an algebra over $B=\oplus_{s=1}^4 \kk e_s$, where $e_s$ is the elementary path corresponding to the $s$-th vertex. We can glue the vertices $1$ and $3$, as well as the vertices $2$ and $4$. The resulting fusion algebra is just $\kk\bar{Q}_1$, and we have a double quasi-Poisson bracket by Theorem \ref{ThmFusBr} given by \eqref{Eq:FusQ1}, 
where $\delta=+1$ (resp. $\delta=-1$) if we fuse $e_3$ onto $e_1$ (resp. $e_1$ onto $e_3$), and where 
$\delta'=+1$ (resp. $\delta'=-1$) if we fuse $e_4$ onto $e_2$ (resp. $e_2$ onto $e_4$). 
\end{exmp}

\begin{exmp} \label{ExpVdBQ1}
Up to localisation, we claim that the algebra $A$ with double quasi-Poisson bracket given by Case 1 with \eqref{Eq:Q1cqs12} is quasi-Hamiltonian when $\gamma\phi=0$. In such a case, we set $\alpha=\frac{\delta}{2}$ for some $\delta=\pm 1$. 

If $\phi=0$, consider  the localisation of $A$ at $\delta \gamma+st$ and $\delta\gamma+ts$. This is equivalent to require that the element $\delta\gamma e_1+ts$ is invertible in $e_1 A e_1$, while $\delta\gamma e_2 + st$ is invertible in $e_2 A e_2$. We can easily check that $\Phi_1=(\delta\gamma e_1+ts)^{\delta}$ and $\Phi_2=(\delta\gamma e_2+st)^{-\delta}$ satisfy \eqref{Phim}. Hence $\Phi=\Phi_1+\Phi_2$ is a moment map in the localised algebra. 

If $\gamma=0$, we require that $ts$ (resp. $st$) is invertible in $e_1 A e_1$ (resp. $e_2 A e_2$) with local inverse $(ts)^{-1}$ (resp. $(st)^{-1}$). We then further require that we have local inverses for $\phi e_1+(ts)^{-1}$ and $\phi e_2+(st)^{-1}$. As a result, we can check that $\Phi=\Phi_1+\Phi_2$ is a moment map for $\Phi_1=(\delta\phi e_1+(ts)^{-1})^{-\delta}$ and $\Phi_2=(\delta\phi e_2+(st)^{-1})^{\delta}$. 

When $\gamma=\phi=0$, both constructions give the same quasi-Hamiltonian algebra. 
\end{exmp}
\begin{rem}
For $\phi=0$ and $\gamma=\delta=+1$  in Example \ref{ExpVdBQ1}, this corresponds to Van den Bergh's key example of quasi-Hamiltonian algebra associated to the double of the quiver $1 \to 2$ given in \cite[\S 6.5]{VdB1} (see Theorem \ref{ThmVdB}).  
\end{rem}

\subsection{Free algebra on two generators} \label{classifF2}

Consider $A=\kk\langle s,t \rangle$ with $B=\kk$. To obtain new examples of double quasi-Poisson brackets on $A$, we assume that we have a double bracket such that 
\begin{subequations}
  \begin{align}
    \dgal{t,t}=&\lambda (t \otimes 1 - 1 \otimes t) + \mu (t^2 \otimes 1 - 1 \otimes t^2) + \nu (t^2 \otimes t - t \otimes t^2)\,, \label{Eqtt} \\
   \dgal{s,s}=&l (s \otimes 1 - 1 \otimes s) + m (s^2 \otimes 1 - 1 \otimes s^2) + n (s^2 \otimes s - s \otimes s^2) \,, \label{Eqss}
  \end{align}
\end{subequations}
with coefficients in $\kk$ that satisfy $4(\mu^2-\lambda \nu)=1$ and  $4(m^2- ln)=1$. 
Furthermore, we consider that the double bracket between $s$ and $t$ has the form 
\begin{equation}
  \begin{aligned} \label{EqtsA}
    \dgal{t,s}=&\,\,\,\alpha_0\, t^2 \otimes 1 + \alpha_0'\, 1 \otimes t^2 + \beta_0\, s^2 \otimes 1 + \beta_0'\, 1 \otimes s^2 + \gamma_0 \, t \otimes t + \gamma_1 \, s \otimes s \\
&+ \alpha_1\, ts \otimes 1 + \alpha_1'\, st \otimes 1 + \alpha_2 \, t \otimes s + \alpha_2' \, s \otimes t + \alpha_3 \, 1 \otimes ts + \alpha_3' \, 1 \otimes st \\
&+ \beta_1 \, t \otimes 1 + \beta_1' \, 1 \otimes t + \beta_2 \, s \otimes 1 + \beta_2' \, 1 \otimes s + \gamma \, 1 \otimes 1\,,
  \end{aligned}
\end{equation}
with all coefficients in $\kk$. 
In other words, if we fix a degree on $A$ by $|t|=|s|=1$ and extend it to $A\otimes A$, we assume that the double bracket $\dgal{t,s}$ has degree at most $+2$. 
We wish to formulate a classification of the double quasi-Poisson brackets of the above form. To do so, introduce the conditions 
\begin{equation*}
\begin{aligned}
  \text{(C1)}&\quad \lambda=\nu=0,\, \mu =\pm\frac12\,, \quad \text{(C1')}\quad \mu=0,\, \nu=\frac{-1}{4\lambda}\in \kk^\times\,,  \\
\text{(C2)}&\quad l=n=0,\, m =\pm\frac12\,, \quad \text{(C2')}\quad m=0,\, n=\frac{-1}{4l}\in \kk^\times\,.
\end{aligned}
\end{equation*}
We say that a double bracket $\dgal{-,-}$ on $A$ of the form \eqref{Eqtt}--\eqref{Eqss} and \eqref{EqtsA} is  \emph{reduced} if it satisfies either (C1) or (C1'), together with either (C2) or (C2'). 
It is not difficult to see that, up to an affine change of variables $t\mapsto t+\rho_t$, $s \mapsto s+\rho_s$, for suitable $\rho_t,\rho_s \in \kk$, any double bracket $\dgal{-,-}$ on $A$ of the form \eqref{Eqtt}--\eqref{Eqss} and \eqref{EqtsA} can be put into reduced form.

\begin{prop} \label{Pr:Free2}
  Any double bracket $\dgal{-,-}$ on $A$ of the form \eqref{Eqtt}--\eqref{Eqss} and \eqref{EqtsA} which is quasi-Poisson is isomorphic to one of the following reduced double quasi-Poisson brackets : 

\vspace{0.2cm}

\underline{\emph{Case 1:}} For any $\gamma_0,\gamma_1\in \kk$, $\mu=\pm\frac12$, $\alpha\in \kk$ such that $\alpha^2=\frac14 + \gamma_0\gamma_1$,
\begin{equation}  
  \begin{aligned} 
&\dgal{t,t}=\mu(t^2 \otimes 1 - 1 \otimes t^2)\,, \quad 
\dgal{s,s}=\mu(s^2 \otimes 1 - 1 \otimes s^2)\,, \\
&\dgal{t,s}=\gamma_0 t\otimes t + \gamma_1 s \otimes s + \mu (st \otimes 1 - 1 \otimes ts) + \alpha (t \otimes s + s \otimes t)\,,     \label{MyCas1} 
  \end{aligned}  
\end{equation}

\underline{\emph{Case 2:}} For any $\gamma\in \kk$, $\alpha,\mu=\pm\frac12$, 
\begin{equation}
  \begin{aligned} 
&\dgal{t,t}=\mu(t^2 \otimes 1 - 1 \otimes t^2)\,, \quad 
\dgal{s,s}=-\mu(s^2 \otimes 1 - 1 \otimes s^2)\,, \\
&\dgal{t,s}=\alpha (st \otimes 1 + 1 \otimes ts) + \mu ( s \otimes t - t \otimes s ) + \gamma 1\otimes 1\,,      \label{MyCas5}
  \end{aligned}
\end{equation}

\underline{\emph{Case 3:}} For any $m,\mu=\pm\frac12$, 
\begin{equation}
  \begin{aligned}
&\dgal{t,t}=\mu(t^2 \otimes 1 - 1 \otimes t^2)\,, \quad 
\dgal{s,s}=m(s^2 \otimes 1 - 1 \otimes s^2)\,, \\
&\dgal{t,s}= \mu (st \otimes 1 -t \otimes s + s \otimes t - 1 \otimes ts) \,,    \label{MyCas4}
  \end{aligned}
\end{equation}

\underline{\emph{Case 4:}} For any $\alpha,m,\mu=\pm\frac12$, 
\begin{equation}
  \begin{aligned}
&\dgal{t,t}=\mu(t^2 \otimes 1 - 1 \otimes t^2)\,, \quad 
\dgal{s,s}=m(s^2 \otimes 1 - 1 \otimes s^2)\,, \\
&\dgal{t,s}= \alpha (st \otimes 1 -t \otimes s - s \otimes t + 1 \otimes ts) \,,    \label{MyCas6}
  \end{aligned}
\end{equation}

\underline{\emph{Case 5:}} For any $n \in \kk^\times$, $\alpha,\mu=\pm\frac12$,
\begin{equation}
  \begin{aligned}
&\dgal{t,t}=\mu(t^2 \otimes 1 - 1 \otimes t^2)\,, \quad 
\dgal{s,s}=\frac{-1}{4n}(s\otimes 1 - 1 \otimes s) + n (s^2 \otimes s - s \otimes s^2)\,, \\
&\dgal{t,s}= \alpha (st \otimes 1 -t \otimes s - s \otimes t + 1 \otimes ts) \,,    \label{MyCas17}
  \end{aligned}
\end{equation}

\underline{\emph{Case 6:}} For any $n \in \kk^\times$, $\mu=\pm\frac12$, 
\begin{equation}
  \begin{aligned}
&\dgal{t,t}=\mu(t^2 \otimes 1 - 1 \otimes t^2)\,, \quad 
\dgal{s,s}=\frac{-1}{4n}(s\otimes 1 - 1 \otimes s) + n (s^2 \otimes s - s \otimes s^2)\,, \\
&\dgal{t,s}= \mu (st \otimes 1 -t \otimes s + s \otimes t - 1 \otimes ts) \,,    \label{MyCas16}
  \end{aligned}
\end{equation}

\underline{\emph{Case 7:}} For any $n,\nu \in \kk^\times$, $\alpha=\pm\frac12$, 
\begin{equation}
  \begin{aligned}
&\dgal{t,t}=\frac{-1}{4\nu}(t \otimes 1 - 1 \otimes t) + \nu (t^2 \otimes t - t \otimes t^2)\,, \\ 
&\dgal{s,s}=\frac{-1}{4n}(s\otimes 1 - 1 \otimes s) + n (s^2 \otimes s - s \otimes s^2)\,, \\
&\dgal{t,s}= \alpha (st \otimes 1 -t \otimes s - s \otimes t + 1 \otimes ts) \,,    \label{MyCas18}
  \end{aligned}
\end{equation}
\end{prop}

\begin{rem}
  Under the automorphism of $A$ given by $s\mapsto t$, $t \mapsto s$, the cases given by \eqref{MyCas1}, \eqref{MyCas5}, \eqref{MyCas6} and \eqref{MyCas18} are invariant; we obtain from the other cases \eqref{MyCas4}, \eqref{MyCas17} and \eqref{MyCas16} three additional cases that do not appear in Proposition \ref{Pr:Free2}. In particular, this explains why there is no other occurrence of the case $\nu \neq0$ than in \eqref{MyCas18}. 
\end{rem}

The proof of Proposition \ref{Pr:Free2} is quite tedious and not interesting, so we skip it until Appendix \ref{App:F2}. The idea is to realise that the two conditions 
\begin{equation*}
\dgal{t,t,t}=\frac14 (1+\tau_{(123)}+\tau_{(123)}^2)(1 \otimes t^2 \otimes t - 1 \otimes t \otimes t^2)\,, \quad 
\dgal{s,s,s}=\frac14 (1+\tau_{(123)}+\tau_{(123)}^2)(1 \otimes s^2 \otimes s - 1 \otimes s \otimes s^2)\,,
\end{equation*}
obtained from \eqref{qPabc} are trivially satisfied by Proposition \ref{Pr:Free1} since we require $4(\mu^2-\lambda \nu)=1$ and  $4(m^2- ln)=1$. Using that a triple bracket is cyclically antisymmetric and is completely determined by its value on generators, it remains to check for which coefficients we have the equalities 
\begin{equation} \label{Eqtts14}
\begin{aligned}
  \dgal{t,t,s}=& \frac14 \Big(
s t \otimes t \otimes   1 - s t \otimes 1  \otimes  t - s  \otimes  t^2 \otimes 1  + s \otimes  t \otimes  t  \\
& \qquad  -t \otimes t \otimes s +  t \otimes 1  \otimes  t s + 1 \otimes  t^2 \otimes s - 1  \otimes t \otimes  t s \Big)\,,
\end{aligned}
\end{equation}
\begin{equation} \label{Eqsst14}
\begin{aligned}
  \dgal{s,s,t}=& \frac14 \Big(
t s \otimes s \otimes   1 - t s \otimes 1  \otimes  s - t  \otimes  s^2 \otimes 1  + t \otimes  s \otimes  s  \\
& \qquad  -s\otimes s \otimes t +  s\otimes 1  \otimes  s t + 1 \otimes  s^2 \otimes t - 1  \otimes s \otimes  s t \Big)\,,
\end{aligned}
\end{equation}
also obtained from \eqref{qPabc}.

\subsubsection{Fusion for Proposition \ref{Pr:Free2}} \label{ss:F2fusion}

We can use Theorem \ref{ThmFusBr} to obtain the following result. 

\begin{thm} \label{ThmF2fus}
 Up to localisation, any double quasi-Poisson bracket on $A$ of the form \eqref{Eqtt}--\eqref{Eqss} and \eqref{EqtsA} is isomorphic to a reduced double quasi-Poisson bracket obtained by fusion. 
\end{thm}
The proof follows by combining the different examples that we give now together with Proposition \ref{Pr:Free2}. 

\begin{exmp} \label{Fus:MyCas1} \emph{(Fusion for Case 1.)} 
 For any $\alpha,\gamma_0,\gamma_1\in \kk$ such that $\alpha^2=\frac14 + \gamma_0\gamma_1$, we can consider $\kk \bar{Q}_1$ with the double quasi-Poisson bracket given by \eqref{Eq:Q1cqs12} in Proposition \ref{Pr:Q1} with $\gamma=\gamma_1$, $\phi=\gamma_0$. We form the algebra $A$ by locally inverting $ts= e_1 ts e_1$ and $st= e_2 st e_2$. We can introduce $\bar{s}=(ts)^{-1}t=t(st)^{-1}\in e_1 A e_2$. The double quasi-Poisson bracket descends to $A$ in such a way that 
 \begin{equation*}
  \dgal{t,t}=0=\dgal{\bar s,\bar s}\,, \quad \dgal{t,\bar{s}}=\gamma_0 t\otimes t +\gamma_1 \bar{s} \otimes \bar{s} + \alpha(t \otimes \bar{s} + \bar{s} \otimes t). 
 \end{equation*}
Fusing $e_1$ and $e_2$, we get the fusion algebra $A^f=\kk\langle t^{\pm 1},s^{\pm1}\rangle$ with double quasi-Poisson bracket given by \eqref{MyCas1}, where $\mu=+\frac12$ (resp. $\mu=-\frac12$) if we fuse $e_2$ onto $e_1$ (resp. $e_1$ onto $e_2$) by using \eqref{vv} (resp. \eqref{uu}). 
\end{exmp}

\begin{exmp} \label{Fus:MyCas2} \emph{(Fusion for Case 2.)} 
For any $\gamma \in \kk$ and $\delta=\pm1$, the localisation $A$ of the path algebra $\kk \bar{Q}_1$ at $a=\delta \gamma+ts$ and $b=\delta\gamma+st$ is a quasi-Hamiltonian $B$-algebra for $B=\kk e_1 \oplus \kk e_2$ by Example \ref{ExpVdBQ1} (with $\phi=0$). 
The fusion algebra $A^f$ obtained by fusing $e_2$ onto $e_1$ can be identified with $\kk\langle s,t \rangle_{a,b}$.  It is a quasi-Hamiltonian algebra  with double quasi-Poisson bracket 
\begin{equation} \label{Eq:FusMyCas2}
\begin{aligned}
     \dgal{t,t}=&\frac12(t^2 \otimes 1 - 1 \otimes t^2)\,, \quad 
\dgal{s,s}=\frac12(1 \otimes s^2 - s^2 \otimes 1)\,, \\
\dgal{t,s}=&\gamma 1 \otimes 1 + \frac{\delta}{2}(st \otimes 1 + 1 \otimes ts) + \frac12 (s \otimes t - t \otimes s)\,,
\end{aligned}
\end{equation}
using successively \eqref{vv}, \eqref{uu} and \eqref{vu}. The moment map $\Phi=a^\delta b^{- \delta}$ is obtained by Theorem \ref{ThmFusMomap}. 
If we fuse $e_1$ onto $e_2$ instead, the factors $\frac12$ appearing in \eqref{Eq:FusMyCas2} are replaced by $-\frac12$ and the moment map is now $\Phi=b^{- \delta}a^\delta$. 
\end{exmp}
  
\begin{rem} After fusion, the case $\gamma=\delta=+1$ treated in Example \ref{Fus:MyCas2} corresponds to Van den Bergh's quasi-Hamiltonian algebra associated to a one-loop quiver \cite{VdB1} (see Theorem \ref{ThmVdB}). The case $\gamma=0$ appears  after localisation on $A'=\kk \langle s^{\pm1},t^{\pm1}\rangle$ in \cite{CF}, and gives the quasi-Hamiltonian structure for the fundamental group of a torus with one marked boundary component \cite{MT14}  (see Theorem \ref{ThmPI}).
\end{rem}

\begin{exmp} \label{Fus:MyCas4-16} \emph{(Fusion for Cases 3,6.)} 
We consider the algebra $\kk\langle s\rangle$ with double quasi-Poisson bracket \eqref{Eqss}, and $\kk Q_1$ for the quiver $Q_1$ given by $t:1 \to 2$ endowed with the zero double quasi-Poisson bracket. Consider the direct sum $A=\kk Q_1 \oplus \kk\langle s\rangle$, where we denote the identity of $\kk\langle s\rangle$ as $e_3$. This is a double quasi-Poisson algebra by Remark \ref{RemOplus}. 

If we fuse $e_3$ onto $e_2$ (resp. $e_2$ onto $e_3$) and call it $e_2$, we obtain the fusion algebra $A'$ with double quasi-Poisson bracket \eqref{Eqss}, $\dgal{t,t}=0$  and 
\begin{equation*}
 \dgal{t,s}=\alpha (e_2 \otimes ts - s \otimes t)\,, \quad \alpha=+\frac12 \text{ (resp. }\alpha=-\frac12\text{).}
\end{equation*}
Then, if we fuse $e_2$ onto $e_1$ (resp. $e_1$ onto $e_2$) which becomes the unit in the fusion algebra $A''$, we have a double quasi-Poisson bracket 
given by \eqref{Eqss} and 
\begin{equation*}
\dgal{t,t}=\mu (t^2 \otimes 1 - 1 \otimes t^2)\,, \quad 
 \dgal{t,s}=\alpha (1 \otimes ts - s \otimes t)+\mu (st \otimes 1 - t \otimes s)\,, 
\end{equation*}
where $\mu=\frac12$ (resp. $\mu=-\frac12$). When $\alpha=-\mu$, we get \eqref{MyCas4} if $n=l=0$, or we get \eqref{MyCas16} if $m=0$.  
\end{exmp}

\begin{exmp} \label{Fus:MyCas6-17-18} \emph{(Fusion for Cases 4,5,7.)} 
We consider the algebras $\kk\langle t\rangle$ and $\kk\langle s\rangle$ with double quasi-Poisson brackets \eqref{Eqtt}--\eqref{Eqss}. 
Then $A=\kk\langle t\rangle\oplus \kk\langle s\rangle$ is a double quasi-Poisson algebra by Remark \ref{RemOplus}, and we denote $e_1=(1,0)$, $e_2=(0,1)$. If we fuse $e_2$ onto $e_1$ (resp. $e_1$ onto $e_2$) which is the unit in the fusion algebra $A'$, we get a double quasi-Poisson bracket given by \eqref{Eqtt}--\eqref{Eqss} and 
\begin{equation*}
 \dgal{t,s}=\alpha(st \otimes 1 + 1 \otimes ts - s \otimes t - t \otimes s)\,,
\end{equation*}
with $\alpha=+\frac12$ (resp. $\alpha=-\frac12$). For $n=l=\nu=\lambda=0$ we get \eqref{MyCas6}, for $m=\nu=\lambda=0$ we get \eqref{MyCas17}, while for $m=\mu=0$ we get \eqref{MyCas18}.
\end{exmp}

\section{Representations spaces and (quasi-)Poisson algebras} \label{ss:qHscheme}

\subsection{Generalities on representation spaces}  \label{sss:RepSpaces}

We assume that $A$ is a finitely generated associative algebra over $B=\oplus_{s=1}^K\kk e_s$, with $e_s e_t = \delta_{st} e_s$. 
Following \cite[Section 7]{VdB1} (see also \cite[Section 4]{CB11} and \cite[Section 3]{MT14}), let $I=\{1,\ldots,K\}$ and choose a dimension vector $\alpha \in \N^I$, setting $N=\sum_{s\in I}\alpha_s$. We consider the representation space (relative to $B$) $\Rep(A,\alpha)$. The representation space is the affine scheme whose  coordinate ring $\OO(\Rep(A,\alpha))$ is generated by symbols $a_{ij}$ for $a\in A$, $1\leq i,j \leq N$, which satisfy 
\begin{equation*}
  (a+b)_{ij}=a_{ij}+b_{ij}\,, \quad (ab)_{ij}=\sum_{k=1}^N a_{ik}b_{kj}\,, 
\end{equation*}
together with the condition that for any $1 \leq s \leq K$ the matrix $\X(e_s)=((e_s)_{ij})_{ij}$ is the $s$-th diagonal identity block of size $\alpha_s$. In other words, we have that $(e_s)_{ij}=\delta_{ij}$ if $\alpha_1 + \ldots + \alpha_{s-1}+1\leq i,j\leq \alpha_1 + \ldots + \alpha_s$, while it is zero otherwise. Note that this implies  $1_{ij}=\delta_{ij}$ for all $1\leq i,j \leq N$. 
To ease notations, denote by $R=\OO(\Rep(A,\alpha))$ the coordinate ring, and for any $a\in A$ set $\X(a)$ to denote the matrix with entries $a_{ij}\in R$.

By definition of $\Rep(A,\alpha)$, any element $a\in A$ induces functions $(a_{ij})_{ij}$ on $\Rep(A,\alpha)$, and we would like to extend this definition to derivations. We associate to any $\delta \in D_{A/B}$ the vector fields  $\delta_{ij}\in \Der(R)$, $1\leq i,j \leq N$, defined by  
\begin{equation} \label{Eq:Standard}
 \delta_{ij}(b_{kl})=\delta(b)'_{kj}\delta(b)''_{il}\,, 
\end{equation}
and introduce the vector field-valued matrix $\X(\delta)$ with $(i,j)$ entry $\delta_{ij}$. 
We call the particular disposition of indices in \eqref{Eq:Standard} the \emph{standard index notation} as in \cite{VdB2}. More generally, for an element $\delta=\delta_1\ldots \delta_n \in (D_BA)_n$ we define  $\delta_{ij}\in \bigwedge_R^n \Der(R)$ from the matrix identity $\X(\delta)=\X(\delta_1)\ldots \X(\delta_n)$, and we set $\tr \X(\delta)=\sum_i \delta_{ii}$. 
 
\begin{prop} \emph{(\cite[Propositions 7.5.1,7.5.2]{VdB1})} \label{PropBr}
Assume that $\dgal{-,-}$ is a $B$-linear double bracket defined on $A$. Then there is a unique antisymmetric biderivation $\br{-,-}$ on $R$ such that 
\begin{equation} \label{EqBrDefn}
  \br{a_{ij},b_{kl}}=\dgal{a,b}'_{kj}\dgal{a,b}''_{il}\,,
\end{equation}
for any $a,b \in A$. Moreover, for any $a,b,c\in A$, 
\begin{equation} \label{JacVdB}
  Jac(a_{ij},b_{kl},c_{uv})=\dgal{a,b,c}_{uj,il,kv}- \dgal{a,c,b}_{kj,iv,ul}\,,
\end{equation}
where, on the left-hand side, $Jac:R^{\times 3}\to R$ is defined by 
\begin{equation*}
  Jac(g_1,g_2,g_3)=\br{g_1,\br{g_2,g_3}}\,+\, \br{g_2,\br{g_3,g_1}}\,+\, \br{g_3,\br{g_1,g_2}}\,, \quad g_1,g_2,g_3 \in R\,,
\end{equation*}
while on the right-hand side $\dgal{-,-,-}$ is the triple bracket \eqref{Eq:TripBr} defined by $\dgal{-,-}$, and we write for $a=a'\otimes a'' \otimes a'''\in A^{\otimes 3}$ that $a_{ij,kl,uv}=a'_{ij}a''_{kl}a'''_{uv}$. 
\end{prop} 

We now remark the following result,  which will be important in \ref{ss:qPA}. 
\begin{lem} \label{LemUseful}
  Assume that $Q \in (D_BA)_n$, and denote by $\dgal{-,\ldots,-}$ the corresponding differential $n$-bracket given by Proposition \ref{Prop:BrQ}. For any $a=a^1\otimes \ldots \otimes a^n \in A^{\otimes n}$, introduce 
\begin{equation*}
  a_{(u_1v_1,\ldots,u_n v_n)}=a^1_{u_1 v_1} \ldots a^n_{u_n v_n} \,\,\in R\,, 
\end{equation*}
with indices in the set $\{1,\ldots,N\}$. 
Consider the natural action of $S_{n}$ on $\{1,\ldots,n\}$ and the action of $S_{n-1}$ on $\{2,\ldots,n\}$ obtained by fixing the element $1$. Then the following holds 
\begin{equation} \label{relQRep}
  \tr \X(Q)(a^1_{u_1 v_1}, \ldots, a^n_{u_n v_n})
=\sum_{\tilde{\sigma} \in S_{n-1}}\epsilon(\tilde{\sigma})\dgal{a^{1},a^{\tilde{\sigma}(2)},\ldots,a^{\tilde{\sigma}(n)}}_{\tilde{\sigma}(u,v)}
\end{equation}
where $\tilde{\sigma}(u,v):=(u_{\tilde{\sigma}(n)}v_1,u_1 v_{\tilde{\sigma}(2)},\ldots,u_{\tilde{\sigma}(n-1)}v_{\tilde{\sigma}(n)})$, while 
 $\epsilon(\tilde{\sigma})=+1$ if $\tilde{\sigma}$ is an even permutation, and $\epsilon(\tilde{\sigma})=-1$ if $\tilde{\sigma}$ is an odd  permutation. 
\end{lem}
\begin{proof}
By linearity, we can just assume that $Q=\delta^1\ldots\delta^n$ with each $\delta^i\in D_{A/B}$. 
We can write 
\begin{equation*}
  \begin{aligned}
  & \tr \X(Q)(a^1_{u_1 v_1}, \ldots, a^n_{u_n v_n})
= \sum_{i_1,\ldots,i_n}  (\delta^1_{i_1 i_2} \wedge \ldots \wedge \delta^n_{i_n i_1})(a^1_{u_1 v_1}, \ldots, a^n_{u_n v_n}) \\
=& \sum_{i_1,\ldots,i_n} \sum_{\sigma \in S_n} \epsilon(\sigma) \,
\delta^1_{i_1 i_2}(a^{\sigma(1)}_{u_{\sigma(1)}v_{\sigma(1)}}) \ldots 
\delta^q_{i_q i_{q+1}}(a^{\sigma(q)}_{u_{\sigma(q)}v_{\sigma(q)}}) \ldots
\delta^n_{i_n i_1}(a^{\sigma(n)}_{u_{\sigma(n)}v_{\sigma(n)}})\,.
  \end{aligned}
\end{equation*}
Using \eqref{Eq:Standard} and summing over all $i_q$, we get that this equals 
\begin{equation*}
  \begin{aligned}
 &\sum_{\sigma \in S_n}\epsilon(\sigma)\,
(\delta^1(a^{\sigma(1)})' \delta^2(a^{\sigma(2)})'')_{u_{\sigma(1)} v_{\sigma(2)}} \ldots 
(\delta^q(a^{\sigma(q)})' \delta^{q+1}(a^{\sigma(q+1)})'')_{u_{\sigma(q)} v_{\sigma(q+1)}} \ldots 
(\delta^n(a^{\sigma(n)})' \delta^1(a^{\sigma(1)})'')_{u_{\sigma(n)} v_{\sigma(1)}}
\\
&=\sum_{\sigma \in S_n} \epsilon(\sigma) \,\left(
\delta^n(a^{\sigma(n)})' \delta^1(a^{\sigma(1)})'' \otimes 
\delta^1(a^{\sigma(1)})' \delta^2(a^{\sigma(2)})''\otimes  \ldots \otimes
\delta^{n-1}(a^{\sigma(n-1)})' \delta^1(a^{\sigma(1)})''\right)_{\sigma(u,v)} \,,
  \end{aligned}
\end{equation*}
where $\sigma(u,v)=(u_{\sigma(n)} v_{\sigma(1)},u_{\sigma(1)}v_{\sigma(2)},\ldots,u_{\sigma(n-1)}v_{\sigma(n)})$.

Next, remark that we can identify any $\sigma \in S_n$ with $\tilde{\sigma}\tau^i$, where $\tau=(1 \ldots n)$, $i\in \{0, \ldots, n-1\}$,   and $\tilde{\sigma}\in S_{n-1}$ acts on $\{2,\ldots,n\}$. Given $\sigma$, the pair $(i,\tilde\sigma)$ is unique and satisfies $\epsilon(\sigma)=(n-1)i+\epsilon(\tilde \sigma)$. Moreover, the action of $\sigma \in S_n$ on $A^{\otimes n}$ decomposes into the permutation $\tau^i$ of the factors and the action of $\tilde{\sigma }\in S_{n-1}$ fixing the first copy in the tensor product. Therefore, we can write 
$\tr \X(Q)(a^1_{u_1 v_1}, \ldots, a^n_{u_n v_n}) $ as follows 
\begin{equation} \label{Useful1}
\sum_{\tilde \sigma \in S_{n-1}}\epsilon(\tilde \sigma)
\sum_{i=0}^{n-1}(-1)^{(n-1)i}\left(
\delta^n(a^{\sigma(n)})' \delta^1(a^{\sigma(1)})'' \otimes \ldots \otimes
\delta^{n-1}(a^{\sigma(n-1)})' \delta^1(a^{\sigma(1)})''\right)_{\sigma(u,v)}\,,
\end{equation}
where $\sigma(u,v)=(u_{\sigma(n)} v_{\sigma(1)},u_{\sigma(1)}v_{\sigma(2)},\ldots,u_{\sigma(n-1)}v_{\sigma(n)})$ and we put $\sigma=\tilde{\sigma}\tau^i$. 

Meanwhile, remark that we can get from Proposition \ref{Prop:BrQ} 
\begin{equation*}
\begin{aligned}
    \dgal{b^1,\ldots,b^n}=\sum_{i=0}^{n-1} (-1)^{(n-1)i} \,\, &
\delta^{\tau^{-i}(n)}(b^n)'\delta^{\tau^{-i}(1)}(b^1)'' \otimes \ldots \otimes 
\delta^{\tau^{-i}(q)}(b^q)'\delta^{\tau^{-i}(q+1)}(b^{q+1})'' \otimes \ldots \\
& \qquad \qquad \ldots \otimes 
\delta^{\tau^{-i}(n-1)}(b^{n-1})'\delta^{\tau^{-i}(n)}(b^n)''  \,.
\end{aligned}
\end{equation*}
If we extend the action of $S_{n-1}$ on $\{2,\ldots,n\}$ to $\{1,\ldots,n\}$ by setting $\tilde \sigma(1)=1$, we find that  
\begin{equation*}
\begin{aligned}
&\sum_{\tilde{\sigma} \in S_{n-1}}\epsilon(\tilde{\sigma})\dgal{a^{\tilde{\sigma}(1)},a^{\tilde{\sigma}(2)},\ldots,a^{\tilde{\sigma}(n)}}_{\tilde{\sigma}(u,v)} \\
=&\sum_{\tilde{\sigma} \in S_{n-1}}\epsilon(\tilde{\sigma})
\sum_{i=0}^{n-1} (-1)^{(n-1)i} \,\, 
\left(\delta^{\tau^{-i}(n)}( a^{\tilde{\sigma}(n)} )'\delta^{\tau^{-i}(1)}(a^{\tilde{\sigma}(1)})'' \otimes \ldots \otimes 
\delta^{\tau^{-i}(n-1)}(a^{\tilde{\sigma}(n-1)})'\delta^{\tau^{-i}(n)}(a^{\tilde{\sigma}(n)})''\right)_{\tilde{\sigma}(u,v)}  \,.
\end{aligned}
\end{equation*}
where $\tilde{\sigma}(u,v):=(u_{\tilde{\sigma}(n)}v_{\tilde{\sigma}(1)},u_{\tilde{\sigma}(1)} v_{\tilde{\sigma}(2)},\ldots,u_{\tilde{\sigma}(n-1)}v_{\tilde{\sigma}(n)})$. Now, we remark that if we simultaneously apply $\tau^i$ on the tensor product and on the indices $\tilde{\sigma}(u,v)$, then each term on the right-hand side is unchanged. But doing so is equivalent to replace any element $q\in \{1,\ldots,n\}$ (before applying $\tilde{\sigma}$ !) by $\tau^i(q)$ in the indices occurring in the tensor product as well as in $\tilde{\sigma}(u,v)$.  This gives nothing else that \eqref{Useful1}. 
\end{proof}

We will particularly be interested in the case $n=3$, which takes the following form.
\begin{lem} \label{LemTrivect}
  Assume that $Q\in (D_BA)_3$, and denote by $\dgal{-,-,-}_Q$ the corresponding differential triple bracket. 
With the notation introduced in Lemma \ref{LemUseful}, we have  for any $a,b,c\in A$ 
\begin{equation} \label{Trivect}
  \tr\X (Q)(a_{ij},b_{kl},c_{uv})=(\dgal{a,b,c}_Q)_{uj,il,kv}- (\dgal{a,c,b}_Q)_{kj,iv,ul}\,.
\end{equation} 
\end{lem}

\begin{rem}
Let us look again at Proposition \ref{PropBr} when $\dgal{-,-}$ is differential for some $P \in (D_BA)_2$. 
First, looking at Lemma \ref{LemUseful} with $n=2$, the right-hand side of \eqref{relQRep} is the same as the right-hand side of \eqref{EqBrDefn} when $a_{ij}=a^1_{u_1 v_1},b_{kl}=a^2_{u_2v_2}$. Hence,  $\br{-,-}$ is equivalently defined by the bivector field $\tr \X(P)$ on $\Rep(A,\alpha)$, as first observed in \cite[\S 7.8]{VdB1}. 

Next, note that the left-hand side of \eqref{JacVdB} is obtained by applying the trivector $\frac12 [\tr \X(P),\tr \X(P)]$, where $[-,-]$ is the (geometric) Schouten-Nijenhuis bracket.  But it was remarked in \cite[\S 7.7]{VdB1} that taking traces defines a Lie algebra homomorphism from the algebraic to the geometric Schouten-Nijenhuis bracket, so that $\tr \X(\brSN{P,P})=[\tr \X(P),\tr \X(P)]$. Now, by Proposition \ref{TripleDiff}, the triple bracket $\dgal{-,-,-}$ defined by $\dgal{-,-}$ is differential with trivector $\frac12 \brSN{P,P}$. Therefore, \eqref{JacVdB} becomes a corollary of \eqref{Trivect} with $Q=\frac12 \brSN{P,P}$. 
\end{rem}

\subsection{Quasi-Poisson algebras}   \label{ss:qPA} 
Let $\g$ be a Lie algebra over $\kk$ such that $\g$ is equipped with a non-degenerate symmetric bilinear form denoted $(-|-)$. Furthermore, assume that the form is $\g$-invariant, i.e. $(\eta_1 | [\eta_2,\eta_3])=([\eta_1,\eta_2]|\eta_3)$ for all $\eta_1,\eta_2,\eta_3 \in \g$. 
If we take dual bases $(\vep_i),(\vep^i)$  under $(-|-)$, then we can define the \emph{Cartan trivector} $\phi \in \bigwedge^3 \g$ given by 
\begin{equation} \label{CartanTri}
  \phi=\frac{1}{12}\sum_{i,j,k}(\vep^i | [\vep^j,\vep^k])\, \vep_i \wedge \vep_j \wedge \vep_k\,.
\end{equation}

Following \cite[Section 2]{MT14} from now on, we assume that $\g$ acts on a commutative $\kk$-algebra $R$ by derivation, so that the map $\g \to \Der(R)$ is a Lie algebra homomorphism. Denoting by $\eta_R$ the action of $\eta \in \g$ on $R$, the latter means that $[\eta^1,\eta^2]_R(a)=\eta^1_R (\eta^2_R(a)) - \eta^2_R (\eta^1_R(a))$ for any $a \in R$, $\eta_1,\eta_2 \in \g$. We say that $R$ is a \emph{quasi-Poisson algebra} over $\g$ if $R$ is equipped with an anti-symmetric biderivation 
$\br{-,-}$ such that for any $\eta \in \g$ and $a,b,c\in R$ 
\begin{subequations}
  \begin{align}
    & \eta_R ( \br{a,b})=\br{\eta_R (a),b} + \br{a,\eta_R (b)}\,, \label{Eq:qP1}\\
&\br{a,\br{b,c}} + \br{b,\br{c,a}}  + \br{c,\br{a,b}}=\frac12 \phi_R(a,b,c)\,. \label{Eq:qP2}
  \end{align}
\end{subequations}
Here, $\phi_R$ is the image of the Cartan trivector induced by the map $\g^{\otimes 3} \times R^{\times 3} \to \kk$ given by 
\begin{equation*}
 (\eta^1 \otimes \eta^2 \otimes \eta^3,a,b,c)\mapsto 
(\eta^1 \otimes \eta^2 \otimes \eta^3)_R(a,b,c):=\eta^1_R(a)\eta^2_R(b)\eta^3_R(c)\,.
\end{equation*}
The operation $\br{-,-}$ is called a \emph{quasi-Poisson bracket}.  Note that if $R^\g\subset R$ is the subalgebra of $\g$-invariant elements, i.e. $R^\g=\{a \in R \mid \eta_R(a)=0\, \forall \eta \in \g\}$, then $\br{-,-}$ descends to a Poisson bracket on $R^\g$  since the right-hand side of \eqref{Eq:qP2} vanishes. 
\begin{rem}
In this work, we restrict the definition of quasi-Poisson algebra to the case where $\phi$ is the Cartan trivector \eqref{CartanTri}, in analogy with \cite{QuasiP,VdB1}. Working in greater generalities, Massuyeau and Turaev considered an arbitrary element $\phi \in \bigwedge^3 \g$, from which we still get a Poisson bracket on $R^\g$ \cite[\S 2.2]{MT14}. This notion also encompasses Poisson algebras when we take $\g=\{0\}$. 
\end{rem}

Assume that we are also given an arbitrary group $G$ acting on the left on $\g$ by Lie algebra automorphisms. (We do not require that $\g=\text{Lie}(G)$.) For any $g\in G$, we write the action as $\eta \mapsto { }^g\!\eta$, $\eta \in \g$. We say that $R$ is a $(G,\g)$-algebra if $R$ is a $\g$-algebra endowed with a compatible left $G$-action : 
\begin{equation}\label{Eq:qP3}
  ({ }^g\! \eta)_R\,a \,=\,g.\eta_R(g^{-1}.a)\,,\quad g \in G\,, \eta \in \g\,, a \in R\,.
\end{equation}
 We say that $R$ is a \emph{quasi-Poisson algebra} over the pair $(G,\g)$ if $R$ is a $(G,\g)$-algebra and if $R$ is a quasi-Poisson algebra over $\g$ such that for any $g\in G$, $a,b \in R$
\begin{subequations}
  \begin{align}
    & g. \br{a,b}=\br{g.a,g.b} \,, \label{Eq:qP4}\\
& { }^g\!\phi=\frac{1}{12}\sum_{i,j,k}(\vep^i | [\vep^j,\vep^k])\,\, { }^g\!\vep_i \wedge { }^g\!\vep_j \wedge { }^g\!\vep_k =\phi\,. \label{Eq:qP5}
  \end{align}
\end{subequations}
We easily see that if $R^G\subset R$ is the subalgebra of $G$-invariant elements, then the quasi-Poisson bracket descends to a Poisson bracket on $R^G\cap R^{\g}$.   

We now consider $R=\OO(\Rep(A,\alpha))$ as in \ref{sss:RepSpaces}. The algebra $R$ is naturally endowed with an action of $\Gl_\alpha=\prod_{s=1}^K \Gl_{\alpha_s}(\kk)$, which is given in matrix notation by $g . \X(a)=g^{-1}\X(a)g$ for all $a \in A$, $g \in \Gl_\alpha$. 
We can also consider the Lie algebra $\g_\alpha=\prod_{s=1}^K \gl_{\alpha_s}(\kk)$ of $\Gl_\alpha$, which acts by derivation on $R$ as $\eta_R (\X(a))=[\X(a),\eta]$, for all $a\in A$, $\eta\in \g_\alpha$. 
We can  endow $\g_\alpha$ with the trace pairing $(\eta_1|\eta_2)=\tr(\eta_1 \eta_2)$, and consider the left adjoint action of $\Gl_\alpha$ on $\g_\alpha$ so that \eqref{Eq:qP3} is satisfied. The following result generalises \cite[Theorem 7.12.2]{VdB1}, see also \cite[Lemma 4.4]{MT14}. (This was already noticed by Van den Bergh without a proof, as mentioned in \cite[Remark 7.12.3]{VdB1}.)

\begin{thm} \label{Thm:qRed}
Assume that $(A,\dgal{-,-})$ is a double quasi-Poisson algebra over $B$. Then the algebra $R=\OO(\Rep(A,\alpha))$ is a quasi-Poisson algebra over the pair $(\Gl_\alpha,\g_\alpha)$ for the quasi-Poisson bracket defined by Proposition \ref{PropBr}.  
\end{thm}
\begin{proof}
Showing \eqref{Eq:qP1}, \eqref{Eq:qP4} and \eqref{Eq:qP5} is easy, so we are left to show \eqref{Eq:qP2} on generators of the coordinate ring $R$. Hence,  fix $a,b,c\in A$. 
We remark that by \cite[Proposition 7.12.1]{VdB1} the 3-vector field $\phi_R$ is given by $\frac16 \sum_{s=1}^K \tr \X(E_s^3)$, hence we can write for any $1\leq i,j,k,l,u,v \leq N$ 
\begin{equation*}
\frac12 \phi_R(a_{ij},b_{kl},c_{uv})=\frac{1}{12}\sum_s \tr \X(E_s^3)(a_{ij},b_{kl},c_{uv})\,.
\end{equation*}
Using Lemma \ref{LemTrivect},  this is the same as
\begin{equation*}
\frac12 \phi_R(a_{ij},b_{kl},c_{uv})= 
  \left(\dgal{a,b,c}_{\frac{1}{12}\sum_s E_s^3}\right)_{uj,il,kv} 
- \left(\dgal{a,c,b}_{\frac{1}{12}\sum_s E_s^3}\right)_{kj,iv,ul}\,.
\end{equation*}
But then,  since the double bracket is quasi-Poisson we get by definition 
\begin{equation} \label{Eq:qRed}
\frac12 \phi_R(a_{ij},b_{kl},c_{uv})= 
   \dgal{a,b,c}_{uj,il,kv}- \dgal{a,c,b}_{kj,iv,ul}\,,
\end{equation}
where the triple bracket is defined from $\dgal{-,-}$ using \eqref{Eq:TripBr}. 
The right-hand side of \eqref{Eq:qRed} is nothing else than $Jac(a_{ij},b_{kl},c_{uv})$ by \eqref{JacVdB}. 
\end{proof}

If $\kk$ is algebraically closed, we can use Le Bruyn-Procesi Theorem \cite[Theorem 1]{LBP} to get that $A^{\Gl_\alpha}$ is generated by functions $\tr \X(a)$, $a\in A$, see e.g. \cite[Remark 4.3]{CB11}. In particular,  $A^{\Gl_\alpha}=A^{\Gl_\alpha} \cap A^{\g_\alpha}$. 
\begin{cor} \label{CorGL}
  Assume that $(A,\dgal{-,-})$ is a double quasi-Poisson algebra over $B$. If $\kk$ is an algebraically closed field of characteristic $0$, then the algebra $R^{\Gl_\alpha}=\OO(\Rep(A,\alpha)/\!/\Gl_\alpha)$ is a Poisson algebra whose Poisson bracket is induced by the quasi-Poisson bracket on $R$. 
\end{cor}

\begin{exmp}
  Fix integers $M \geq 1$ and $k_m\geq 3$ for $1\leq m \leq M$. Let $N=\max(k_1,\ldots,k_M)$. Combining Example \ref{Exmp:xN} and Theorem \ref{Thm:qRed}, we get that the algebra 
\begin{equation*}
  R=\kk\left[X_{m,ij} \mid 1\leq m \leq M,\,\, 1\leq i,j \leq N\right]
/\left(X_m^{k_m}=0_N \text{ for }X_m=(X_{m,ij}),\,\, 1\leq m \leq M\right)
\end{equation*}
is a quasi-Poisson algebra over the pair  $(\Gl_N(\kk),\gl_N(\kk))$ with quasi-Poisson bracket 
 \begin{equation*}
   \begin{aligned}
     \br{X_{m,ij},X_{m,kl}}=\frac12(X_m^2)_{kj} \delta_{il} - \frac12 \delta_{kj} (X_m^2)_{il}\,,& \quad 1\leq m \leq M\,, \\
 \br{X_{m,ij},X_{n,kl}}=\frac12 (X_n X_m)_{kj} \delta_{il} + \frac12 \delta_{kj} (X_m X_n)_{il} 
- \frac12 X_{n,kj} X_{m,il} - \frac12 X_{m,kj} X_{n,il}\,, &
 \quad 1\leq m < n \leq M\,.
   \end{aligned}
 \end{equation*}
When all the $(k_m)_m$ are equal, this gives a quasi-Poisson algebra structure on the coordinate ring corresponding to $M$ copies of the space of  nilpotent $N \times N$ matrices.  
\end{exmp}

\begin{exmp} \label{Ex:RepMT}
  If $\kk=\R$, we have by \cite[Appendix B]{MT14} that the double quasi-Poisson bracket of Massuyeau and Turaev given in Theorem \ref{ThmPI} endows $\Rep(\kk \pi_1(\Sigma,\ast),N)$ with the quasi-Poisson bracket given in \cite{QuasiP}.  
\end{exmp}

\subsection{Moment maps and Poisson algebra}
Consider the quasi-Poisson algebra $(R,\br{-,-})$ over the pair $(\Gl_\alpha,\g_\alpha)$ obtained from the double quasi-Poisson algebra $(A,\dgal{-,-})$ by Theorem \ref{Thm:qRed}. We now assume that $A$ is a quasi-Hamiltonian algebra, i.e. it is endowed with a moment map $\Phi\in \oplus_s e_s A e_s$. For any $(q_s)\in (\kk^\times)^K$, let $q=\sum_s q_s e_s \in B^\times$ and define the ideal $J_q$ generated by the entries of the matrix identity $\X(\Phi)-\X(q)=0_N$. 
We can form the algebra $R_q=R/J_q$, and denote by $\bar{r}$ the image of an element $r\in R$ under the projection $R\to R_q$. 

We clearly have that $J_q$ is $\Gl_\alpha$- and $\g_\alpha$-invariant, so that we can consider the induced actions on $R_q=R/J_q$. If we let $R_q^t\subset R_q$ denote the subalgebra generated by elements $\tr(\bar{r})$, $r\in R$, we can see that $R_q^t\subset R_q^{\Gl_\alpha}\cap R_q^{\g_\alpha}$.   
The next result follows either from \cite[Proposition 6.8.1]{VdB1} and \cite[Theorem 4.5]{CB11}, or from \cite[Proposition 7.13.2]{VdB1} and quasi-Hamiltonian reduction \cite{QuasiP}.

\begin{thm} \label{Thm:qRedMom}
  Let $(A,\dgal{-,-},\Phi)$ be a quasi-Hamiltonian algebra over $B$. Then, for any $q\in B^\times$, the algebra $R_q^t$ is a Poisson algebra whose Poisson bracket is induced by the quasi-Poisson bracket on $R$. 
\end{thm}

\begin{cor}
  Assume that $(A,\dgal{-,-},\Phi)$ is a  quasi-Hamiltonian algebra over $B$, and fix $q \in B^\times$. 
If $\kk$ is an algebraically closed field of characteristic $0$, then the algebra $R_q^{\Gl_\alpha}=\left(\OO(\Rep(A,\alpha))/(\X(\Phi-q))\right)^{\Gl_\alpha}$ is a Poisson algebra. 
\end{cor}

\begin{exmp}
  If $\kk$ is algebraically closed, the double quasi-Poisson bracket of Van den Bergh given in Theorem \ref{ThmVdB} (with $\gamma_a=+1$ for all $a\in \bar{Q}$) defines a Poisson structure on multiplicative quiver varieties of Crawley-Boevey and Shaw \cite{CBShaw}, see \cite[Theorem 1.1]{VdB1}. 
\end{exmp}


\appendix 

\section{Vanishing of the map \texorpdfstring{$\kappa$}{kappa}} \label{Ann:Kappa}

In this appendix, we prove Lemma \ref{LemE1E2}. Note that 
 $\kappa$ is a linear combination of triple brackets, so it is itself a triple bracket. By definition, it is a derivation in its last argument and is cyclically anti-symmetric. Thus, to show that $\kappa$ vanishes, it suffices to show that it is equal to zero when applied to generators of $A^f$. 
Before tackling this task, we use \eqref{Eq:TripBr} and remark that we can write 
\begin{equation*}
\begin{aligned}
  &\kappa(-,-,-)=\dgal{-,-,-}^f-\dgal{-,-,-}-\dgal{-,-,-}_{fus} \\
=&\sum_{r \in \Z_3} \tau_{(123)}^r \circ \left[(\dgal{-,-}\otimes 1_A)\circ(1_A \otimes \dgal{-,-}_{fus})+ (\dgal{-,-}_{fus}\otimes 1_A)\circ(1_A \otimes \dgal{-,-}) \right] \circ \tau_{(123)}^{-r}\,,
\end{aligned}
\end{equation*}
where $1_A$ is the identity map. 
Therefore, evaluated on some elements $a,b,c \in A^f$, we can write 
\begin{equation}\label{Eqkap}
  \begin{aligned}
    \kappa(a,b,c)=&
\,\,\underbrace{\dgal{a,\dgal{b,c}'_{fus}}\otimes \dgal{b,c}_{fus}''}_{A}
+\underbrace{\dgal{a,\dgal{b,c}'}_{fus}\otimes \dgal{b,c}''}_{A'} \\
&+ \underbrace{\tau_{(123)}\dgal{b,\dgal{c,a}'_{fus}}\otimes \dgal{c,a}_{fus}''}_{B} 
+ \underbrace{\tau_{(123)}\dgal{b,\dgal{c,a}'}_{fus}\otimes \dgal{c,a}''}_{B'}\\
&+ \underbrace{\tau_{(132)}\dgal{c,\dgal{a,b}'_{fus}}\otimes \dgal{a,b}_{fus}''}_{C}
+ \underbrace{\tau_{(132)}\dgal{c,\dgal{a,b}'}_{fus}\otimes \dgal{a,b}''}_{C'}\,,
  \end{aligned}
\end{equation}
so that we will write down the terms $A,B,C,A',B',C'$ for the different types of generators. Using the cyclicity, we only have twenty cases to check. We will only detail the computations in the first few cases, and we will give the final form of the terms $A,B,C,A',B',C'$ in the remaining cases so that the reader can check that they sum up to zero. 

Before beginning with the calculations, we remark that  identities involving the double bracket $\dgal{-,-}$ follow from extension from $A$ to $A^f$ which respects the derivation property in each variable. That is, given $e_+,f_+ \in \{\epsilon,e_{12}\}$ and $e_-,f_-\in \{\epsilon,e_{21}\}$, we have for any $a=e_+ \alpha e_-$, $b=f_+ \beta f_-$ with $\alpha,\beta\in A$ that  
\begin{equation} \label{EqAf}
  \dgal{a,b}=f_+ \dgal{\alpha,\beta}' e_- \otimes e_+ \dgal{\alpha,\beta}'' f_-\,.
\end{equation}
Here, in the left-hand side we have the induced double bracket on $A^f$, while the double bracket in the right-hand side is the original one on $A$.
Recall that we can choose generators $a,b\in A^f$ that admit such a decomposition by Lemma \ref{AfGenerat}.

\subsection{All generators of the same type} \label{proofSame} We drop the idempotent $\epsilon$ in our computations since this is the unit in $A^f$.

\noindent \emph{Generators of the second type.} Write $a=e_{12}\alpha$, $b=e_{12}\beta$ and $c=e_{12}\gamma$ for $\alpha,\beta,\gamma \in e_2 A \epsilon$. 
Using \eqref{uu}, then  the derivation property for the outer bimodule structure in the second entry of the double bracket on $A^f$ together with \eqref{EqAf}, we get that 
\begin{equation*}
  \begin{aligned}
A(a,b,c)=&-\frac12 \dgal{e_{12}\alpha,e_{12}\gamma e_{12}\beta}\otimes e_1 
=-\frac12 \left(e_{12}\gamma \dgal{e_{12}\alpha,e_{12}\beta} 
+ \dgal{e_{12}\alpha,e_{12}\gamma }e_{12}\beta\right)\otimes e_1    \\
=&-\frac12 e_{12}\dgal{\alpha,\gamma}' \otimes e_{12} \dgal{\alpha,\gamma}''e_{12}\beta \otimes e_1 
-\frac12 e_{12}\gamma e_{12}\dgal{\alpha,\beta}' \otimes e_{12}\dgal{\alpha,\beta}''\otimes e_1\,.
  \end{aligned}
\end{equation*}
Similarly we obtain 
\begin{equation*}
  \begin{aligned}
B(a,b,c)
=&-\frac12 \tau_{(123)}(e_{12}\dgal{\beta,\alpha}' \otimes e_{12} \dgal{\beta,\alpha}''e_{12}\gamma \otimes e_1 + e_{12}\alpha e_{12}\dgal{\beta,\gamma}' \otimes e_{12}\dgal{\beta,\gamma}''\otimes e_1)\, \\
=&-\frac12 e_1 \otimes e_{12}\dgal{\beta,\alpha}' \otimes e_{12} \dgal{\beta,\alpha}''e_{12}\gamma  
-\frac12 e_1 \otimes e_{12}\alpha e_{12}\dgal{\beta,\gamma}' \otimes e_{12}\dgal{\beta,\gamma}''\,,  \\
C(a,b,c)
=&-\frac12 \tau_{(132)}(e_{12}\dgal{\gamma,\beta}' \otimes e_{12} \dgal{\gamma,\beta}''e_{12}\alpha \otimes e_1 + e_{12}\beta e_{12}\dgal{\gamma,\alpha}' \otimes e_{12}\dgal{\gamma,\alpha}''\otimes e_1)\, \\
=&-\frac12  e_{12} \dgal{\gamma,\beta}''e_{12}\alpha \otimes e_1 \otimes  e_{12}\dgal{\gamma,\beta}' 
- \frac12  e_{12}\dgal{\gamma,\alpha}''\otimes e_1 \otimes e_{12}\beta e_{12}\dgal{\gamma,\alpha}'\,.
  \end{aligned}
\end{equation*}
Now, remark that \eqref{EqAf} gives $\dgal{e_{12} \beta,e_{12} \gamma}=e_{12}\dgal{\beta,\gamma}'\otimes e_{12}\dgal{\beta,\gamma}''$, so that the element in the first copy of $A^{\otimes 2}$ is also a generator of the second type. Using \eqref{uu} for the expression of $\dgal{-,-}_{fus}$, we get 
\begin{equation*}
  \begin{aligned}
A'(a,b,c)=&\dgal{e_{12}\alpha,e_{12}\dgal{\beta,\gamma}'}_{fus}\otimes e_{12}\dgal{\beta,\gamma}'' \\
=&\frac12 e_1 \otimes e_{12}\alpha e_{12}\dgal{\beta,\gamma}'\otimes e_{12}\dgal{\beta,\gamma}''
- \frac12 e_{12}\dgal{\beta,\gamma}' e_{12}\alpha \otimes e_1 \otimes e_{12}\dgal{\beta,\gamma}''\,.
  \end{aligned}
\end{equation*}
In the same way, we find 
\begin{equation*}
  \begin{aligned}
B'(a,b,c)=&
\frac12 \tau_{(123)}( e_1 \otimes e_{12}\beta e_{12}\dgal{\gamma,\alpha}'\otimes e_{12}\dgal{\gamma,\alpha}''
-  e_{12}\dgal{\gamma,\alpha}' e_{12}\beta \otimes e_1 \otimes e_{12}\dgal{\gamma,\alpha}'' ) \\
=& \frac12 e_{12}\dgal{\gamma,\alpha}'' \otimes e_1 \otimes e_{12}\beta e_{12}\dgal{\gamma,\alpha}'
- \frac12  e_{12}\dgal{\gamma,\alpha}'' \otimes e_{12}\dgal{\gamma,\alpha}' e_{12}\beta \otimes e_1  \,, \\
C'(a,b,c)=&
\frac12 \tau_{(132)}( e_1 \otimes e_{12}\gamma e_{12}\dgal{\alpha,\beta}'\otimes e_{12}\dgal{\alpha,\beta}''
-  e_{12}\dgal{\alpha,\beta}' e_{12}\gamma \otimes e_1 \otimes e_{12}\dgal{\alpha,\beta}'' ) \\
=& \frac12  e_{12}\gamma e_{12}\dgal{\alpha,\beta}'\otimes e_{12}\dgal{\alpha,\beta}'' \otimes e_1
- \frac12  e_1 \otimes e_{12}\dgal{\alpha,\beta}'' \otimes e_{12}\dgal{\alpha,\beta}' e_{12}\gamma\,.
  \end{aligned}
\end{equation*}
Summing all terms, we obtain after obvious cancellations 
\begin{equation*}
  \begin{aligned}
\kappa(a,b,c)=&
-\frac12 (e_{12}\dgal{\alpha,\gamma}' \otimes e_{12} \dgal{\alpha,\gamma}''e_{12}\beta \otimes e_1 
+  e_{12}\dgal{\gamma,\alpha}'' \otimes e_{12}\dgal{\gamma,\alpha}' e_{12}\beta \otimes e_1) \\
&-\frac12 (e_1 \otimes e_{12}\dgal{\beta,\alpha}' \otimes e_{12} \dgal{\beta,\alpha}''e_{12}\gamma 
+ e_1 \otimes e_{12}\dgal{\alpha,\beta}'' \otimes e_{12}\dgal{\alpha,\beta}' e_{12}\gamma) \\
&-\frac12 ( e_{12} \dgal{\gamma,\beta}''e_{12}\alpha \otimes e_1 \otimes  e_{12}\dgal{\gamma,\beta}'
+  e_{12}\dgal{\beta,\gamma}' e_{12}\alpha \otimes e_1 \otimes e_{12}\dgal{\beta,\gamma}'' )\,.
  \end{aligned}
\end{equation*}
It remains to notice in the last expression that all lines vanish using the cyclic antisymmetry of the double bracket.

\noindent \emph{Generators of the third type.} Write $a=\alpha e_{21}$, $b=\beta e_{21}$ and $c=\gamma e_{21}$ for $\alpha,\beta,\gamma \in \epsilon A e_2$. From \eqref{vv} and \eqref{EqAf} we get that 
\begin{equation*}
  \begin{aligned}
A(a,b,c)=&\frac12 \dgal{\alpha e_{21},\gamma e_{21}\beta e_{21}}\otimes e_1     \\
=&\frac12 \dgal{\alpha,\gamma}' e_{21} \otimes  \dgal{\alpha,\gamma}''e_{21} \beta e_{21} \otimes e_1 
+\frac12 \gamma e_{21}\dgal{\alpha,\beta}' e_{21} \otimes \dgal{\alpha,\beta}''e_{21} \otimes e_1\,.
  \end{aligned}
\end{equation*}
Similarly we obtain 
\begin{equation*}
  \begin{aligned}
B(a,b,c)
=&\frac12 e_1 \otimes \dgal{\beta,\alpha}' e_{21} \otimes  \dgal{\beta,\alpha}''e_{21} \gamma e_{21} 
+\frac12 e_1 \otimes \alpha e_{21}\dgal{\beta,\gamma}' e_{21} \otimes \dgal{\beta,\gamma}''e_{21} \,,  \\
C(a,b,c)
=&\frac12   \dgal{\gamma,\beta}''e_{21} \alpha e_{21} \otimes e_1  \otimes \dgal{\gamma,\beta}' e_{21}
+\frac12  \dgal{\gamma,\alpha}''e_{21} \otimes e_1 \otimes \beta e_{21}\dgal{\gamma,\alpha}' e_{21}\,. 
  \end{aligned}
\end{equation*}
Noticing from \eqref{EqAf} that $\dgal{b,c}'=\dgal{\beta,\gamma}' e_{21}$ is a generator of the third type, we get again from \eqref{vv} 
\begin{equation*}
  \begin{aligned}
A'(a,b,c)=&\dgal{\alpha e_{21},\dgal{\beta,\gamma}' e_{21}}_{fus}\otimes \dgal{\beta,\gamma}'' e_{21} \\
=&\frac12 \dgal{\beta,\gamma}' e_{21} \alpha e_{21} \otimes e_1 \otimes \dgal{\beta,\gamma}'' e_{21}
- \frac12 e_1 \otimes \alpha e_{21} \dgal{\beta,\gamma}' e_{21} \otimes \dgal{\beta,\gamma}'' e_{21}\,.
  \end{aligned}
\end{equation*}
Analogously 
\begin{equation*}
  \begin{aligned}
B'(a,b,c)
=& \frac12 \dgal{\gamma,\alpha}'' e_{21} \otimes \dgal{\gamma,\alpha}' e_{21} \beta e_{21} \otimes e_1 
- \frac12 \dgal{\gamma,\alpha}'' e_{21} \otimes e_1 \otimes \beta e_{21} \dgal{\gamma,\alpha}' e_{21}  \,,  \\
C'(a,b,c)
=&\frac12  e_1 \otimes \dgal{\alpha,\beta}'' e_{21} \otimes  \dgal{\alpha,\beta}' e_{21} \gamma e_{21}
- \frac12   \gamma e_{21} \dgal{\alpha,\beta}' e_{21} \otimes \dgal{\alpha,\beta}'' e_{21} \otimes e_1 \,. 
  \end{aligned}
\end{equation*}
Summing all terms yield 
\begin{equation*}
  \begin{aligned}
\kappa(a,b,c)=&
+\frac12 (\dgal{\alpha,\gamma}' e_{21} \otimes  \dgal{\alpha,\gamma}''e_{21} \beta e_{21} \otimes e_1
+ \dgal{\gamma,\alpha}'' e_{21} \otimes \dgal{\gamma,\alpha}' e_{21} \beta e_{21} \otimes e_1 ) \\
&+\frac12 (e_1 \otimes \dgal{\beta,\alpha}' e_{21} \otimes  \dgal{\beta,\alpha}''e_{21} \gamma e_{21} 
+  e_1 \otimes \dgal{\alpha,\beta}'' e_{21} \otimes  \dgal{\alpha,\beta}' e_{21} \gamma e_{21}) \\
&+\frac12 (\dgal{\gamma,\beta}''e_{21} \alpha e_{21} \otimes e_1  \otimes \dgal{\gamma,\beta}' e_{21} 
+  \dgal{\beta,\gamma}' e_{21} \alpha e_{21} \otimes e_1 \otimes \dgal{\beta,\gamma}'' e_{21} )\,,
  \end{aligned}
\end{equation*}
which is zero using the cyclic antisymmetry. 

\noindent \emph{Generators of the first type.} For $a,b,c\in \epsilon A \epsilon$, we have by \eqref{EqAf} that the  double bracket $\dgal{-,-}$ evaluated on any two of these elements belongs to $(\epsilon A \epsilon)^{\otimes 2}$. At the same time, \eqref{tt} gives that $\dgal{\epsilon A \epsilon,\epsilon A \epsilon}_{fus}=0$. Hence all terms in \eqref{Eqkap} trivially vanish and $\kappa(a,b,c)=0$. 

\noindent \emph{Generators of the fourth type.} As in the first type case, we use \eqref{EqAf}  to get that $\dgal{e_{12} A e_{21},e_{12} A e_{21}}\subset (e_{12} A e_{21})^{\otimes 2}$ and \eqref{ww} to obtain $\dgal{e_{12} A e_{21},e_{12} A e_{21}}_{fus}=0$, so that all terms vanish.

\subsection{Two generators of the first type}  

Let $a,b\in \epsilon A \epsilon$. 

\noindent \emph{With one generator of the second type.}  Consider $c=e_{12}\gamma$ for some $\gamma \in e_2 A \epsilon$. Using \eqref{tu} and \eqref{ut},
\begin{equation*}
\begin{aligned}
A=&-\frac12 e_1\dgal{a, b}' \otimes \dgal{a, b}''\otimes e_{12}\gamma\,, \\
B=& \frac12 e_1 a \otimes e_{12}\dgal{b,\gamma}' \otimes \dgal{b,\gamma}''
-\frac12 e_1 \otimes \dgal{b,a }' \otimes \dgal{b,a }'' e_{12}\gamma
-\frac12 e_1 \otimes a e_{12}\dgal{b,\gamma}' \otimes \dgal{b,\gamma}'' \,.
\end{aligned}
\end{equation*}
By \eqref{tt}, $C$ trivially vanishes. It is also the case for $B'$ because $\dgal{e_{12}\gamma,a}'\in \epsilon A \epsilon$. Next we get by \eqref{tu} and \eqref{ut} that 
\begin{equation*}
\begin{aligned}
  A'=& \frac12 e_1 \otimes a e_{12} \dgal{b,\gamma}' \otimes \dgal{b,\gamma}'' 
- \frac12 e_1 a \otimes e_{12}\dgal{b,\gamma}' \otimes \dgal{b,\gamma}''\,, \\
    C'=&\frac12  e_1 \dgal{a,b}' \otimes \dgal{a,b}'' \otimes e_{12}\gamma 
- \frac12 e_1 \otimes \dgal{a,b}'' \otimes \dgal{a,b}'e_{12}\gamma\,,
\end{aligned}
\end{equation*}
so that all terms cancel out together (after using the cyclic antisymmetry, which we will need in each of the remaining cases).  

\noindent \emph{With one generator of the third type.}  Consider $c=\gamma e_{21}$ for some $\gamma \in \epsilon A e_2$. We get from \eqref{tv} and \eqref{vt}  that 
\begin{equation*}
\begin{aligned}
  A=&\frac12 \dgal{a,\gamma}' \otimes \dgal{a,\gamma}'' e_{21}b \otimes e_1
+ \frac12 \gamma e_{21}\dgal{a,b}' \otimes \dgal{a,b}'' \otimes e_1 
- \frac12 \dgal{a,\gamma}' \otimes \dgal{a,\gamma}'' e_{21} \otimes b e_1\,, \\
B=& \frac12 \gamma e_{21} \otimes \dgal{b,a}' \otimes \dgal{b,a}'' e_1\,.
\end{aligned}
\end{equation*}
Again using \eqref{tt} we have $C=0$, and $A'=0$ since $\dgal{b,\gamma e_{21}}' \in \epsilon A \epsilon$. Finally, from \eqref{tv} and \eqref{vt} we get 
\begin{equation*}
\begin{aligned}
  B'
=& \frac12 \dgal{\gamma,a}'' \otimes \dgal{\gamma,a}' e_{21}b \otimes e_1 
- \frac12 \dgal{\gamma,a}'' \otimes \dgal{\gamma,a}' e_{21}\otimes b e_1\,,\\
C'
=&\frac12 \gamma e_{21} \otimes \dgal{a,b}'' \otimes \dgal{a,b}' e_1
- \frac12   \gamma e_{21}\dgal{a,b}' \otimes \dgal{a,b}'' \otimes e_1\,,
\end{aligned}
\end{equation*}
and all terms sum up to zero. 

\noindent \emph{With one generator of the fourth type.} Consider $c=e_{12} \gamma e_{21}$ for some $\gamma \in e_2 A e_2$. 
First, using \eqref{tw} and  \eqref{wt}  we get 
\begin{equation*}
\begin{aligned}
    A=& \frac12 e_{12}\dgal{a,\gamma}' \otimes \dgal{a,\gamma}'' e_{21}b \otimes e_1 
+ \frac12 e_{12}\gamma e_{21}\dgal{a,b}' \otimes \dgal{a,b}'' \otimes e_1  \\
&-\frac12 e_{12}\dgal{a,\gamma}' \otimes \dgal{a,\gamma}''e_{21} \otimes b e_1 
- \frac12 e_1\dgal{a,b}' \otimes \dgal{a,b}'' \otimes e_{12}\gamma e_{21}\,, \\
B=& \frac12 e_{12}\gamma e_{21} \otimes \dgal{b,a}' \otimes \dgal{b,a}'' e_1 
+ \frac12 e_1 a \otimes e_{12}\dgal{b,\gamma}' \otimes \dgal{b,\gamma}'' e_{21} \\
&-\frac12 e_1 \otimes \dgal{b,a}'\otimes \dgal{b,a}'' e_{12}\gamma e_{21}
- \frac12 e_1 \otimes ae_{12}\dgal{b,\gamma}' \otimes \dgal{b,\gamma}'' e_{21}\,.
\end{aligned}
\end{equation*}
Again, $C=0$ by \eqref{tt}. Meanwhile, we find from \eqref{tu}, \eqref{tv} and \eqref{wt}
\begin{equation*}
\begin{aligned}
    A'
=&\frac12 e_1 \otimes a e_{12} \dgal{b,\gamma}' \otimes  \dgal{b,\gamma}'' e_{21} 
- \frac12 e_1 a \otimes e_{12}\dgal{b,\gamma}' \otimes \dgal{b,\gamma}'' e_{21}\,, \\
    B'
=&\frac12 e_{12}\dgal{\gamma,a}'' \otimes \dgal{\gamma,a}' e_{21}b \otimes e_1 
- \frac12 e_{12}\dgal{\gamma,a}'' \otimes \dgal{\gamma,a}' e_{21} \otimes b e_1 \,, \\
    C'
=&\frac12  e_{12}\gamma e_{21} \otimes \dgal{a,b}'' \otimes \dgal{a,b}' e_1
+ \frac12  e_1\dgal{a,b}' \otimes \dgal{a,b}'' \otimes e_{12}\gamma e_{21} \\
&- \frac12  e_1 \otimes \dgal{a,b}''  \otimes \dgal{a,b}' e_{12}\gamma e_{21}
- \frac12  e_{12}\gamma e_{21} \dgal{a,b}' \otimes \dgal{a,b}'' \otimes e_1\,.
\end{aligned}
\end{equation*}
Summing terms together, we get $\kappa=0$.

\subsection{Two generators of the second type} 
Let $a=e_{12}\alpha,b=e_{12}\beta $ for $\alpha,\beta\in e_2 A \epsilon$. We only collect the final form of the terms $A,B,C,A',B',C'$ from now on, and the reader can check that they sum up to zero. 

\noindent \emph{With one generator of the first type.}  Consider $c\in \epsilon A \epsilon$.  
\begin{equation*}
 \begin{aligned}
     A
=&\frac12 e_{12}\dgal{\alpha,\beta}'\otimes e_{12}\dgal{\alpha,\beta}'' \otimes e_1 c
- \frac12 \dgal{\alpha,c}' \otimes e_{12} \dgal{\alpha,c}''e_{12}\beta \otimes e_1
- \frac12 ce_{12} \dgal{\alpha,\beta}'\otimes e_{12}\dgal{\alpha,\beta}'' \otimes e_1\,, \\
B
=& -\frac12 e_{12}\alpha \otimes e_1\dgal{\beta,c}' \otimes e_{12} \dgal{\beta,c}''\,, \\
C
=&-\frac12  \dgal{c,\beta}'' e_{12} \alpha \otimes e_1 \otimes e_{12}\dgal{c, \beta}'
- \frac12  \dgal{c,\alpha}'' \otimes e_1 \otimes e_{12}\beta e_{12} \dgal{c,\alpha}'\,, \\ 
A' 
=& \frac12 e_{12}\alpha \otimes e_1 \dgal{\beta,c}' \otimes e_{12} \dgal{\beta,c}''
-\frac12 \dgal{\beta,c}'e_{12}\alpha \otimes e_1 \otimes e_{12}\dgal{\beta,c}''\,, \\
B' 
=& \frac12 \dgal{c,\alpha}'' \otimes e_1 \otimes e_{12}\beta e_{12} \dgal{c,\alpha}' 
-\frac12 \dgal{c,\alpha}'' \otimes e_{12}\dgal{c,\alpha}' e_{12}\beta \otimes e_1 \,, \\
C' 
=& \frac12  c e_{12}\dgal{\alpha,\beta}' \otimes e_{12}\dgal{\alpha,\beta}'' \otimes e_1
-\frac12 e_{12}\dgal{\alpha,\beta}' \otimes e_{12}\dgal{\alpha,\beta}'' \otimes e_1c\,.
 \end{aligned}
\end{equation*}

\noindent \emph{With one generator of the third type.}  Consider $c=\gamma e_{21}$ for some $\gamma \in \epsilon A e_2$. 
\begin{equation*}
\begin{aligned}
  A
=&\frac12 e_{12}\dgal{\alpha,\beta}' \otimes e_{12}\dgal{\alpha,\beta}'' \otimes e_1 \gamma e_{21}
- \frac12 \dgal{\alpha,\gamma}' \otimes e_{12}\dgal{\alpha,\gamma}'' e_{21}\otimes e_{12}\beta e_1\,, \\
B
=&\frac12 \gamma e_{21} \otimes e_{12}\dgal{\beta, \alpha}' \otimes e_{12}\dgal{\beta,\alpha}'' e_1 
- \frac12  e_{12} \alpha \otimes e_1 \dgal{\beta,\gamma}' \otimes e_{12}\dgal{\beta,\gamma}''e_{21} \,, \\
C
=&- \frac12 \dgal{\gamma,\beta}'' e_{12}\alpha \otimes e_1 \otimes  e_{12}\dgal{\gamma,\beta}' e_{21}
-\frac12  \dgal{\gamma,\alpha}'' \otimes e_1 \otimes e_{12}\beta e_{12}\dgal{\gamma,\alpha}'e_{21}\,, \\
A'
=& \frac12 e_{12}\alpha \otimes e_1 \dgal{\beta,\gamma}' \otimes e_{12} \dgal{\beta,\gamma}''e_{21}
-\frac12 \dgal{\beta,\gamma}'e_{12}\alpha \otimes e_1 \otimes e_{12}\dgal{\beta,\gamma}''e_{21}\,, \\
B'
 =& \frac12  \dgal{\gamma,\alpha}''  \otimes e_1 \otimes e_{12} \beta e_{12} \dgal{\gamma,\alpha}' e_{21} 
 - \frac12  \dgal{\gamma,\alpha}'' \otimes e_{12} \dgal{\gamma,\alpha}' e_{21} \otimes e_{12} \beta e_1\,, \\
C'
=& \frac12   \gamma e_{21} \otimes e_{12}\dgal{\alpha,\beta}'' \otimes e_{12}\dgal{\alpha,\beta}'e_1 
- \frac12   e_{12}\dgal{\alpha,\beta}' \otimes e_{12}\dgal{\alpha,\beta}'' \otimes e_1 \gamma e_{21}\,.
 \end{aligned}
\end{equation*}

\noindent \emph{With one generator of the fourth type.} Consider $c=e_{12} \gamma e_{21}$ for some $\gamma \in e_2 A e_2$. 
\begin{equation*}
\begin{aligned}
  A
=&-\frac12 e_{12}\dgal{\alpha,\gamma}' \otimes e_{12}\dgal{\alpha,\gamma}'' e_{21} \otimes e_{12} \beta e_1\,, \\
B
=&\frac12 e_{12} \gamma e_{21} \otimes e_{12}\dgal{\beta,\alpha}' \otimes e_{12} \dgal{\beta,\alpha}'' e_1  
- \frac12 e_1 \otimes e_{12}\dgal{\beta,\alpha}' \otimes e_{12}\dgal{\beta,\alpha}'' e_{12}\gamma e_{21} \\
&- \frac12 e_1 \otimes e_{12}\alpha e_{12}\dgal{\beta,\gamma}' \otimes e_{12}\dgal{\beta,\gamma}'' e_{21} \,, \\
C 
=&- \frac12 e_{12}\dgal{\gamma,\beta}'' e_{12}\alpha \otimes e_1 \otimes  e_{12}\dgal{\gamma,\beta}' e_{21}
-\frac12  e_{12}\dgal{\gamma,\alpha}'' \otimes e_1 \otimes e_{12}\beta e_{12}\dgal{\gamma,\alpha}'e_{21}\,, \\
A'
=& \frac12 e_1 \otimes e_{12}\alpha e_{12}\dgal{\beta,\gamma}'   \otimes e_{12} \dgal{\beta,\gamma}''e_{21}
-\frac12   e_{12}\dgal{\beta,\gamma}'e_{12}\alpha \otimes e_1   \otimes e_{12}\dgal{\beta,\gamma}''e_{21}\,, \\
B'
 =& \frac12 e_{12} \dgal{\gamma,\alpha}''  \otimes e_1 \otimes e_{12} \beta e_{12} \dgal{\gamma,\alpha}' e_{21} 
 - \frac12 e_{12} \dgal{\gamma,\alpha}'' \otimes e_{12} \dgal{\gamma,\alpha}' e_{21} \otimes e_{12} \beta e_1\,, \\
C'
=& \frac12   e_{12}\gamma e_{21} \otimes e_{12}\dgal{\alpha,\beta}'' \otimes e_{12}\dgal{\alpha,\beta}' e_1 
- \frac12  e_1 \otimes e_{12}\dgal{\alpha,\beta}'' \otimes e_{12}\dgal{\alpha,\beta}' e_{12}\gamma e_{21}\,.
 \end{aligned}
\end{equation*}

\subsection{Two generators of the third type} 
Let $a=\alpha e_{21},b=\beta e_{21}$ for $\alpha,\beta\in  \epsilon A e_2$. 

\noindent \emph{With one generator of the first type.}  Consider $c\in \epsilon A \epsilon$. 
\begin{equation*} 
   \begin{aligned}
  A=&\frac12 \dgal{\alpha ,c}' e_{21} \otimes \dgal{\alpha ,c}'' e_1\otimes \beta e_{21}\,, \\
B
=&\frac12 e_1 \otimes \dgal{\beta,\alpha}' e_{21} \otimes \dgal{\beta,\alpha}'' e_{21}c  
+ \frac12 e_1 \otimes \alpha e_{21}\dgal{\beta,c}' e_{21} \otimes \dgal{\beta,c}'' 
- \frac12 c e_1 \otimes \dgal{\beta,\alpha}' e_{21} \otimes \dgal{\beta,\alpha}'' e_{21} \,,  \\
C
=& \frac12   \dgal{c,\beta}'' e_{21} \alpha e_{21} \otimes e_1 \otimes \dgal{c,\beta}'
+ \frac12   \dgal{c,\alpha}'' e_{21} \otimes e_1 \otimes \beta e_{21} \dgal{c,\alpha}' \,, \\
A'
=&\frac12 \dgal{\beta,c}'e_{21} \alpha e_{21} \otimes e_1 \otimes \dgal{\beta,c}'' 
- \frac12 e_1 \otimes  \alpha e_{21} \dgal{\beta,c}'e_{21} \otimes \dgal{\beta,c}''\,, \\
B'
=&\frac12 \dgal{c,\alpha}'' e_{21} \otimes \dgal{c,\alpha}'e_1 \otimes \beta e_{21} 
- \frac12 \dgal{c,\alpha}'' e_{21} \otimes  e_1 \otimes \beta e_{21} \dgal{c,\alpha}'  \,, \\
C'
=&\frac12   e_1 \otimes \dgal{\alpha,\beta}''e_{21} \otimes \dgal{\alpha,\beta}'e_{21} c
- \frac12  c e_1 \otimes \dgal{\alpha,\beta}''e_{21} \otimes \dgal{\alpha,\beta}'e_{21}    \,.
 \end{aligned}
\end{equation*}

\noindent \emph{With one generator of the second type.}  Consider $c=e_{12}\gamma$ for some $\gamma \in e_2 A \epsilon$. 
\begin{equation*}
 \begin{aligned}
     A
=& \frac12 e_{12} \dgal{\alpha ,\gamma}'e_{21} \otimes \dgal{\alpha ,\gamma}'' e_1 \otimes \beta e_{21} 
- \frac12 e_1 \dgal{\alpha ,\beta}' e_{21} \otimes \dgal{\alpha ,\beta}'' e_{21} \otimes e_{12} \gamma\,,  \\
B
=&\frac12  e_1 \alpha e_{21} \otimes e_{12} \dgal{\beta ,\gamma}' e_{21} \otimes \dgal{\beta ,\gamma}''  
- \frac12  e_{12} \gamma e_1 \otimes  \dgal{\beta ,\alpha}' e_{21} \otimes \dgal{\beta ,\alpha}'' e_{21} \,, \\
C
=& \frac12   e_{12}\dgal{\gamma,\beta}'' e_{21} \alpha e_{21} \otimes e_1 \otimes \dgal{\gamma,\beta}'
+ \frac12   e_{12}\dgal{\gamma,\alpha}'' e_{21} \otimes e_1 \otimes \beta e_{21} \dgal{\gamma,\alpha}' \,, \\
A'
=&\frac12 e_{12} \dgal{\beta,\gamma}' e_{21} \alpha e_{21} \otimes e_1 \otimes \dgal{\beta,\gamma}''
- \frac12 e_1 \alpha e_{21} \otimes e_{12}\dgal{\beta,\gamma}'e_{21} \otimes \dgal{\beta,\gamma}''  \,,\\
B'
=&\frac12 e_{12}\dgal{\gamma,\alpha}'' e_{21} \otimes \dgal{\gamma,\alpha}'e_1 \otimes \beta e_{21} 
- \frac12 e_{12}\dgal{\gamma,\alpha}'' e_{21} \otimes  e_1 \otimes \beta e_{21} \dgal{\gamma,\alpha}'  \,,\\
C'
=&\frac12    e_1 \dgal{\alpha,\beta}' e_{21} \otimes \dgal{\alpha,\beta}'' e_{21} \otimes e_{12} \gamma
- \frac12   e_{12} \gamma e_1 \otimes \dgal{\alpha,\beta}'' e_{21} \otimes \dgal{\alpha,\beta}'e_{21}    \,.
 \end{aligned}
\end{equation*}

\noindent \emph{With one generator of the fourth type.} Consider $c=e_{12} \gamma e_{21}$ for some $\gamma \in e_2 A e_2$. 
\begin{equation*}
 \begin{aligned}
     A
=& \frac12 e_{12}\dgal{\alpha,\gamma}'e_{21} \otimes \dgal{\alpha,\gamma}'' e_{21}\beta e_{21} \otimes e_1 
+ \frac12 e_{12}\gamma e_{21} \dgal{\alpha,\beta}'e_{21} \otimes \dgal{\alpha,\beta}'' e_{21} \otimes e_1  \\
&- \frac12 e_1 \dgal{\alpha,\beta}'e_{21} \otimes \dgal{\alpha,\beta}'' e_{21} \otimes e_{12}\gamma e_{21}\,, \\
B
=&\frac12 e_1 \alpha e_{21} \otimes e_{12}\dgal{\beta,\gamma}'e_{21} \otimes \dgal{\beta,\gamma}'' e_{21}\,, \\
C
=& \frac12   e_{12}\dgal{\gamma,\beta}'' e_{21} \alpha e_{21} \otimes e_1 \otimes \dgal{\gamma,\beta}' e_{21}
+ \frac12   e_{12}\dgal{\gamma,\alpha}'' e_{21} \otimes e_1 \otimes \beta e_{21} \dgal{\gamma,\alpha}' e_{21} \,, \\
A'
=&\frac12 e_{12} \dgal{\beta,\gamma}' e_{21} \alpha e_{21} \otimes e_1 \otimes \dgal{\beta,\gamma}'' e_{21} 
- \frac12 e_1 \alpha e_{21} \otimes e_{12}\dgal{\beta,\gamma}'e_{21} \otimes \dgal{\beta,\gamma}''  e_{21}  \,,\\
B'
=&\frac12 e_{12}\dgal{\gamma,\alpha}'' e_{21} \otimes \dgal{\gamma,\alpha}'e_{21} \beta e_{21} \otimes e_1
- \frac12 e_{12}\dgal{\gamma,\alpha}'' e_{21} \otimes e_1 \otimes \beta e_{21} \dgal{\gamma,\alpha}' e_{21} \,,\\
C'
=&\frac12    e_1 \dgal{\alpha,\beta}'e_{21} \otimes \dgal{\alpha,\beta}''e_{21} \otimes e_{12}\gamma e_{21}
- \frac12 e_{12} \gamma e_{21} \dgal{\alpha,\beta}'e_{21} \otimes \dgal{\alpha,\beta}''e_{21} \otimes e_1  \,.
 \end{aligned}
\end{equation*}

\subsection{Two generators of the fourth type} 
Let $a=e_{12}\alpha e_{21},b=e_{12}\beta e_{21}$ for $\alpha,\beta\in  e_2 A e_2$. 

\noindent \emph{With one generator of the first type.}  Consider $c\in \epsilon A \epsilon$. 
We get $C=0$, while 
\begin{equation*}
 \begin{aligned} 
    A
=& \frac12 \dgal{\alpha,c}' e_{21} \otimes e_{12} \dgal{\alpha,c}'' e_1 \otimes e_{12} \beta e_{21} 
+ \frac12  e_{12} \dgal{\alpha,\beta}' e_{21} \otimes e_{12} \dgal{\alpha,\beta}'' e_{21} \otimes e_1 c   \\
&- \frac12 \dgal{\alpha,c}' e_{21} \otimes e_{12} \dgal{\alpha,c}'' e_{12} \beta e_{21} \otimes e_1 
- \frac12 c e_{12} \dgal{\alpha,\beta}' e_{21} \otimes e_{12} \dgal{\alpha,\beta}'' e_{21} \otimes e_1   \,, \\
B
=& \frac12 e_1 \otimes e_{12}\dgal{\beta,\alpha}'e_{21} \otimes e_{12}\dgal{\beta,\alpha}'' e_{21} c 
+ \frac12  e_1 \otimes e_{12}\alpha e_{21}\dgal{\beta,c}' e_{21} \otimes e_{12} \dgal{\beta,c}''  \\
&- \frac12 c e_1 \otimes e_{12} \dgal{\beta,\alpha}' e_{21} \otimes e_{12} \dgal{\beta,\alpha}'' e_{21}
- \frac12  e_{12} \alpha e_{21} \otimes e_1 \dgal{\beta,c}' e_{21} \otimes e_{12} \dgal{\beta,c}''   \,, 
 \end{aligned}
\end{equation*}
\begin{equation*}
 \begin{aligned} 
A'
=& \frac12 e_{12}\alpha e_{21} \otimes e_1 \dgal{\beta ,c}' e_{21} \otimes e_{12} \dgal{\beta ,c}''
- \frac12 e_1 \otimes e_{12} \alpha e_{21} \dgal{\beta ,c}' e_{21} \otimes e_{12} \dgal{\beta ,c}''\,, \\
B'
=& \frac12 \dgal{c,\alpha}'' e_{21} \otimes e_{12} \dgal{c,\alpha}' e_1 \otimes e_{12} \beta e_{21}
- \frac12 \dgal{c,\alpha}'' e_{21} \otimes e_{12} \dgal{c,\alpha}' e_{12} \beta e_{21} \otimes e_1\,, \\
C'
=&\frac12  e_1 \otimes e_{12}\dgal{\alpha,\beta}''e_{21} \otimes  e_{12}\dgal{\alpha,\beta}'e_{21} c
+\frac12  c e_{12}\dgal{\alpha,\beta}'e_{21} \otimes e_{12}\dgal{\alpha,\beta}''e_{21} \otimes  e_1   \\
&- \frac12   c e_1 \otimes e_{12}\dgal{\alpha,\beta}''e_{21} \otimes e_{12}\dgal{\alpha,\beta}'e_{21}
- \frac12  e_{12}\dgal{\alpha,\beta}'e_{21} \otimes e_{12}\dgal{\alpha,\beta}''e_{21}  \otimes  e_1 c  \,.
 \end{aligned}
\end{equation*}

\noindent \emph{With one generator of the second type.}  Consider $c=e_{12}\gamma$ for some $\gamma \in e_2 A \epsilon$.   We get $C=0$, $A'=0$,  while 
\begin{equation*}
 \begin{aligned} 
    A
=&\frac12 e_{12}\dgal{\alpha,\gamma}' e_{21} \otimes e_{12} \dgal{\alpha,\gamma}'' e_1 \otimes e_{12}\beta e_{21} 
- \frac12 e_{12}\dgal{\alpha,\gamma}' e_{21} \otimes e_{12} \dgal{\alpha,\gamma}'' e_{12} \beta e_{21} \otimes e_1 \\
&-\frac12 e_{12}\gamma e_{12}\dgal{\alpha,\beta}' e_{21}\otimes e_{12}\dgal{\alpha,\beta}'' e_{21} \otimes e_1\,,\\
     B
=&-\frac12 e_{12} \gamma e_1\otimes  e_{12}\dgal{\beta,\alpha}' e_{21} \otimes e_{12} \dgal{\beta,\alpha}'' e_{21}
\,, \\
B'
=& \frac12 e_{12} \dgal{\gamma,\alpha}'' e_{21} \otimes e_{12} \dgal{\gamma,\alpha}' e_1 \otimes e_{12} \beta e_{21}
- \frac12 e_{12} \dgal{\gamma,\alpha}'' e_{21} \otimes e_{12} \dgal{\gamma,\alpha}' e_{12} \beta e_{21} \otimes e_1\,, \\
C'
=& \frac12  e_{12}\gamma e_{12}\dgal{\alpha,\beta}'e_{21} \otimes e_{12}\dgal{\alpha,\beta}''e_{21} \otimes e_1
-\frac12  e_{12}\gamma e_1 \otimes e_{12}\dgal{\alpha,\beta}''e_{21} \otimes e_{12}\dgal{\alpha,\beta}'e_{21}  \,.
 \end{aligned}
\end{equation*}

\noindent \emph{With one generator of the third type.}  Consider $c=\gamma e_{21}$ for some $\gamma \in \epsilon A e_2$.   We get $C=0$, $B'=0$, while 
\begin{equation*}
 \begin{aligned} 
    A=&
\frac12 e_{12}\dgal{\alpha,\beta}'e_{21}\otimes e_{12}\dgal{\alpha,\beta}'' e_{21}\otimes e_1 \gamma e_{21}\,, \\
B
=&\frac12 e_1 \otimes e_{12}\dgal{\beta,\alpha}'e_{21} \otimes e_{12}\dgal{\beta,\alpha}'' e_{21}\gamma e_{21}
+ \frac12 e_1 \otimes e_{12} \alpha e_{21} \dgal{\beta,\gamma}' e_{21} \otimes e_{12} \dgal{\beta,\gamma}'' e_{21}\\
&-\frac12 e_{12}\alpha e_{21} \otimes e_1 \dgal{\beta,\gamma}' e_{21}\otimes e_{12}\dgal{\beta,\gamma}''e_{21}\,,\\
A'
=& \frac12 e_{12}\alpha e_{21} \otimes e_1 \dgal{\beta,\gamma}' e_{21} \otimes e_{12}\dgal{\beta,\gamma}'' e_{21} 
-\frac12 e_1 \otimes e_{12}\alpha e_{21}\dgal{\beta,\gamma}'e_{21}\otimes e_{12}\dgal{\beta,\gamma}'' e_{21}\,,\\
C'
=& \frac12 e_1  \otimes e_{12}\dgal{\alpha,\beta}''e_{21} \otimes e_{12}\dgal{\alpha,\beta}'e_{21} \gamma e_{21}
-\frac12 e_{12}\dgal{\alpha,\beta}'e_{21}    \otimes e_{12}\dgal{\alpha,\beta}''e_{21} \otimes e_1 \gamma e_{21} \,.
 \end{aligned}
\end{equation*}

\subsection{Remaining cases} \label{proofRemain} We now take three different types of generators. 

\noindent \emph{No generator of the fourth type.} Let $a\in \epsilon A \epsilon$, $b=e_{12}\beta$ for $\beta\in e_2 A \epsilon$ and $c=\gamma e_{21}$ for $\gamma \in \epsilon A e_2$. We have $A'=0$, while 
\begin{equation*}
 \begin{aligned} 
    A
=&\frac12 e_{12}\dgal{a,\beta}' \otimes \dgal{a,\beta}'' \otimes e_1 \gamma e_{21} 
- \frac12 \dgal{a,\gamma}' \otimes \dgal{a,\gamma}''e_{21} \otimes e_{12} \beta e_1\,,\\
B
=& \frac12 \gamma e_{21}\otimes  \dgal{\beta,a}' \otimes e_{12}\dgal{\beta,a}'' e_1\,, \\ 
C
=& -\frac12  \dgal{\gamma,a}'' \otimes e_{12} \beta  \otimes e_1\dgal{\gamma,a}' e_{21}\,,\\
B'
=& \frac12 \dgal{\gamma,a}'' \otimes e_{12}\beta \otimes e_1 \dgal{\gamma,a}' e_{21}
- \frac12 \dgal{\gamma,a}'' \otimes \dgal{\gamma,a}'e_{21} \otimes e_{12}\beta e_1\,, \\
C'
=& \frac12   \gamma e_{21} \otimes \dgal{a,\beta}''  \otimes e_{12}\dgal{a,\beta}' e_1 
- \frac12   e_{12}\dgal{a,\beta}' \otimes \dgal{a,\beta}'' \otimes e_1 \gamma e_{21}  \,.
 \end{aligned}
\end{equation*}

\noindent \emph{No generator of the third type.} Let $a\in \epsilon A \epsilon$, $b=e_{12}\beta$ for $\beta\in e_2 A \epsilon$ and $c=e_{12}\gamma e_{21}$ for $\gamma \in e_2 A e_2$. 
\begin{equation*}
 \begin{aligned} 
    A
=&-\frac12e_{12}\dgal{a,\gamma}' \otimes \dgal{a,\gamma}'' e_{21}\otimes e_{12}\beta e_1\,,\\
B
=&\frac12 e_{12}\gamma e_{21} \otimes \dgal{\beta,a}' \otimes e_{12}\dgal{\beta,a}'' e_1 
+ \frac12 e_1 a \otimes e_{12}\dgal{\beta,\gamma}' \otimes e_{12}\dgal{\beta,\gamma}'' e_{21}   \\
&-\frac12 e_1 \otimes \dgal{\beta,a}'\otimes e_{12}\dgal{\beta,a}'' e_{12} \gamma e_{21} 
-\frac12 e_1 \otimes a e_{12} \dgal{\beta,\gamma}'  \otimes e_{12} \dgal{\beta,\gamma}'' e_{21}\,, \\
C
=& - \frac12  e_{12} \dgal{\gamma,a}''  \otimes e_{12}\beta \otimes e_1 \dgal{\gamma,a}' e_{21} \,, 
 \end{aligned}
\end{equation*}
\begin{equation*}
 \begin{aligned} 
A'
=& \frac12 e_1 \otimes a e_{12} \dgal{\beta,\gamma}' \otimes e_{12} \dgal{\beta,\gamma}'' e_{21}
- \frac12  e_1 a \otimes e_{12} \dgal{\beta,\gamma}' \otimes e_{12} \dgal{\beta,\gamma}'' e_{21} \,, \\
B'
=& \frac12 e_{12}\dgal{\gamma,a}'' \otimes e_{12}\beta \otimes e_1 \dgal{\gamma,a}' e_{21}
- \frac12 e_{12}\dgal{\gamma,a}'' \otimes \dgal{\gamma,a}'e_{21} \otimes e_{12}\beta e_1\,, \\
C'
=& \frac12    e_{12}\gamma e_{21} \otimes \dgal{a,\beta}'' \otimes  e_{12}\dgal{a,\beta}'e_1
- \frac12      e_1 \otimes \dgal{a,\beta}'' \otimes e_{12}\dgal{a,\beta}' e_{12}\gamma e_{21}  \,.
 \end{aligned}
\end{equation*}

\noindent \emph{No generator of the second type.} This case and the next one are a bit tedious. We set $a\in \epsilon A \epsilon$, $b=\beta e_{21}$ for $\beta\in \epsilon A e_2$ and $c=e_{12}\gamma e_{21}$ for $\gamma \in e_2 A e_2$. 
\begin{equation*}
 \begin{aligned} 
    A
=&\frac12 e_{12}\dgal{a,\gamma}' \otimes \dgal{a,\gamma}'' e_{21}\beta e_{21}\otimes e_1
+ \frac12  e_{12}\gamma e_{21}\dgal{a,\beta}'\otimes \dgal{a,\beta}'' e_{21} \otimes e_1  \\
&- \frac12 e_1 \dgal{a,\beta}' \otimes \dgal{a,\beta}'' e_{21} \otimes e_{12}\gamma e_{21}\,, \\
B
=&\frac12 e_{12}\gamma e_{21} \otimes \dgal{\beta,a}' e_{21} \otimes \dgal{\beta,a}'' e_1 
+ \frac12 e_1 a \otimes e_{12}\dgal{\beta,\gamma}'e_{21} \otimes \dgal{\beta,\gamma}'' e_{21}   \\
&-\frac12 e_1 \otimes \dgal{\beta,a}' e_{21}\otimes \dgal{\beta,a}'' e_{12} \gamma e_{21} 
-\frac12 e_1 \otimes a e_{12} \dgal{\beta,\gamma}' e_{21} \otimes  \dgal{\beta,\gamma}'' e_{21}\,, \\
C
=& \frac12  e_{12}\dgal{\gamma,\beta}'' e_{21}a \otimes e_1 \otimes \dgal{\gamma,\beta}'e_{21}
+ \frac12   e_{12}\dgal{\gamma,a}'' \otimes e_1 \otimes  \beta e_{21}\dgal{\gamma,a}' e_{21}  \\
&-\frac12  e_{12}\dgal{\gamma,\beta}'' e_{21} \otimes a e_1 \otimes \dgal{\gamma,\beta}' e_{21}\,, 
 \end{aligned}
\end{equation*}
\begin{equation*}
 \begin{aligned} 
A'
=& \frac12 e_{12}\dgal{\beta,\gamma}' e_{21} a \otimes e_1  \otimes \dgal{\beta,\gamma}''e_{21}
+ \frac12 e_1 \otimes a e_{12}\dgal{\beta,\gamma}' e_{21}  \otimes \dgal{\beta,\gamma}''e_{21}  \\
&-\frac12 e_{12}\dgal{\beta,\gamma}' e_{21} \otimes a e_1  \otimes \dgal{\beta,\gamma}''e_{21}
- \frac12 e_1 a \otimes e_{12}\dgal{\beta,\gamma}' e_{21}  \otimes \dgal{\beta,\gamma}''e_{21}   \,, \\ 
B'
=& \frac12 e_{12}\dgal{\gamma,a}'' \otimes \dgal{\gamma,a}'e_{21} \beta e_{21}\otimes e_1
- \frac12 e_{12}\dgal{\gamma,a}'' \otimes e_1 \otimes \beta e_{21} \dgal{\gamma,a}'e_{21}   \,, \\ 
C'
=& \frac12  e_{12}\gamma e_{21}  \otimes \dgal{a,\beta}''e_{21} \otimes \dgal{a,\beta}' e_1 
+ \frac12  e_1 \dgal{a,\beta}'   \otimes \dgal{a,\beta}''e_{21} \otimes e_{12}\gamma e_{21} \\
&-\frac12  e_1  \otimes \dgal{a,\beta}''e_{21} \otimes \dgal{a,\beta}'e_{12}\gamma e_{21}
-\frac12  e_{12}\gamma e_{21} \dgal{a,\beta}'  \otimes \dgal{a,\beta}''e_{21} \otimes e_1  \,.
 \end{aligned}
\end{equation*}

\noindent \emph{No generator of the first type.} Let $a=e_{12}\alpha$ for $\alpha \in e_2 A \epsilon$, $b=\beta e_{21}$ for $\beta\in \epsilon A e_2$ and $c=e_{12}\gamma e_{21}$ for $\gamma \in e_2 A e_2$.
\begin{equation*}
 \begin{aligned} 
    A
=&\frac12 e_{12}\dgal{\alpha,\gamma}' \otimes e_{12}\dgal{\alpha,\gamma}'' e_{21}\beta e_{21}\otimes e_1
+ \frac12  e_{12}\gamma e_{21}\dgal{\alpha,\beta}'\otimes e_{12}\dgal{\alpha,\beta}'' e_{21} \otimes e_1  \\
&- \frac12 e_1 \dgal{\alpha,\beta}' \otimes e_{12}\dgal{\alpha,\beta}'' e_{21} \otimes e_{12}\gamma e_{21}\,, \\
B
=& \frac12 e_{12}\gamma e_{21}  \otimes  e_{12}\dgal{\beta,\alpha}' e_{21} \otimes \dgal{\beta,\alpha}'' e_1
- \frac12  e_1 \otimes  e_{12}\dgal{\beta,\alpha}' e_{21} \otimes \dgal{\beta,\alpha}'' e_{12}\gamma e_{21}\\
&- \frac12 e_1 \otimes e_{12}\alpha e_{12}\dgal{\beta,\gamma}'e_{21}\otimes \dgal{\beta,\gamma}'' e_{21} \,, \\
C
=& \frac12  e_{12} \dgal{\gamma,\alpha}'' \otimes e_1 \beta e_{21}   \otimes e_{12}\dgal{\gamma,\alpha}' e_{21}
- \frac12  e_{12}\dgal{\gamma,\beta}'' e_{21}\otimes e_{12} \alpha e_1 \otimes \dgal{\gamma,\beta}' e_{21}\,,
 \end{aligned}
\end{equation*}
\begin{equation*}
 \begin{aligned} 
A'
=& \frac12   e_1 \otimes e_{12}\alpha e_{12}\dgal{\beta,\gamma}' e_{21} \otimes \dgal{\beta,\gamma}'' e_{21}
- \frac12  e_{12}\dgal{\beta,\gamma}' e_{21} \otimes e_{12}\alpha e_1 \otimes \dgal{\beta,\gamma}'' e_{21}  \,, \\
B'
=& \frac12 e_{12}\dgal{\gamma,\alpha}'' \otimes e_{12}\dgal{\gamma,\alpha}' e_{21} \beta e_{21}\otimes e_1 
- \frac12  e_{12}\dgal{\gamma,\alpha}'' \otimes e_1 \beta e_{21}\otimes e_{12}\dgal{\gamma,\alpha}'e_{21} \,,\\
C'
=& \frac12  e_{12}\gamma e_{21}  \otimes e_{12}\dgal{\alpha,\beta}''e_{21} \otimes \dgal{\alpha,\beta}' e_1 
+ \frac12  e_1 \dgal{\alpha,\beta}'   \otimes e_{12}\dgal{\alpha,\beta}''e_{21} \otimes e_{12}\gamma e_{21} \\
&-\frac12  e_1  \otimes e_{12}\dgal{\alpha,\beta}''e_{21} \otimes \dgal{\alpha,\beta}'e_{12}\gamma e_{21}
-\frac12  e_{12}\gamma e_{21} \dgal{\alpha,\beta}'  \otimes e_{12}\dgal{\alpha,\beta}''e_{21} \otimes e_1  \,.
 \end{aligned}
\end{equation*}


\section{Proof of Lemma \ref{LemMomap}} \label{App:PfMomap}

Note that $\Tr(\Phi_s)=\epsilon \Phi_s \epsilon$ for $s\neq 2$, while $\Tr(\Phi_2)=e_{12} \Phi_2 e_{21}$. In particular, using that for $s\neq 2$ we have $\Phi_s=e_s \Phi_s e_s$, we get $\Tr(\Phi_s)=\Phi_s$ by understanding that equality in $A^f$. 

\subsection{Moment map condition for the non-fused idempotents} First, assume that  $s\neq 1,2$. Then, using Lemma \ref{LemF12}, we get 
\begin{equation*}
\Tr(E_1)(\Tr(\Phi_s))=\Phi_s e_1 \otimes e_1- e_1 \otimes  e_1 \Phi_s
      =0\,, \quad 
\Tr(E_2)(\Tr(\Phi_s))=0\,,
\end{equation*}
which gives $\dgal{\Tr(\Phi_s),-}_{fus}=0$. Therefore, if $a=e_+ \alpha e_-$ is a generator of $A^f$, 
\begin{equation*}
  \dgal{\Tr(\Phi_s),a}^f = \dgal{\Tr(\Phi_s),a} =  e_+ \dgal{\Phi_s,\alpha}' \epsilon \otimes \epsilon \dgal{\Phi_s,\alpha}'' e_-\,,
\end{equation*}
where the double bracket in the last equality is taken in $A$. 
By assumption $\Phi_s$ satisfies \eqref{Phim} for  $\dgal{-,-}$ on $A$ so that 
\begin{equation} \label{EqPhisa}
  \dgal{\Tr(\Phi_s),a}^f = 
\frac12 (e_+\alpha e_s  \otimes  \Phi_se_-  
-e_+ e_s  \otimes \Phi_s \alpha e_- 
+  e_+\alpha \Phi_s  \otimes  e_s e_- 
-e_+\Phi_s  \otimes  e_s \alpha e_-) \,,
\end{equation}
where we omitted to write the idempotents $\epsilon$, because with $s\neq 1,2$ we get $e_s \epsilon=e_s =\epsilon e_s$. It remains to see that it coincides with \eqref{EqPhis} in all four cases of generators. For example, if $a=e_{12}\alpha \epsilon$ with $\alpha \in e_2 A \epsilon$, we obtain for $e_+=e_{12},e_-=\epsilon$
\begin{equation*}
  \dgal{\Tr(\Phi_s),e_{12}\alpha \epsilon}^f = \frac12 (a e_s  \otimes  \Phi_s   +  a \Phi_s  \otimes  e_s  ) \,,
\end{equation*}
because the second and last terms in \eqref{EqPhisa} disappear since $e_{12}e_s=0=e_{12}\Phi_s$. Meanwhile, the right-hand side of \eqref{EqPhis} reads in that case 
\begin{equation*}
  \frac12 (a e_s \otimes \Tr(\Phi_s) + a \Tr(\Phi_s) \otimes e_s - e_s \otimes \Tr(\Phi_s) a - \Tr(\Phi_s) \otimes e_s a )\,,
\end{equation*}
and the last two terms disappear as $s\neq1,2$. Indeed $e_s a=e_s e_{12}\alpha=0$ and $\Tr(\Phi_s)a=\epsilon(e_s \Phi_s e_s)\epsilon (e_{12}\alpha)=e_s \Phi_s e_s e_{12}\alpha=0$. The two expressions coincide, and the result is similar with the other types of generators.

\subsection{Moment map condition at the fused idempotent} 
 Using the derivation properties and decomposing the double bracket $\dgal{-,-}^f$ as $\dgal{-,-}+\dgal{-,-}_{fus}$, we obtain for $a=e_+ \alpha e_-\in A^f$, $\alpha \in A$, that 
\begin{equation} \label{EqPhiWork}
\begin{aligned}
     \dgal{\Phi_1^f,a}^f=& 
\Tr(\Phi_1)e_{12}\ast e_+ \dgal{\Phi_2,\alpha} e_- \ast e_{21} 
+ \epsilon \ast e_+ \dgal{\Phi_1,\alpha} e_- \ast \epsilon \Tr(\Phi_2) \\
&+ \Tr(\Phi_1)\ast \dgal{\Tr(\Phi_2),e_+ \alpha e_-}_{fus} + \dgal{\Tr(\Phi_1),e_+ \alpha e_-}_{fus} \ast \Tr(\Phi_2)\,.
\end{aligned}
\end{equation}
The first two terms can easily be obtained from \eqref{Phim}. 
Since $\Tr(\Phi_2)$ is a generator of fourth type \eqref{type4},  we need \eqref{wt}--\eqref{ww} to evaluate the third term. 
In the exact same way, as  $\Tr(\Phi_1)$ is a generator of  first type \eqref{type1},  we need \eqref{tt}--\eqref{tw} to evaluate the last term. 
Thus, we check separately the four types of generators.

\noindent \emph{On a generator of the first type.} We let $a\in \epsilon A \epsilon$, hence $e_+=e_-=\epsilon$ and $a=\alpha$. We directly get by  \eqref{Phim} that $\dgal{\Phi_2,a}=0$ since $e_2 a = 0 = ae_2$, while $\dgal{\Tr(\Phi_1),a}_{fus}=0$ by \eqref{tt}. For the remaining two terms, we have on one hand by \eqref{Phim} 
 \begin{equation*} 
 \dgal{\Phi_1,a}=\frac12 (ae_1\otimes \Tr(\Phi_1)-e_1 \otimes \Tr(\Phi_1) a +  a \Tr(\Phi_1) \otimes e_1-\Tr(\Phi_1) \otimes e_1 a)\,,
 \end{equation*}
after projecting the equality in $A^f$ where  $\Tr(\Phi_1)=\Phi_1$. On the other hand by \eqref{wt} 
 \begin{equation*} 
 \dgal{\Tr(\Phi_2),a}_{fus}=\frac12(a e_1 \otimes \Tr(\Phi_2) + \Tr(\Phi_2) \otimes e_1 a - a \Tr(\Phi_2) \otimes e_1 - e_1 \otimes \Tr(\Phi_2) a)\,.
 \end{equation*}
Putting this back in \eqref{EqPhiWork} yields 
\begin{equation*}
\begin{aligned}
     \dgal{\Phi_1^f,a}^f=& 
\frac12 (ae_1\Tr(\Phi_2)\otimes \Tr(\Phi_1)-e_1\Tr(\Phi_2) \otimes \Tr(\Phi_1) a +  a \Tr(\Phi_1)\Tr(\Phi_2) \otimes e_1-\Tr(\Phi_1)\Tr(\Phi_2) \otimes e_1 a)
 \\
&+\frac12(a e_1 \otimes \Tr(\Phi_1)\Tr(\Phi_2) + \Tr(\Phi_2) \otimes \Tr(\Phi_1)e_1 a - a \Tr(\Phi_2) \otimes \Tr(\Phi_1)e_1 - e_1 \otimes \Tr(\Phi_1)\Tr(\Phi_2) a)\,.
\end{aligned}
\end{equation*}
Using that $\Tr(\Phi_1)=e_1\Tr(\Phi_1)e_1$ and $\Tr(\Phi_2)=e_1\Tr(\Phi_2)e_1$ allows us to conclude after cancellation of the first and seventh terms, and the second and sixth terms.

\noindent \emph{On a generator of the second type.} Let $a=e_{12}\alpha \epsilon$ with $e_+=e_{12}$, $e_-=\epsilon$, $\alpha \in e_2 A \epsilon$. 
We get from \eqref{Phim} 
 \begin{equation*}
 \dgal{\Phi_1,\alpha}=\frac12 (\alpha e_1\otimes \Phi_1 +  \alpha \Phi_1 \otimes e_1 )\,, \quad 
\dgal{\Phi_2,\alpha}=-\frac12 (e_2 \otimes \Phi_2 \alpha +\Phi_2 \otimes e_2 \alpha)\,,
 \end{equation*}
because $e_1 \alpha=0$ and $\alpha e_2=0$. Meanwhile, \eqref{tu} and \eqref{wu} give
\begin{equation*}
 \dgal{\Tr(\Phi_1),a}_{fus}=\frac12 (e_1 \otimes \Tr(\Phi_1) a - e_1 \Tr(\Phi_1) \otimes a) \,, \quad 
\dgal{\Tr(\Phi_2),a}_{fus}=\frac12 (a e_1 \otimes \Tr(\Phi_2) - a \Tr(\Phi_2) \otimes e_1)\,.
\end{equation*}
Hence, \eqref{EqPhiWork} gives 
\begin{equation*}
\begin{aligned}
     \dgal{\Phi_1^f,a}^f=& 
-\frac12 (\, e_{12}e_{21} \otimes \Tr(\Phi_1)e_{12}\Phi_2 \alpha +e_{12}\Phi_2e_{21} \otimes \Tr(\Phi_1)e_{12}e_2 \alpha \, )  \\
&+ \frac12 (\,e_{12}\alpha e_1 \Tr(\Phi_2)\otimes \Phi_1 +  e_{12}\alpha \Phi_1 \Tr(\Phi_2) \otimes e_1 \,) \\
&+\frac12 (a e_1 \otimes \Tr(\Phi_1)\Tr(\Phi_2) - a \Tr(\Phi_2) \otimes \Tr(\Phi_1)e_1) \\
&+ \frac12 (e_1\Tr(\Phi_2) \otimes \Tr(\Phi_1) a - e_1 \Tr(\Phi_1)\Tr(\Phi_2) \otimes a)\,.
\end{aligned}
\end{equation*}
This equality holds in $A^f$ where $\Tr(\Phi_1)=\Phi_1$, $\Tr(\Phi_2)=e_{12} \Phi_2 e_{21}$ and $a=e_{12}\alpha$. Thus it is not hard to rewrite all factors in the four first terms as $\Tr(\Phi_s),a$ or the idempotents (we have to note for the first term that $e_{12}\Phi_2\alpha=e_{12}\Phi_2 e_2 \alpha = e_{12}\Phi_2 e_{21}e_{12}\alpha=\Tr(\Phi_2)a$).  After cancelling out the second (resp. third) with the seventh (resp. sixth) term, we get the desired result.

\noindent \emph{On a generator of the third type.} Let $a=\epsilon\alpha e_{21}$ with $e_+=\epsilon$, $e_-=e_{21}$,  $\alpha \in \epsilon A e_2$. 
Using  \eqref{Phim} yields
 \begin{equation*}
 \dgal{\Phi_1,\alpha}=-\frac12 (e_1 \otimes \Phi_1 \alpha +\Phi_1 \otimes e_1 \alpha)\,, \quad 
\dgal{\Phi_2,\alpha}=\frac12 (\alpha e_2\otimes \Phi_2 + \alpha \Phi_2 \otimes e_2)\,,
 \end{equation*}
because $ \alpha e_1=0$ and $e_2\alpha =0$. From  \eqref{tv} and \eqref{wv} we obtain 
\begin{equation*}
 \dgal{\Tr(\Phi_1),a}_{fus}=\frac12 (a \Tr(\Phi_1) \otimes e_1 - a \otimes \Tr(\Phi_1) e_1) \,, \quad 
\dgal{\Tr(\Phi_2),a}_{fus}=\frac12 (\Tr(\Phi_2) \otimes e_1 a - e_1 \otimes \Tr(\Phi_2) a  )\,.
\end{equation*}
Summing everything inside \eqref{EqPhiWork}, we get 
\begin{equation*} 
\begin{aligned}
     \dgal{\Phi_1^f,a}^f=& 
\frac12 (\alpha e_{21} \otimes \Tr(\Phi_1)e_{12} \Phi_2e_{21}  + \alpha \Phi_2 e_{21}  \otimes \Tr(\Phi_1)e_{12} e_2e_{21} )\\
&-\frac12 (e_1 \Tr(\Phi_2) \otimes \Phi_1 \alpha e_{21} +\Phi_1 \Tr(\Phi_2) \otimes e_1 \alpha e_{21}) \\
&+ \frac12 (\Tr(\Phi_2) \otimes \Tr(\Phi_1)e_1 a - e_1 \otimes \Tr(\Phi_1)\Tr(\Phi_2) a  ) \\
&+ \frac12 (a \Tr(\Phi_1)\Tr(\Phi_2) \otimes e_1 - a\Tr(\Phi_2) \otimes \Tr(\Phi_1) e_1)\,.
\end{aligned}
\end{equation*}
By arguments similar to the previous case, we can rewrite the four first terms using $a,\Tr(\Phi_1),\Tr(\Phi_2)$ and the idempotents $e_1,e_2$ so that the second and eighth terms cancel out, while the third and fifth terms cancel out. The remaining terms give the desired result.

\noindent \emph{On a generator of the fourth type.} We let $a=e_{12}\alpha e_{21}$ with $e_+=e_{12}$, $e_-=e_{21}$, $\alpha\in e_2 A e_2$. We directly get by  \eqref{Phim} that $\dgal{\Phi_1,a}=0$, and by \eqref{ww} that  $\dgal{\Tr(\Phi_2),a}_{fus}=0$. For the remaining two terms, we have by \eqref{Phim} and \eqref{tw}
 \begin{equation*} 
\begin{aligned}
 \dgal{\Phi_2,\alpha}=&\frac12 (\alpha e_2\otimes \Phi_2-e_2 \otimes \Phi_2 \alpha +  \alpha \Phi_2 \otimes e_2-\Phi_2 \otimes e_2 \alpha)\,, \\
 \dgal{\Tr(\Phi_1),a}_{fus}=&\frac12(a \Tr(\Phi_1) \otimes e_1 + e_1 \otimes \Tr(\Phi_1) a - a \otimes \Tr(\Phi_1) e_1 - e_1 \Tr(\Phi_1) \otimes a )\,.
\end{aligned}
 \end{equation*}
Thus, we get after some easy manipulations
\begin{equation*} 
\begin{aligned}
     \dgal{\Phi_1^f,a}^f=& 
\frac12 (a\otimes \Tr(\Phi_1)\Tr(\Phi_2)-e_1 \otimes \Tr(\Phi_1)\Tr(\Phi_2)a +  a\Tr(\Phi_2) \otimes \Tr(\Phi_1)-\Tr(\Phi_2) \otimes \Tr(\Phi_1)a) \\
&+\frac12(a \Tr(\Phi_1)\Tr(\Phi_2) \otimes e_1 + e_1  \Tr(\Phi_2)\otimes \Tr(\Phi_1) a - a \Tr(\Phi_2) \otimes \Tr(\Phi_1) e_1 - e_1 \Tr(\Phi_2) \Tr(\Phi_1) \otimes a )\,,
\end{aligned}
\end{equation*}
from which we can conclude.


\section{Proof of Proposition \ref{Pr:Q1}} \label{App:PfQ1}

Note that any $B$-linear double bracket on $A$ of degree at most $+4$ on generators needs to satisfy 
\begin{subequations}
  \begin{align}
    \dgal{t,t}=&\lambda (tst \otimes t - t \otimes tst)\,, \quad 
   \dgal{s,s}=l (sts \otimes s - s \otimes sts)\,, \label{EqQ1A} \\
    \dgal{t,s}=&\gamma e_2 \otimes e_1 + \alpha_1' st \otimes e_1 + \alpha_3 e_2 \otimes ts 
+ \phi_0 stst \otimes e_1 + \phi_1 st \otimes ts + \phi_2 e_2 \otimes tsts\,, \label{EqQ1B}
  \end{align}
\end{subequations}
after using that $t=e_1 t e_2,s=e_2 s e_1$ with the cyclic antisymmetry and the derivation rules.  Moreover, if $\dgal{-,-}$ is a double quasi-Poisson bracket it must satisfy \eqref{qPabc} on generators, and this is easily seen to be equivalent to require that 
\begin{subequations}
  \begin{align}
   & \dgal{t,t,t}=0\,, \quad \dgal{s,s,s}=0\,, \label{EqQ1aaa} \\
& \dgal{t,t,s}=\frac14 (st \otimes t \otimes e_1 - e_2 \otimes t \otimes ts)\,, \label{EqQ1tts} \\
& \dgal{s,s,t}=\frac14 (ts \otimes s \otimes e_2 - e_1 \otimes s \otimes st)\,. \label{EqQ1sst} 
  \end{align}
\end{subequations}

\begin{lem} \label{Lem:Q1A}
If \eqref{EqQ1aaa} holds, then either $\lambda=l=0$ or 
\begin{equation}  \label{EqCond1}
  \gamma=0,\quad \phi_1=0, \quad \alpha_1'=-\alpha_3, \quad \phi_0=-\phi_2\,.
\end{equation}
\end{lem}
\begin{proof}
  By \eqref{Eq:TripBr}, we have that for any $a \in A$, 
\begin{equation}\label{Eq:Triaaa}
  \dgal{a,a,a}=(1+\tau_{(123)}+\tau_{(132)})\dgal{a,\dgal{a,a}'}\otimes \dgal{a,a}''\,.
\end{equation}
We first look at the case $a=t$. Using \eqref{EqQ1A}, we can find that 
\begin{equation*}
  \dgal{t,\dgal{t,t}'}\otimes \dgal{t,t}''=
\lambda^2 \, tstst \otimes t \otimes t - \lambda^2\, t \otimes tstst \otimes t 
- \lambda^2\, tst \otimes t \otimes tst + \lambda^2\, t \otimes tst \otimes tst 
+ t \dgal{t,s}t\otimes t\,.
\end{equation*}
The first four terms cancel if we take their sum under cyclic permutations, so that we can write 
\begin{equation*}
\begin{aligned}
  \dgal{t,t,t}=&\lambda(1+\tau_{(123)}+\tau_{(132)}) t\dgal{t,s}t\otimes t \\
=&\lambda(1+\tau_{(123)}+\tau_{(132)})\left[\gamma t\otimes t \otimes t + (\alpha_1'+\alpha_3)t \otimes t \otimes tst + (\phi_0+\phi_2) t \otimes t \otimes tstst + \phi_1 t \otimes tst \otimes tst \right].
\end{aligned}
\end{equation*}
Therefore either $\lambda=0$, or the different coefficients vanish i.e. $\gamma=0$, $\phi_1=0$ while $\alpha_1'=-\alpha_3$ and $\phi_0=-\phi_2$. Doing the computation with $s$ instead of $t$, we need either $l=0$ or the same four conditions. 
\end{proof}

\begin{lem} \label{Lem:Q1B}
If $\lambda=0$ and \eqref{EqQ1tts} holds, then 
\begin{subequations}
\begin{align}
  &\phi_0=0, \quad \phi_2=0, \label{LQ1a}\\
&(\alpha_1')^2=\frac14 + \phi_1 \gamma, \quad \alpha_3^2=\frac14 + \phi_1 \gamma. \quad 
(\alpha_1'-\alpha_3)\gamma=0, \quad (\alpha_1'-\alpha_3)\phi_1=0\,. \label{LQ1b}
\end{align}
\end{subequations}
The same identities are satisfied if $l=0$ and \eqref{EqQ1sst} holds. 
\end{lem}
\begin{proof}
  When we compute $\dgal{t,t,s}$ using \eqref{Eq:TripBr}, we get that the term $(st)^3 \otimes t \otimes e_1$ only appears with a factor $\phi_0^2$, and $e_2 \otimes t \otimes (ts)^3$ only appears with a factor $-\phi_2^2$. Therefore, if \eqref{EqQ1tts} is satisfied we need $\phi_0=\phi_2=0$ which gives \eqref{LQ1a}. 

Under the conditions from \eqref{LQ1a}, the only terms remaining in $\dgal{t,t,s}$ are given by 
$st \otimes t \otimes e_1$, $e_2 \otimes t \otimes ts$, $e_2 \otimes t \otimes e_1$ and $st \otimes t \otimes ts$ with respective coefficients $(\alpha_1')^2-\phi_1 \gamma$, $-((\alpha_3)^2-\phi_1 \gamma)$, $(\alpha_1'-\alpha_3)\gamma$ and $(\alpha_1'-\alpha_3)\phi_1$. Comparing with \eqref{EqQ1tts}, we get \eqref{LQ1b}. 

The method is exactly the same in the case $l=0$ assuming that \eqref{EqQ1sst} holds. 
\end{proof}

We get by combining Lemmas \ref{Lem:Q1A} and \ref{Lem:Q1B} that if $\lambda=l=0$ as well as $\alpha_1'\neq \alpha_3$, we are in the case 1.a) of Proposition \ref{Pr:Q1}.  If $\alpha_1'=\alpha_3$ instead, we are in the case 1.b).

\medskip

We now assume that at least one of the two constants $\lambda,l$ is nonzero. Hence, if the double bracket \eqref{EqQ1A}--\eqref{EqQ1B} satisfies \eqref{EqQ1aaa}, it must be such that  
\begin{equation}  \label{EqQ1Bbis}
  \dgal{t,s}=\alpha_3 (e_2 \otimes ts - st \otimes e_1) + \phi_0 (stst \otimes e_1 - e_2 \otimes tsts)\,, \quad 
\alpha_3,\phi_0\in \kk\,,
\end{equation}
using Lemma \ref{Lem:Q1A}. 

\begin{lem} If \eqref{EqQ1tts} holds, then $\phi_0=0$, $l\lambda=0$ and $\alpha_3^2=\frac14$.  Moreover, the same statement holds if \eqref{EqQ1sst} holds. 
\end{lem}
\begin{proof}
Developing $\dgal{t,t,s}$ with \eqref{Eq:TripBr}, we get that the term $e_2\otimes tstst \otimes ts$ only appears with a factor $\phi_0^2$. (This is also true for $e_2 \otimes t \otimes tststs$, $st \otimes tstst \otimes e_1$ and $ststst \otimes t \otimes e_1$ with factor $-\phi_0^2$.) Therefore $\phi_0=0$. Under this condition, we obtain that 
\begin{equation*}
  \dgal{t,t,s}=\alpha_3^2(st \otimes t \otimes e_1 - e_2 \otimes t \otimes ts) + \lambda l (st \otimes t \otimes tsts - stst \otimes t \otimes ts)\,,
\end{equation*}
and we get the remaining two equalities by comparing this expression with \eqref{EqQ1tts}. The computation for $\dgal{s,s,t}$ with \eqref{EqQ1sst} is similar and gives the second result. 
\end{proof}

As a consequence of this lemma, $\phi_0$ vanishes and  $\alpha_3=\pm\frac12$ in \eqref{EqQ1Bbis}. Furthermore, either we have $\lambda\neq0$ with $l=0$, or we have $l \neq 0$ with $\lambda=0$.  These are respectively Case 2 and Case 3 from Proposition \ref{Pr:Q1}.


\section{Proof of Proposition \ref{Pr:Free2}} \label{App:F2}

\subsection{Coefficients verifying the triple brackets identities}

The strategy of the proof is given after Proposition \ref{Pr:Free2}. In this subsection, we gather a list of equalities that the coefficients appearing in the double bracket must satisfy in order for the corresponding triple bracket to satisfy \eqref{Eqtts14} or \eqref{Eqsst14}. 

\subsubsection{First conditions}

\begin{lem}
If a double bracket given by \eqref{Eqtt}--\eqref{Eqss} and \eqref{EqtsA} satisfies \eqref{Eqtts14}, then we have $\beta_0=\beta_0'=\alpha_1=\alpha_3'=0$. 
\end{lem}
\begin{proof}
  Without computing all terms, we can remark that in $\dgal{t,t,s}$ (obtained from \eqref{Eq:TripBr} using \eqref{Eqtt}, \eqref{EqtsA}) the element $s^3 \otimes 1 \otimes 1$ appears with coefficient $\beta_0^2$, and so do respectively $1 \otimes 1 \otimes s^3$, $t^2s \otimes 1 \otimes 1$ and $1 \otimes 1 \otimes st^2$ with coefficients $-(\beta_0')^2$, $\alpha_1^2$ and $-(\alpha_3')^2$. None of these expressions appear \eqref{Eqtts14}. 
\end{proof}

We can go through a similar argument using $\dgal{s,s,t}$ instead. 
\begin{lem}
If a double bracket given by \eqref{Eqtt}--\eqref{Eqss} and \eqref{EqtsA} satisfies \eqref{Eqsst14}, then we have $\alpha_0=\alpha_0'=\alpha_1=\alpha_3'=0$. 
\end{lem}

Hence, we are left to discuss the coefficients of the double bracket given by \eqref{Eqtt}--\eqref{Eqss} and 
\begin{equation}
  \begin{aligned}
    \dgal{t,s}=&\,\,\, \gamma_0 \, t \otimes t + \gamma_1 \, s \otimes s  + \alpha_1'\, st \otimes 1 + \alpha_2 \, t \otimes s + \alpha_2' \, s \otimes t + \alpha_3 \, 1 \otimes ts  \\
&+ \beta_1 \, t \otimes 1 + \beta_1' \, 1 \otimes t + \beta_2 \, s \otimes 1 + \beta_2' \, 1 \otimes s + \gamma \, 1 \otimes 1\,. \label{Eqts}
  \end{aligned}
\end{equation}

\subsubsection{Identities verified by the coefficients when \eqref{Eqtts14} is satisfied}

\begin{lem} \label{Lem:nu0}
Consider a double bracket defined on $A$ by \eqref{Eqtt}, \eqref{Eqss} and \eqref{Eqts}, with $\nu=0$, $\lambda \in \kk$ and $\mu \in \{\pm \frac12\}$. Then \eqref{Eqtts14} is satisfied if and only if the following list of identities hold : 
\begin{subequations}
  \begin{align}
    &\alpha_1',\alpha_3 =\pm \frac12\,, \quad \alpha^2_2=\frac14+\gamma_1 \gamma_0\,, \quad 
(\alpha'_2)^2=\frac14+\gamma_1 \gamma_0\,, \label{nu0a}\\
&\frac14 + \alpha_2 \alpha_3=-\mu (\alpha_2 + \alpha_3)\,, \quad 
\frac14 + \alpha_1' \alpha_2' = \mu (\alpha_1'+\alpha_2')\,, \label{nu0abis}\\
&\gamma_1 (\alpha_1'- \mu)=0\,, \quad \gamma_1 (\alpha_2'-\alpha_2)=0\,, \quad \gamma_1 (\alpha_3 + \mu)=0\,, \label{nu0b}\\ 
& \beta_2(\alpha_1'-\mu)=0\,, \quad \beta_2 (\alpha_2'-\mu)- \gamma_1 \beta_1'=0\,, \quad (\beta_2-\lambda)(\alpha_1'+\alpha_2')-\gamma_1 \beta_1=0\,, \label{nu0c}\\
& \beta_2'(\alpha_3+\mu)=0\,, \quad \beta_2' (\alpha_2+\mu)- \gamma_1 \beta_1=0\,, \quad (\beta_2'+\lambda)(\alpha_2+\alpha_3)-\gamma_1 \beta_1'=0\,, \label{nu0d}\\
&\gamma_0 (\alpha_1'-\mu)=0\,,\quad \gamma_0(\alpha_2'-\alpha_2)=0\,, \quad  
\gamma_0(\alpha_3 + \mu)=0\,,\quad \label{nu0e} \\
& \beta_1 (\alpha_1' - \alpha_2)+ \gamma_0 (\beta_2-\lambda)=0\,, \quad \beta_1'(\alpha_2'-\alpha_3)-\gamma_0(\beta_2'+\lambda)=0\,, \label{nu0f}\\
&\beta_1' (\alpha_1' - \mu)- \beta_1 (\alpha_3+\mu)=0\,, \quad 
\beta_1 (\alpha_2' - \mu)- \beta_1' (\alpha_2+\mu)+\gamma_0 \lambda=0\,, \label{nu0g} \\
&\gamma (\alpha_2 + \mu) - \beta_2 \beta_1=0\,, \quad \gamma (\alpha_1' - \alpha_3)+ \beta_1'(\beta_2-\lambda)-\beta_1(\beta_2'+\lambda)=0\,, \quad 
\gamma (\alpha_2' - \mu ) - \beta_2' \beta_1'=0\,, \label{nu0h} \\
& \beta_2' (\beta_2'+\lambda)-\gamma_1 \gamma=0\,, \quad \beta_2 (\beta_2-\lambda)- \gamma_1 \gamma=0\,, \label{nu0i} \\
&(\beta_2-\beta_2'-\lambda)\gamma_1=0\,, \quad (\beta_2 - \beta_2'-\lambda)\gamma=0\,. \label{nu0j}
  \end{align}
\end{subequations}
\end{lem}
\begin{proof}
We collect now all nonzero terms that appear in the expansion of $\dgal{t,t,s}$ obtained from \eqref{Eq:TripBr}, leaving the cumbersome (but elementary) computations to the reader. 

The coefficients for $t \otimes t \otimes s$, $s \otimes t \otimes t$, $1 \otimes t \otimes ts$ and $st \otimes t \otimes 1$ are respectively $\gamma_0\gamma_1-\alpha^2_2$, $(\alpha_2')^2-\gamma_0\gamma_1$, $-\alpha_3^2$ and $(\alpha_1')^2$. The coefficient for $t \otimes 1 \otimes ts$ and $1 \otimes t^2 \otimes s$ is $- \alpha_2 \alpha_3- \mu (\alpha_2+\alpha_3)$, while we have for $st \otimes 1 \otimes t$ and $s \otimes t^2 \otimes 1$ the coefficient $\alpha_1' \alpha_2' - \mu (\alpha_1' + \alpha_2')$. Since these terms appear in \eqref{Eqtts14}, this gives \eqref{nu0a} and \eqref{nu0abis}. In particular, all the other coefficients in the expansion of $\dgal{t,t,s}$ must vanish. 

The coefficients for $st \otimes 1 \otimes s$, $s \otimes t \otimes s$ and $s \otimes 1 \otimes ts$ are respectively 
$\gamma_1 (\alpha_1'- \mu)$, $\gamma_1 (\alpha_2'-\alpha_2)$ and $\gamma_1 (\alpha_3 + \mu)$, which yields \eqref{nu0b}. 

The vanishing of the coefficients for $st \otimes 1 \otimes 1$, $s \otimes 1 \otimes t$ and $s \otimes t \otimes 1$ gives successively the three identities in \eqref{nu0c}. Similarly $1 \otimes 1 \otimes ts$, $t \otimes 1 \otimes s$ and $1 \otimes t \otimes s$ imply \eqref{nu0d}, while $1 \otimes t^2 \otimes t$, $t \otimes t^2 \otimes 1$ and $t \otimes t \otimes t$ give \eqref{nu0e}. 

The coefficients for $t \otimes t \otimes 1,1 \otimes t \otimes t$  and $1\otimes t^2 \otimes1, t \otimes 1 \otimes t$ give \eqref{nu0f} and \eqref{nu0g} respectively. With $t \otimes 1 \otimes 1$, $1 \otimes t \otimes 1$ and $1 \otimes 1 \otimes t$ we obtain \eqref{nu0h}. 

The terms $1 \otimes 1 \otimes s$ and $s \otimes 1 \otimes 1$ give \eqref{nu0i}. We finally get \eqref{nu0j} from $s \otimes 1 \otimes s$ and $1 \otimes 1 \otimes 1$.  
\end{proof}

In the exact same way, we get the next lemma. 
\begin{lem} \label{Lem:nuneq}
Consider a double bracket defined on $A$ by \eqref{Eqtt}, \eqref{Eqss} and \eqref{Eqts}, with $\nu\neq 0$ and $\lambda, \mu \in \kk$ satisfying $4(\mu^2-\lambda \nu) = 1$. Then \eqref{Eqtts14} is satisfied if and only if the following list of identities holds : 
\begin{subequations}
  \begin{align}
    &\gamma_0=0\,, \quad \gamma_1=0\,, \quad \beta_2=0\,, \quad \beta_2'=0\,, \label{nuneq0}\\
    &\alpha_1',\alpha_2,\alpha_2',\alpha_3 =\pm \frac12\,, \quad \frac14 + \alpha_2 \alpha_3=-\mu (\alpha_2 + \alpha_3)\,, \quad \frac14 + \alpha_1' \alpha_2' = \mu (\alpha_1'+\alpha_2')\,, \label{nuneqb}\\
& \alpha_1'=-\alpha_2'\,, \qquad \alpha_2=-\alpha_3\,, \qquad   \beta_1'=-\beta_1\,, \label{nuneqc} \\
& \beta_1(\alpha_2+\alpha_2')=0\,, \quad \beta_1(\alpha_1'+\alpha_3)-\nu \gamma=0\,, \quad  
\beta_1(\alpha_2'-\alpha_3)+\nu \gamma=0\,, \quad \beta_1(\alpha_1'-\alpha_2)-\nu \gamma=0\,, \label{nuneqd} \\ 
&\gamma(\alpha_2+\mu)=0\,, \quad \gamma(\alpha_1'-\alpha_3)=0\,, \quad 
\gamma (\alpha_2'-\mu)=0\,, \quad \lambda \gamma =0\,. \label{nuneqe}
  \end{align}
\end{subequations}
\end{lem}

\begin{rem}
These results are easily adapted to the case where the double bracket is Poisson, i.e. when the associated triple bracket \eqref{Eq:TripBr} identically vanishes. In such a case, we require $4(\mu^2-\lambda \nu) = 0$ to get $\dgal{t,t,t}=0$ by \cite[Proposition A.1]{P16}. 

If $\nu=\mu=0$, then $\dgal{t,t,s}=0$ when the conditions \eqref{nu0a}--\eqref{nu0j} of Lemma \ref{Lem:nu0} are satisfied with the extra requirements that all the terms containing a factor $\mu$ are removed, and that 
 all the terms $\pm\frac12$ and $\frac14$ in \eqref{nu0a}--\eqref{nu0abis} are removed (in particular $\alpha_1'=\alpha_3=0$).

If $\nu \neq0$ and  $\mu^2-\lambda \nu = 0$ then $\dgal{t,t,s}=0$ when the conditions \eqref{nuneq0}--\eqref{nuneqe} of Lemma \ref{Lem:nu0} are satisfied with the extra requirements that the terms $\pm \frac12$ and $\frac14$ appearing in the identities \eqref{nuneqb} are removed. 
\end{rem}

\subsubsection{Identities verified by the coefficients when \eqref{Eqsst14} is satisfied} \label{subsst}

We can obtain the analogues of Lemmae \ref{Lem:nu0} and \ref{Lem:nuneq} when \eqref{Eqsst14} is satisfied as follows. Using the cyclic antisymmetry of the double bracket, remark that we can get from \eqref{Eqts} 
\begin{equation}
  \begin{aligned}
    \dgal{s,t}=&\,\,\, - \gamma_1 \, s \otimes s - \gamma_0 \, t \otimes t  
-\alpha_3 \,  ts \otimes 1 - \alpha_2 \, s \otimes t - \alpha_2' \, t \otimes s - \alpha_1'\, 1 \otimes st  \\
&- \beta_2' \, s \otimes 1 - \beta_2 \, 1 \otimes s - \beta_1' \, t \otimes 1 - \beta_1 \, 1 \otimes t 
- \gamma \, 1 \otimes 1 \,. \label{Eqstbis}
  \end{aligned}
\end{equation}
Comparing \eqref{Eqtt} and \eqref{Eqss}, then doing the same with \eqref{Eqts} and \eqref{Eqstbis}, one can see that to compute $\dgal{s,s,t}$ one just needs to consider  $\dgal{t,t,s}$ in which we replace all variables $s$ by $t$ and vice-versa, then do the following changes in the coefficients 
\begin{equation}
  \begin{aligned} \label{Eq:Mapping}
  &  \lambda \mapsto l\,, \quad \mu \mapsto m\,, \quad \nu \mapsto n\,, \\
&\gamma_0 \mapsto -\gamma_1\,, \quad \gamma_1 \mapsto - \gamma_0\,, \quad 
\alpha_1' \mapsto - \alpha_3\,, \quad  \alpha_2 \mapsto - \alpha_2\,, \quad \alpha_2' \mapsto - \alpha_2'\,, \quad 
\alpha_3 \mapsto - \alpha_1'\,, \\ 
&\beta_1 \mapsto - \beta_2'\,, \quad \beta_1' \mapsto - \beta_2\,, \quad 
\beta_2 \mapsto - \beta_1'\,, \quad \beta_2' \mapsto - \beta_1\,, \quad \gamma \mapsto - \gamma\,.
  \end{aligned}
\end{equation}

For $n=0$, $l \in \kk$ and $m \in \{\pm \frac12\}$, we have that \eqref{Eqsst14} is satisfied if and only if the list of identities obtained by applying \eqref{Eq:Mapping} to \eqref{nu0a}--\eqref{nu0j} is verified. 

For $n\neq 0$ and $l,m \in \kk$ satisfying $4(m^2-ln) = 1$, we have that \eqref{Eqsst14} is satisfied if and only if the list of identities obtained by applying \eqref{Eq:Mapping} to \eqref{nuneq0}--\eqref{nuneqe} is verified.

\subsection{Splitting the identities into cases}

\begin{lem} \label{Lem:tts}
Consider a reduced double bracket defined on $A$ by \eqref{Eqtt}, \eqref{Eqss} and \eqref{Eqts}, with $\nu=\lambda=0$ and $\mu \in \{\pm \frac12\}$. Then \eqref{Eqtts14} is satisfied if and only if the double bracket verifies one of the following cases 

\underline{\emph{Case A:}} For $\gamma_0,\gamma_1\in \kk^\times$, then $\gamma\in \kk$ is free while 
\begin{equation}  
  \begin{aligned} 
& \alpha_1'=\mu, \quad \alpha_3=-\mu,\quad \alpha_2'=\alpha_2 \,\, \text{ with }\, \alpha_2^2=\frac14 + \gamma_0 \gamma_1\,, \\
& \beta_1=\frac{\gamma_0 \beta_2}{\alpha_2-\mu}, \quad \beta_1'=\frac{\gamma_0 \beta_2}{\alpha_2+\mu}, \quad 
\beta_2'=\beta_2\,\, \text{ with }\beta_2^2=\gamma \gamma_1\,. 
\label{ttsCasA} 
  \end{aligned}  
\end{equation}

\underline{\emph{Case B:}} For $\gamma_1\in \kk^\times$, $\gamma_0=0$, then $\beta_2 \in \kk$ is free while 
\begin{equation}
\alpha_1'=\mu, \quad \alpha_3=-\mu, \quad \beta_2'=\beta_2, \quad \gamma=\frac{\beta_2^2}{\gamma_1} \,, \label{ttsCasBgen} 
\end{equation}
and one of the following two sets of conditions holds : 
\begin{subequations}
  \begin{align}
\text{\emph{B1)}}&\quad \alpha_2'=\alpha_2=\mu, \quad \beta_1'=0, \quad \beta_1=\frac{2\mu \beta_2}{\gamma_1}\,, 
\label{ttsCasB1} \\     
\text{\emph{B2)}}&\quad \alpha_2'=\alpha_2=-\mu, \quad \beta_1=0, \quad \beta_1'=-\frac{2\mu \beta_2}{\gamma_1}\,. 
\label{ttsCasB2} 
  \end{align}
\end{subequations}

\underline{\emph{Case C:}} For $\gamma_0\in \kk^\times$, $\gamma_1=0$, then 
\begin{equation}
\alpha_1'=\mu, \quad \alpha_3=-\mu, \quad \beta_2'=\beta_2=0, \quad \gamma=0 \,, \label{ttsCasCgen} 
\end{equation}
and one of the following two sets of conditions holds : 
\begin{subequations}
  \begin{align}
\text{\emph{C1)}}&\quad \alpha_2'=\alpha_2=\mu, \quad \beta_1'=0, \quad \beta_1\in \kk\,, 
\label{ttsCasC1} \\     
\text{\emph{C2)}}&\quad \alpha_2'=\alpha_2=-\mu, \quad \beta_1=0, \quad \beta_1\in \kk\,. 
\label{ttsCasC2} 
  \end{align}
\end{subequations}

\underline{\emph{Case D:}} For $\gamma_0=\gamma_1=0$, then $\beta_2'=\beta_2=0$ and one of the following sets of conditions holds : 

\noindent if $(\alpha_1',\alpha_3)=(-\mu,\mu)$, 
\begin{equation}
 \text{\emph{D1)}}\quad \alpha_1'=\alpha_2=-\mu, \quad \alpha_3=\alpha_2'=\mu, \quad \beta_1=\beta_1'=\gamma=0\,;
\label{ttsCasD1} 
\end{equation}
if $(\alpha_1',\alpha_3)=(\mu,\mu)$,
\begin{subequations}
  \begin{align}    
\text{\emph{D2.1)}}&\quad \alpha_1'=\alpha_2'=\alpha_3=\mu, \quad \alpha_2=-\mu, \quad \beta_1=0, \quad \beta_1',\gamma\in \kk\,,
\label{ttsCasD21} \\
\text{\emph{D2.2)}}&\quad \alpha_1'=\alpha_3=\mu, \quad \alpha_2=\alpha_2'=-\mu, \quad \beta_1=\beta_1'=\gamma=0\,;
\label{ttsCasD22} 
  \end{align}
\end{subequations}
if $(\alpha_1',\alpha_3)=(-\mu,-\mu)$,
\begin{subequations}
  \begin{align}    
\text{\emph{D3.1)}}&\quad \alpha_1'=\alpha_2=\alpha_3=-\mu, \quad \alpha_2'=\mu, \quad \beta_1'=0, \quad \beta_1,\gamma\in \kk\,,
\label{ttsCasD31} \\
\text{\emph{D3.2)}}&\quad \alpha_1'=\alpha_3=-\mu, \quad \alpha_2=\alpha_2'=\mu, \quad \beta_1=\beta_1'=\gamma=0\,;
\label{ttsCasD32} 
  \end{align}
\end{subequations}
if $(\alpha_1',\alpha_3)=(\mu,-\mu)$,
\begin{subequations}
  \begin{align}    
\text{\emph{D4.1)}}&\quad \alpha_1'=\alpha_2=\alpha_2'=\mu, \quad \alpha_3=-\mu, \quad \beta_1'=\gamma=0, \quad \beta_1\in \kk\,,
\label{ttsCasD41} \\
\text{\emph{D4.2)}}&\quad \alpha_1'=\mu, \quad \alpha_3=\alpha_2=\alpha_2'=-\mu, \quad \beta_1=\gamma=0, \quad \beta_1'\in \kk\,,
\label{ttsCasD42} \\
\text{\emph{D4.3)}}&\quad \alpha_1'=\alpha_2=\mu, \quad \alpha_3=\alpha_2'=-\mu, \quad \gamma=0, \quad \beta_1'=-\beta_1,\,\, \beta_1\in \kk\,,
\label{ttsCasD43} \\
\text{\emph{D4.4)}}&\quad \alpha_1'=\alpha_2'=\mu, \quad \alpha_3=\alpha_2=-\mu, \quad \beta_1=\beta_1'=\gamma=0\,.
\label{ttsCasD44} 
  \end{align}
\end{subequations}
\end{lem}

The proof of Lemma \ref{Lem:tts} consists in listing the possible coefficients of a reduced double bracket that satisfy Lemma \ref{Lem:nu0}. The next lemma is obtained similarly from Lemma \ref{Lem:nuneq}. 

\begin{lem} \label{Lem:ttsNU}
Consider a reduced double bracket defined on $A$ by \eqref{Eqtt}, \eqref{Eqss} and \eqref{Eqts}, with $\nu\in \kk^\times$ and $\mu=0$, $\lambda=\frac{-1}{4\nu}$. Then \eqref{Eqtts14} is satisfied if and only if the double bracket verifies 
\begin{equation}
\gamma_0=\gamma_1=0, \quad \beta_2=\beta_2'=0, \quad \gamma=0, \quad \alpha_2=\pm \frac12\,, \label{ttsCasNU} 
\end{equation}
and  one of the following two conditions holds : 
\begin{subequations}
  \begin{align}    
\text{\emph{A}$_\nu$\emph{)}}&\quad \alpha_1'=\alpha_2, \quad \alpha_2'=\alpha_3=-\alpha_2, \quad \beta_1'=-\beta_1,\,\, \beta_1\in \kk\,,
\label{ttsCasANU} \\
\text{\emph{B}$_\nu$\emph{)}}&\quad \alpha_2'=\alpha_2, \quad \alpha_1'=\alpha_3=-\alpha_2, \quad \beta_1=\beta_1'=0\,.
\label{ttsCasBNU} 
  \end{align}
\end{subequations}
\end{lem}

\begin{rem} \label{RemMapp}
From the discussion in \ref{subsst}, we get that a reduced double bracket defined on $A$ by \eqref{Eqtt}, \eqref{Eqss} and \eqref{Eqts} satisfies  \eqref{Eqsst14} if and only if the double bracket verifies one of the cases from Lemma \ref{Lem:tts} or Lemma \ref{Lem:ttsNU} \emph{after} application of the mapping \eqref{Eq:Mapping} on the different coefficients in each case.  
\end{rem}

\subsection{Finishing the proof}

We need to see which conditions from Lemma \ref{Lem:tts} or Lemma \ref{Lem:ttsNU} are compatible with at least one of the conditions obtained by applying the mapping \eqref{Eq:Mapping}, as explained in Remark \ref{RemMapp}. 

For example, applying transformation \eqref{Eq:Mapping} to the case D4.4 in Lemma \ref{Lem:tts} yields 
\begin{equation}
\begin{aligned}
 \text{Case D4.4$^{(s)}$)}\quad& \gamma_0=\gamma_1=0, \quad \beta_1=\beta_1'=0\,, \\
 &\alpha_3=\alpha_2'=-m, \quad \alpha_1'=\alpha_2=m, \quad \beta_2=\beta_2'=\gamma=0\,.
\label{sstCasD44}
\end{aligned}
\end{equation}
A quick inspection shows that this is compatible with the conditions of the cases D1, D4.3 given by \eqref{ttsCasD1}, \eqref{ttsCasD43} in Lemma \ref{Lem:tts}, and A$_{\nu}$ given by \eqref{ttsCasANU} in Lemma \ref{Lem:ttsNU}. 
In the first two cases, and under the isomorphism $t\mapsto s$, $s\mapsto t$ (with $\mu \leftrightarrow m$), the obtained double quasi-Poisson brackets satisfy Case 3 of Proposition \ref{Pr:Free2} given by \eqref{MyCas4}.  
In the last case, the double bracket is isomorphic to Case 6 of Proposition \ref{Pr:Free2} given by \eqref{MyCas16} under the same isomorphism (with $m\mapsto \mu$, $\nu \mapsto n$).



\begin{thebibliography}{1}

\bibitem[AKKN1]{AKKN18} Alekseev, A., Kawazumi, N., Kuno, Y., Naef, F.: {\it The Goldman-Turaev Lie bialgebra in genus zero and the Kashiwara-Vergne problem}. Adv. Math. 326, 1--53 (2018); 
\href{https://arxiv.org/abs/1703.05813}{arXiv:1703.05813}. 


\bibitem[AKKN2]{AKKN} Alekseev, A., Kawazumi, N., Kuno, Y., Naef, F.: {\it The Goldman-Turaev Lie bialgebras
and the Kashiwara-Vergne problem in higher genera}. Preprint;
\href{https://arxiv.org/abs/1804.09566}{arXiv:1804.09566}. 

  
\bibitem[AKSM]{QuasiP} Alekseev, A., Kosmann-Schwarzbach, Y., Meinrenken, E.: {\it Quasi-{P}oisson manifolds}. Canad. J. Math. 54(1), 3--29 (2002);
\href{https://arxiv.org/abs/math/0006168}{arXiv:math/0006168}.


\bibitem[B]{B}  Bielawski, R.: {\it Quivers and Poisson structures.} Manuscripta Math. 141, no. 1-2, 29--49 (2013).
\href{https://arxiv.org/abs/1108.3222}{arXiv:1108.3222}.
  
\bibitem[CB1]{CB99} Crawley-Boevey, W.: {\it Preprojective algebras, differential operators and a {C}onze embedding for deformations of {K}leinian singularities}. Comment. Math. Helv. 74(4), 548--574 (1999).

\bibitem[CB2]{CB11} Crawley-Boevey, W.: {\it Poisson structures on moduli spaces of representations}. J. Algebra 325, 205--215  (2011).

\bibitem[CBS]{CBShaw}  Crawley-Boevey, W., Shaw, P.: {\it Multiplicative preprojective algebras, middle convolution and the Deligne-Simpson problem}. Adv. Math. 201 (1), 180--208 (2006); 
\href{https://arxiv.org/abs/math/0404186}{arXiv:math/0404186}.

\bibitem[CF]{CF}  Chalykh, O., Fairon, M.: {\it Multiplicative quiver varieties and generalised Ruijsenaars-Schneider models}. J. Geom. Phys. 121, 413--437 (2017);  \href{https://arxiv.org/abs/1704.05814}{arXiv:1704.05814}.

\bibitem[CQ]{CQ95}  Cuntz, J., Quillen, D.: {\it Algebra extensions and nonsingularity}. J. Amer. Math. Soc. 8(2), 251--289 (1995).  


\bibitem[IK]{IK} Iyudu, N., Kontsevich, M.: {\it Pre-Calabi-Yau algebras as noncommutative Poisson structures};
\href{http://preprints.ihes.fr/2018/M/M-18-04.pdf}{IHES/M/18/04} (2018).

\bibitem[K]{K} Kontsevich, M.: {\it Formal (non)-commutative symplectic geometry}. In: I.M. Gelfand, L. Corwin, J. Lepowsky (Eds.), The Gelfand Mathematical Seminars, 1990--1992, Birkhauser, Boston, pp. 173--187 (1993).

\bibitem[KR]{KR} Kontsevich, M., Rosenberg, A.L.: {\it Noncommutative smooth spaces}. The Gelfand Mathematical Seminars, 1996–1999, Gelfand Math. Sem., Birkh\"auser Boston, Boston,
MA, pp. 85--108 (2000).

\bibitem[LBP]{LBP}  Le Bruyn, L., Procesi, C.: {\it Semisimple representations of quivers}. Trans. Amer. Math. Soc. 317, no. 2, 585--598 (1990).


  \bibitem[MT]{MT14} Massuyeau, G., Turaev, V.: {\it Quasi-Poisson structures on representation spaces of surfaces}. Int. Math. Res. Not. IMRN, no. 1, 1--64 (2014);
\href{https://arxiv.org/abs/1205.4898}{arXiv:1205.4898}.

\bibitem[ORS1]{ORS} Odesskii, A.V., Rubtsov, V.N., Sokolov, V.V.: {\it Double Poisson brackets on free associative algebras}. Noncommutative birational geometry, representations and combinatorics, Contemp. Math., vol. 592, Amer. Math. Soc., Providence, RI, 2013, pp. 225--239; 
\href{https://arxiv.org/abs/1208.2935}{arXiv:1208.2935}.

\bibitem[ORS2]{ORS2} Odesskii, A.V., Rubtsov, V.N., Sokolov, V.V.: {\it Parameter-dependent associative Yang-Baxter equations and Poisson brackets}. Int. J. Geom. Methods Mod. Phys. 11, no. 9, 1460036, 18 pages (2014); 
\href{https://arxiv.org/abs/1311.4321}{arXiv:1311.4321}.

\bibitem[PVdW]{PV} Pichereau, A., Van de Weyer, G.: {\it Double Poisson cohomology of path algebras of quivers}. J. Algebra 319, no. 5, 2166--2208 (2008); 
\href{https://arxiv.org/abs/math/0701837}{arXiv:math/0701837}.

 \bibitem[P]{P16}  Powell, G.: {\it On double Poisson structures on commutative algebras}. J. Geom. Phys. 110, 1--8  (2016);  \href{https://arxiv.org/abs/1603.07553}{arXiv:1603.07553}.

\bibitem[S]{S13} Sokolov, V.V.: {\it Classification of constant solutions of the associative Yang-Baxter equation
on Mat$_3$}, Theoret. and Math. Phys. 176, no. 3, 1156--1162 (2013), Russian version appears
in Teoret. Mat. Fiz. 176, no. 3, 385--392 (2013);
\href{https://arxiv.org/abs/1212.6421}{arXiv:1212.6421}.


\bibitem[VdB1]{VdB1} Van den Bergh, M.: {\it Double Poisson algebras}. Trans. Amer. Math. Soc., 360, no. 11, 5711--5769 (2008);
\href{https://arxiv.org/abs/math/0410528}{arXiv:math/0410528}.

\bibitem[VdB2]{VdB2} Van den Bergh, M.: {\it Non-commutative quasi-{H}amiltonian spaces}. In: Poisson geometry in mathematics and physics, volume 450 of Contemp. Math., 273--299. Amer. Math. Soc., Providence, RI (2008);
\href{https://arxiv.org/abs/math/0703293}{arXiv:math/0703293}.

\bibitem[VdW]{VdW} Van de Weyer, G.: {\it Double Poisson structures on finite dimensional semi-simple algebras}. 
Algebr. Represent. Theory 11, no. 5, 437--460 (2008).


  \end{thebibliography}
\end{document}